\documentclass{amsart}
\usepackage{ifpdf}
\usepackage{amsmath, amsfonts, amssymb, amscd}
\usepackage[active]{srcltx}
\usepackage[pdftex]{graphicx}

\newcommand{\supp}{\operatorname{Supp}}
\newcommand{\fil}{\operatorname{Fill}}
\newcommand{\sing}{\operatorname{Sing}}

\newcommand{\abs}[1]{\left\lvert{#1}\right\rvert}
\newcommand{\norm}[1]{\left\|{#1}\right\|}

\DeclareMathOperator{\pr}{\rm{pr}}

\DeclareMathOperator{\SL}{\rm{SL}}
\DeclareMathOperator{\inter}{\rm{int}}
\DeclareMathOperator{\bd}{\partial}
\DeclareMathOperator{\cl}{cl}
\DeclareMathOperator{\diam}{\rm{diam}}
\DeclareMathOperator{\conv}{\rm{Conv}}

\DeclareMathOperator{\fix}{\rm{Fix}}

\newcommand{\mc}{\mathcal}

\newcommand{\ol}{\overline}
\renewcommand{\hat}{\widehat}
\newcommand{\til}{\widetilde}

\newcommand{\R}{\mathbb{R}}\newcommand{\N}{\mathbb{N}}
\newcommand{\Z}{\mathbb{Z}}
\newcommand{\T}{\mathbb{T}}
\newcommand{\D}{\mathbb{D}}
\renewcommand{\SS}{\mathbb{S}}

\newcommand{\sm}{\setminus}

\newcommand{\id}{\mathrm{Id}}

\newcommand{\ie}{i.e.\ }
\newcommand{\eg}{e.g.\ }

\newtheorem{theorem}{Theorem}[section]%[theorem] %[section]
\newtheorem{corollary}[theorem]{Corollary}
\newtheorem{lemma}[theorem]{Lemma}
\newtheorem{proposition}[theorem]{Proposition}
\newtheorem{fact}{Fact}[theorem]
\newtheorem{claim}{Claim}
\newtheorem{remark}[theorem]{Remark}

\newtheorem{definition}[theorem]{Definition}

\newtheorem{theoremain}{Theorem}

\newtheorem{propositionmain}[theoremain]{Proposition}

\title[Bounded and unbounded behavior]{Bounded and unbounded behavior for area-preserving rational pseudo-rotations}
\author{Andres Koropecki}
\address{Andres Koropecki. Universidade Federal Fluminense, Instituto de Matem\'atica e Estat\'\i stica, Rua M\'ario Santos Braga S/N, 24020-140 Niteroi, RJ, Brasil}
\email{ak@id.uff.br}
\author{Fabio Armando Tal}
\address{Fabio Armando Tal. Instituto de Matem\'atica e Estat\'\i stica, Universidade de S\~ao Paulo, Rua do Mat\~ao 1010, Cidade Universit\'aria, 05508-090 S\~ao Paulo, SP, Brazil}
\email{fabiotal@ime.usp.br}
\thanks{The first author was partially supported by CNPq-Brasil. The second author was partially supported by FAPESP and CNPq-Brasil}

\begin{document}

\begin{abstract} 
A rational pseudo-rotation $f$ of the torus is a homeomorphism homotopic to the identity with a rotation set consisting of a single vector $v$ of rational coordinates.  We give a classification for rational pseudo-rotations with an invariant measure of full support, in terms of the deviations from the constant rotation $x\mapsto x+v$ in the universal covering. For the simpler case that $v=(0,0)$, it states that either every orbit by the lifted dynamics is bounded, or the displacement of orbits in the universal covering is uniformly bounded in some rational direction (implying that the dynamics is annular) or the set of fixed points of $f$ contains a large continuum which is the complement of a disjoint union of disks (i.e. a fully essential continuum).
In the analytic setting, the latter case is ruled out. In order to prove this classification, we introduce tools that are of independent interest and can be applied in a more general setting: in particular, a geometric result about the quasi-convexity and existence of asymptotic directions for certain chains of disks, and a Poincar\'e recurrence theorem on the universal covering for irrotational measures. 
\end{abstract}

\maketitle

\setcounter{tocdepth}{1}
\tableofcontents

\section{Introduction}

If $f\colon \R/\Z =\T^1 \to \T^1$ is a homeomorphism preserving orientation and $\hat{f}\colon \R\to \R$ is a lift of $f$, there is a corresponding rotation number $\rho(\hat{f}) = \lim_{n\to \infty} (\hat{f}^n(x)-x)/n$, which was defined by Poincar\'e and shown to be is independent of $x\in \R$. The rotation number is a useful invariant for the dynamics: if $\rho(\hat{f})$ is irrational, then $f$ is monotonically semi-conjugate to a rigid irrational rotation, and if $\rho(\hat{f})$ is rational, then $f$ has a periodic point and there is a simple model for the dynamics.

One may try to generalize the notion of rotation number to dimension two, considering a homeomorphism $f\colon \T^2\to \T^2$ homotopic to the identity and a lift $\hat{f}\colon \R^2\to \R^2$. However, in this setting the limit $(\hat{f}^n(x)-x)/n$ often fails to exist, and when it does it depends on the chosen point $x\in \R^2$. This is why one usually defines a \emph{rotation set} $\rho(\hat{f})\subset \R^2$ instead of a rotation number or vector. The rotation set was defined by Misiurewicz and Ziemian \cite{m-z} as the set of all vectors $v\in \R^2$ of the form $$v=\lim_{k\to \infty} \frac{\hat{f}^{n_k}(z_{k})-z_k}{n_k},\quad z_{k}\in \R^2, \quad n_k\to \infty$$

In the special case that the rotation set $\rho(\hat{f})$ contains a unique vector $v$, the map $f$ is called a \emph{pseudo-rotation}. In this case, it is easy to see that $(\hat{f}^n(z)-z)/n\to v$ for any $z\in \R^2$, so one may expect such maps to have more similarities with one-dimensional case.

A pseudo-rotation is called irrational if the corresponding rigid rotation $x\mapsto x+v$ induces a minimal map on $\T^2$ (which is the same as saying that the coordinates of $v$, together with $1$, are rationally independent), and the pseudo-rotation is rational if both coordinates of $v$ are rational (which means that the corresponding rotation $x\mapsto x+v$ is periodic).

Irrational pseudo-rotations have been studied in many works \cite{jager-linearization,jager-bmm,kwapisz-combinatorics,denjoy-rees}, and it is known that they are  not necessarily semi-conjugate to rigid rotations. Moreover, they may exhibit dynamical properties that differ greatly from rigid rotations, like weak-mixing \cite{fayad-mixing,kk-mixing} or positive entropy \cite{rees,denjoy-rees}.

A key property that holds for circle homeorphisms is the property of uniformly bounded deviations, which means that orbits of $\hat{f}$ remain a bounded distance away from orbits of the rigid rotation. In other words, the quantity $\abs{\smash{\hat{f}^n(x)-x-n\alpha}}$ is bounded by a constant independent of $x$ and $n$.
This property often fails to hold, even pointwise, for real analytic area-preserving irrational pseudo-rotations \cite{kk-mixing}. However, J\"ager proved in \cite{jager-linearization} that if one assumes that such maps satisfy the bounded deviations property, then a Poincar\'e-like theorem holds: $f$ is semi-conjugate to the rigid rotation.

We will consider the case of area-preserving rational pseudo-rotations. 
In contrast with the irrational case, one cannot expect to obtain any local or semi-local information from the assumption that the rotation vector is a unique rational point. Indeed, any dynamics that can appear in the closed unit disk can be embedded in the torus, extending it to be the identity outside a neighborhood of the disk, thus obtaining a rational pseudo-rotation. However, one may try to obtain some information about the deviations of the orbits with respect to the rigid rotation.

For any rational pseudo-rotation $f$ there is always a power $f^n$ which has a lift to $\R^2$ with rotation vector $(0,0)$. Thus, after taking an appropriate power we are left with the problem of understanding a homeomorphism $f$ which has a lift $\hat{f}$ such that $\rho(\hat{f})=\{(0,0)\}$. Such an $f$ is called an \emph{irrotational} homeomorphism, and $\hat{f}$ is its irrotational lift. The problem of studying the deviations with respect to the rigid rotation is then reduced to studying the boundedness of the displacement $\hat{f}^n(x)-x$.

In \cite{kt-example} the authors constructed an example of a $C^\infty$ area-preserving irrotational and ergodic homeomorphism such that almost every point in the universal covering has an unbounded orbit \emph{in all directions}. Further, almost every orbit visits every fundamental domain in $\R^2$. The example has the particularity that all the nontrivial dynamics is restricted to an open topological disk $U\subset \T^2$, while the complement of $U$ (which is a large continuum) consists of fixed points. 

The main result of this article implies that this is the only way for an irrotational area-preserving homeomorphism to have orbits which are unbounded in more than one direction in the universal covering; that is, the set of fixed points has to contain a fully essential continuum (\ie the complement of a disjoint union of open topological disks in $\T^2$). 
Moreover, the latter must always be the case unless there is uniformly bounded displacement in some rational direction of $\R^2$:

\begin{theoremain}\label{th:teoremao} Let $f\colon \T^2\to \T^2$ be an irrotational homeomorphism preserving a Borel probability measure $\mu$ of full support, and let $\hat f$ be its irrotational lift. Then one of the following holds:
\begin{itemize}
\item[(i)] $\fix(f)$ is fully essential;
\item[(ii)] Every point of $\R^2$ has a bounded $\hat{f}$-orbit;
\item[(iii)] $\hat{f}$ has uniformly bounded displacement in a rational direction; \ie there is a nonzero $v\in \Z^2$ and $M>0$ such that $$\abs{\smash{\big\langle \hat{f}^n(z)-z; v\big\rangle}}\leq M$$ for all $z\in \R^2$ and $n\in \Z$.%  $$\sup\left\{\abs{\smash{p_v(\hat{f}^n(z)-z)}}: z\in \R^2,\, n\in \Z\right\} = M <\infty.$$
\end{itemize}
\end{theoremain}

Whenever $f$ has a lift $\hat{f}$ such that case (iii) above holds, $f$ is said to be  \emph{annular}. This is because the dynamics of $f$ is essentially that of a homeomorphism of the annulus, after passing to a finite covering (see \cite{KT2012}).
Of course, one obtains a statement for arbitrary rational pseudo-rotations, after replacing $f$ with $f^n$: 

\begin{theoremain}\label{th:teoremao-rational} Let $f\colon \T^2\to \T^2$ be a rational pseudo-rotation preserving a Borel probability measure $\mu$ of full support, and $\hat{f}$ a lift of $f$. Then one of the following holds:
\begin{itemize}
\item[(i)] $\fix(f^k)$ is fully essential for some $k\in \N$;
\item[(ii)] Every orbit of $\hat{f}$ has bounded deviation from the rigid rotation $x\mapsto x+\alpha$, where $\alpha$ is the rotation vector of $\hat{f}$. That is,
$$\sup_{n\in \Z} \norm{\smash{\hat{f}^n(z)-z - n\alpha}}<\infty \text{ for all }z\in \R^2.$$
\item[(iii)] $f$ has uniformly bounded deviations from the rigid rotation in some rational direction; \ie there is a nonzero $v\in \Z^2$ and $M>0$ such that $$\abs{\smash{\big\langle \hat{f}^n(z)-z-n\alpha; v\big\rangle}}\leq M$$ for all $z\in \R^2$ and $n\in \Z$. Equivalently, $f^k$ is annular for some $k\in \N$.
\end{itemize}
\end{theoremain}

Let us point out that the only cases where neither (ii) nor (iii) hold in Theorem \ref{th:teoremao} (and Theorem \ref{th:teoremao-rational} accordingly) are rather pathological. To illustrate this, we have the following 
\begin{propositionmain}\label{pro:non-lc} Under the hypotheses of Theorem \ref{th:teoremao}, suppose that only case (i) holds. Then  the essential connected component of $\fix(f)$ is not locally connected.
\end{propositionmain}

\subsection{The analytic case}
If $f$ is real analytic, then its set of fixed points is an analytic set, hence it is locally connected. In particular, the case from the previous proposition is excluded:
\begin{theoremain}\label{th:teoremao-analytic} Let $f\colon \T^2\to \T^2$ be an irrotational real analytic diffeomorphism preserving a Borel probability measure $\mu$ of full support, and let $\hat f$ be its irrotational lift. Then, either every orbit of $\hat{f}$ is bounded, or every orbit of $\hat{f}$ is uniformly bounded in some rational direction (\ie $f$ is annular).
\end{theoremain}

The above theorem applied to some power of $f$ shows that in the real analytic setting, area-preserving rational pseudo-rotations necessarily have bounded deviations from the rigid rotation, at least in some direction. An analogous property for irrational pseudo-rotations does not hold, in view of the analytic examples from \cite{kk-mixing}. 
Note also that the example from \cite{kt-example} shows that Theorem \ref{th:teoremao-analytic} is false if one replaces `real analytic' by `$C^\infty$'.

\subsection{The area-preserving hypothesis} It is important to note that the hypothesis of the existence of an invariant probability measure of full support in Theorem \ref{th:teoremao} is essential, and it cannot be relaxed to a nonwandering condition. Let us briefly sketch an example, which which was communicated to us by Bassam Fayad. If $X$ is the constant vector field $X(x)=v$ on $\T^2$, where $v\in \R^2$ is some vector of irrational slope, and if $\phi\colon \T^2\to \R$ is a $C^\infty$ function such that $\phi(x_0)=0$ for some $x_0\in \T^2$ and $\phi(x)>0$ otherwise, then the time-one map $f$ of the flow induced by the vector field $\phi X$ on $\T^2$ has $x_0$ as its unique fixed point, and all other orbits are dense. Moreover, since $\phi$ is $C^\infty$ near $x_0$, a direct computation shows that $f$ is irrotational (moreover, one may show that the unique invariant probability measure is the Dirac measure at $x_0$). This example is clearly nonwandering, and none of the cases from Theorem \ref{th:teoremao} hold ($f$ does have bounded displacement, but in an irrational direction). 

\subsection{Poincar\'e recurrence on the lift} In order to prove Theorem \ref{th:teoremao}, we prove a Poincar\'e recurrence type theorem in the lifted dynamics under certain conditions, which can be applied in a more general setting than irrotational homeomorphisms. If $v\in \R^2$ is a nonzero vector, we denote by $H^+_v$ the half-plane $\{u\in\R^2 : \langle u; v\rangle \ge 0 \}$. Recall that the rotation set is always compact and convex \cite{m-z}. One can also define the rotation vector $\rho_\mu(\hat{f})$ associated to an invariant probability measure $\mu$; see Section \ref{sec:poincare} for the definition.

\begin{theoremain}\label{th:naoerrante-ext} 
Let $f\colon \T^2\to \T^2$ be a homeomorphism homotopic to the identity and $\hat{f}$ a lift of $f$ to $\R^2$. Suppose that $(0,0)$ is an extremal point of $\rho(\hat{f})$, and $\rho(\hat{f})\subset H^+_v$ for some $v\in \Z^2$, $v\neq (0,0)$. Then, for any $f$-invariant Borel probability measure $\mu$ such that $\rho_\mu(\hat{f})=(0,0)$, the set of $\hat{f}$-recurrent points projects to a set of full $\mu$-measure in $\T^2$.
\end{theoremain}

As an immediate consequence, we have
\begin{theoremain}\label{th:irrotational-nw} Suppose $f\colon \T^2\to \T^2$ is an irrotational homeomorphism preserving a Borel probability measure with full support $\mu$. Then its irrotational lift $\hat{f}$ is nonwandering.
\end{theoremain} 

\subsection{Geometric results for eventually free chains}
The other theorem that deserves to be highlighted is Theorem \ref{th:geometric}, which is a geometric result about decreasing chains of arcwise connected sets. We omit the statement from the introduction, since it is somewhat technical; we refer the reader to Section \ref{sec:geometric}. We only mention that Theorem \ref{th:geometric} is rather general (considerably more than what is needed in our proof of Theorem \ref{th:teoremao}), and it does not involve any dynamics.

\subsection{Outline of the article}
This article is organized as follows. Section \ref{sec:prelim} introduces some useful notation and some existing results that will be used in many places along this article. Section \ref{sec:geometric} is devoted to the geometric results mentioned in the previous paragraph (in particular, Theorem \ref{th:geometric}). Before proving these results, in \S\ref{sec:chain-annular} we prove as an application a technical result that plays a crucial role in the proof of Theorem \ref{th:teoremao}.

The goal of Section \ref{sec:poincare} is proving Theorem \ref{th:naoerrante-ext}, but it includes some technical results and notions that are also useful in other parts of the article. We begin defining the rotation vector of an invariant measure and recalling its main properties.  In \S\ref{sec:atkinson} we present a ``directional recurrence'' result, which is a consequence of a theorem of Atkinson \cite{atkinson}. Following \cite{Transitivo,aneltransitivo}, \S\ref{sec:omegas} introduces the sets $\omega_v$, which play a fundamental role in the proof of Theorem \ref{th:naoerrante-ext}, presented in the remaining subsections.

Section \ref{sec:brouwer} introduces some results from \cite{KT2012} which rely heavily on the equivariant Brouwer theorem of Le Calvez \cite{lecalvez-equivariant} and a recent result of Jaulent \cite{jaulent}. The main goal of the section is to show that the results about invariant and periodic topological disks that are proved in \cite{KT2012} for maps with a ``gradient-like Brouwer foliation'' remain valid in the context of an irrotational area-preserving homeomorphism. Of particular interest is Proposition \ref{pro:engulfing-finite}, which is the key for ruling out bounded displacement in an irrational direction in Theorem \ref{th:teoremao}.

Finally, Section \ref{sec:teoremao} presents the proof of Theorem \ref{th:teoremao}, and Section \ref{sec:non-lc} proves Proposition \ref{pro:non-lc}.

\section{Notation and preliminaries}\label{sec:prelim}
%We consider $\T^2 = \R^2/\Z^2$ with covering projection $\pi\colon \R^2\to \Z^2$.
We denote by $\N$ the set of \emph{positive} integers. The sets $\R_*$, $\R^2_*$, $\Z_*$ and $\Z^2_*$ denote the set of all non-zero elements of the corresponding spaces, \eg $\Z^2_* =\{v\in \Z^2:v\neq (0,0)\}$ and similarly for the other spaces. 

By $\langle x;y\rangle$ we denote the canonical inner product of two vectors in $\R^2.$  Given $v\in\R^2_*$, we denote by $p_v:\R^2\to\R$ the orthogonal projection
$$p_v(x)={\big\langle x;\frac{v}{\norm{v}}\big\rangle}.$$
For any $v=(a,b)\in\R^2$, we denote by $v^{\perp}$ the orthogonal vector $v^\perp=(-b,a)$, and the translation $x\mapsto x+v$ of $\R^2$ is denoted by $T_v$. 
If $S\subset \R^2$ is a set, we will use both $S+v$ and $T_v(S)$ to denote the translated set $\{x+v : x\in S\}$.

 If $\gamma\colon[0,1]\to X$ is an arc, then $[\gamma]$ denotes its image and $-\gamma$ denotes the reversed arc $(-\gamma)(t) = \gamma(1-t)$. If $\gamma'\colon [0,1]\to X$ is another arc with $\gamma'(0) = \gamma(1)$, then $\gamma*\gamma'\colon [0,1]\to X$ denotes their concatenation.

\subsection{The boundary at infinity}
Given a set $X\subset \R^2$, we say that $X$ accumulates in the direction $v\in \SS^1$ at infinity if there is a sequence $\{x_n\}_{n\geq 0}$ in $X$ such that 
$$\lim_{n\to \infty}\norm{x_n} = \infty \quad \text{ and }\quad \lim_{n\to\infty} \frac{x_n}{\norm{x_n}}=v.$$
The boundary of $X$ at infinity is defined as the set $\bd_\infty X\subset \SS^1$ consisting of all $v\in \SS^1$ such that $X$ accumulates in the direction $v$ at infinity. 

Denoting by $\mathbb{S}^1_\infty$ a disjoint copy of $\mathbb{S}^1$, the space $\R^2_\infty = \R^2\sqcup \mathbb{S}^1_{\infty}$ can be topologized in a way that it is homeomorphic to the closed unit disk $\ol{\D}$ and $\bd_{\R^2_\infty} X = \bd_\infty X \cup \bd_{\R^2} X$ for any $X\subset \R^2$. A basis of open sets in $\R^2_\infty$ is given by the open subsets of $\R^2$ together with sets of the form $V\cup I_\infty$ where $I\subset \mathbb{S}^1$ is an open interval, $I_\infty$ is the corresponding interval in $\mathbb{S}_\infty$, and $V=\{tv : v\in I,\, t\geq M\}$.

\subsection{Essential and inessential sets}\label{sec:essential}
An open subset $U$ of a surface $S$ is said to be \emph{inessential} if every loop in $U$ is homotopically trivial in $S$; otherwise, $U$ is \emph{essential}. An arbitrary set $E\subset S$ is called inessential if it has some inessential open neighborhood. We say that $E$ is \emph{fully essential} if $S\sm E$ is inessential.

%\begin{definition}[Annular set]
%An open connected set $O\subset\T^2$ is called annular if $O$ is essential but not fully essential. An arbitrary set $E$ is called annular if it is essential, and there is a neighborhood of $E$ which is annular.
%\end{definition}

The next proposition is contained in \cite[Proposition 1.3]{KT2012}.
\begin{proposition}\label{pro:compact-ine} If $K\subset \T^2$ is compact and inessential, then any connected component of $\pi^{-1}(K)$ is bounded. Thus, if $U$ is open and fully essential then any connected component of $\R^2\sm \pi^{-1}(U)$ is bounded.
\end{proposition}

We also need the following proposition, included in \cite[Proposition 1.4]{KT2012}.
\begin{proposition}\label{pro:annular-annular} Let $f\colon \T^2\to \T^2$ be a homeomorphism homotopic to the identity.
\begin{itemize}
\item[(1)] If there is an $f$-invariant connected open or closed set which is neither inessential nor fully essential, then $f$ is annular.
\item[(2)] If $f$ is non-annular and has a fixed point, then $f^n$ is non-annular for all $n\in \N$.
%\item[(3)] If $f$ is non-annular, then any connected $f$-invariant open set which is essential is also fully essential.
\end{itemize}
\end{proposition}

%
%\begin{proposition}\label{pr:anularlimitadonahorizontal}
%If $O\subset\T^2$ is open connected and annular, then there exists $w\in\Z^2_*$ such that, if $\hat O$ is a connected component from $\pi^{-1}(O)$, then $\hat O=(\hat O+w),$ but if $w^{\perp}\in\Z^2_*$ is perpendicular to $w,$ then $\hat O\cap(\hat O+w^{\perp})=\emptyset.$  Furthermore, there exists $M>0$ such that
%$\sup_{\hat x\in\hat O}\abs{p_{w^{\perp}}(x)}<M.$ 
%\end{proposition}

%\begin{proposition}
%If $O\subset\T^2$ is open connected and fully essential, then $\hat O=\pi^{-1}(O)$ is a connected set, and every connected component of $\hat O^{C}$ is bounded.
%\end{proposition}

\subsection{The filling of a set}
If $E\subset \R^2$ is connected, we define its \emph{filling} $\fil(E)$ as the union of $E$ with all the bounded connected components of $\R^2\sm E$. Thus $\fil(E)$ is connected and all the connected components of its complement are unbounded.  In particular if $E$ is open, then $\fil(E)$ is an open topological disk. If $g\colon \R^2\to \R^2$ is a homeomorphism then $\fil(g(E))=g(\fil(E))$, so the filling of an invariant set is invariant.

\begin{proposition}\label{pro:fill-ine} $E\cap T_v(E)=\emptyset$ if and only if $\fil(E)\cap T_v(\fil(E))=\emptyset$ ($v\in \R^2_*$)
\end{proposition}
\begin{proof}
The `if' direction is trivial. To prove the `only if' part, suppose $\fil(E)\cap T_v(\fil(E))\neq \emptyset$, and assume first that $T_v(\fil(E))$ is disjoint from $E$. Then $T_v(\fil(E))$ must intersect (and thus be contained in) a bounded connected component of $\R^2\sm E$. This implies that $T_v(\fil(E))$ is bounded (so $\fil(E)$ is also bounded), and $T_v(\fil(E))\subset \fil(E)$. But the latter implies that $T_v^n(\fil(E))\subset \fil(E)$ for all $n\in \N$, contradicting the fact that $\fil(E)$ is bounded. 

Now assume that $T_v(\fil(E))$ intersects $E$. Note that $T_v(\fil(E))=\fil(T_v(E))$. If $T_v(E)$ intersects $E$, we are done. Otherwise, $E$ intersects $\fil(T_v(E))\sm T_v(E)$, which means that $E$ intersects (and is contained in) a bounded connected component of $\R^2\sm T_v(E)$. Thus $E$ is bounded, and so $\fil(E)$ is bounded. Moreover, $E\subset \fil(T_v(E))$, from which $\fil((E)\subset \fil(T_v(E)) = T_v(\fil(E))$. This means that $T_v^{-n}(\fil(E))\subset \fil(E)$, contradicting the fact that $\fil(E)$ is bounded.
\end{proof}

If $E\subset \T^2$ is open or closed, then we define $\fil(E)$ as the union of $E$ with all the inessential connected components of $\T^2\sm E$ (see \cite[\S1.4]{KT2012}). One easily verifies that the filling of invariant sets is invariant.

In the case that $U$ is open and connected, $\fil(U)$ coincides with $\pi(\fil(\hat{U}))$, where $\hat{U}$ is any connected component of $\pi^{-1}(U)$. Moreover,
\begin{itemize}
\item $U$ is inessential if and only if $\fil(U)$ is an open topological disk;
\item $U$ is essential but not fully essential if and only if $\fil(U)$ is a topological annulus;
\item $U$ is fully essential if and only if $\fil(U)=\T^2$.
\end{itemize}

\subsection{The sets $U_\epsilon(z)$} \label{sec:uepsilon}

Let $f\colon S\to S$ be a homeomorphism of an orientable surface $S$. Given $z\in S$ and $\epsilon>0$, denote by $U_\epsilon'(z,f)$ (or simply $U_\epsilon'(z)$ when there is no ambiguity) the connected component of $\bigcup_{n\in \Z} f^n(B_\epsilon(z))$ containing $z$. 
Suppose that $f^n(B_\epsilon(z))$ intersects $B_\epsilon(z)$ for some $n\in \N$ (otherwise, $U_\epsilon'(z)=B_\epsilon(z)$). Since $f$ permutes the connected components of $\bigcup_{n\in \Z} f^n(B_\epsilon(z))$, it follows that $U_\epsilon'(z) = f^n(U_\epsilon'(z))$, and if $n\in \N$ is chosen minimal with that property, then $f^k(U_\epsilon'(z))$ is disjoint from $U_\epsilon'(z))$ whenever $1\leq k < n$. In particular, if $n>1$ then $U_\epsilon'(z)$ is disjoint from its image.

%Using the terminology from \cite{KT2012}, we say that a point $z\in S$ is \emph{inessential} if $U_\epsilon'(x)$ is an inessential set, and otherwise we say that $z$ is essential. 

If $S=\R^2$, we let $U_\epsilon(z) = U_\epsilon(z,f)$ denote the set $\fil(U_\epsilon'(z))$, \ie  the union of $U_\epsilon'(z)$ with all the bounded connected components of its complement. It follows that $U_\epsilon(z)$ is an $f^n$-invariant open topological disk. Moreover, $U_\epsilon(z)$ intersects $f(U_\epsilon(z))$ if and only if $U_\epsilon'(z)$ intersects $f(U_\epsilon'(z))$ (the proof of this fact is similar to the proof of Proposition \ref{pro:fill-ine}). This implies that $U_\epsilon(z)$ is invariant if and only if $U_\epsilon'(z)$ is invariant, and otherwise it is periodic and disjoint from its image.

%To see this, suppose that $n>1$, so that $f(U_\epsilon'(z))$ is disjoint from $U_\epsilon'(z)$, and assume for contradiction that $U_\epsilon(z)$ intersects $f(U_\epsilon(z))$. Then there is a bounded connected component $B$ of $\R^2\sm U_\epsilon'(z)$ such that either $f(U_\epsilon(z))\subset B$ or $U_\epsilon(z)\subset f(B)$. Assume that the former case holds (the latter one is analogous using $f^{-1}$ instead of $f$). Then 
%$$f(U_\epsilon(z)) = \fil(f(U_\epsilon'(z))\subset B\subsetneq U_\epsilon(z),$$
%implying that $f^k(U_\epsilon(z))\subsetneq U_\epsilon(z)$ for all $k\in \N$. In particular, $U_\epsilon(z) = f^n(U_\epsilon(z))\subsetneq U_\epsilon(z)$, a contradiction.

Therefore, if $B_\epsilon(z)$ is not wandering, then the set $U_\epsilon(z)$ is an open topological disk which is either invariant, or periodic and free for $f$. Note also that $U_\epsilon(x)\subset U_{\epsilon'}(x)$ if $\epsilon<\epsilon'$.

\subsection{Strictly toral dynamics}\label{sec:strictly}

 %we say that a nonwandering homeomorphism $f\colon \T^2\to \T^2$ homotopic to the identity is \emph{strictly toral} if $\fix(f)$ is inessential and $f^n$ is non-annular for each $n\in \N$. 

The following result, which is contained in Theorem B from \cite{KT2012}, is critical in this article. 
\begin{theorem}\label{th:essine}
Suppose $f\colon \T^2\to \T^2$ is a nonwandering non-annular homeomorphism homotopic to the identity such that $\fix(f)$ is not fully essential. Then any invariant open topological disk $U\subset \T^2$ is such that the connected components of $\pi^{-1}(U)$ are bounded.
\end{theorem}

%Also recall from \cite{KT2012} that $x\in \T^2$ is an \emph{essential point} of $f$ if $U_\epsilon'(x)$ is essential for any $\epsilon>0$. The set of all essential points is a compact invariant set denoted by $\ess(f)$.

%
%\begin{remark} Note that if $E$ is compact and connected, $E$ is inessential if and only if each connected component of $\pi^{-1}(E)$ is bounded (in the usual sense, \ie relatively compact in $\R^2$).
%\end{remark}

\section{Geometric Results}\label{sec:geometric}

In this section we will prove two general technical lemmas that play a key role in the proof of Theorem \ref{th:teoremao}. The arguments are all geometric in nature and there is no dynamics involved (however, we use arguments from Brouwer theory in the proofs). The theorems will be proved in a setting which is considerably more general than what we need, as we expect that they may be useful in future works. In order to simplify the statements, we introduce some terminology.

A set $S\subset \R^2$ is $r$-quasiconvex for some $r>0$ if $S$ intersects every open ball of radius $r$ contained in the convex hull of $S$. 
We also say that $S$ is $r$-dense if it intersects every open ball of radius $r$ in $\R^2$.

\begin{remark} Our definition of $r$-quasiconvex set differs slightly from the one usually found in the literature, which requires that the convex hull of $S$ be contained in the $r$-neighborhood of $S$. However, a connected set that is $r$-quasiconvex with our definition is always $2r$-quasiconvex with the usual definition (but we do not need this fact).
\end{remark}

Let $\Sigma\subset \R^2$ be a closed subset, which we usually assume to be discrete. We say that a set $U\subset \R^2$ is \emph{$\Sigma$-free} if $T_v(U)\cap U=\emptyset$ for all $v\in \Sigma$. An \emph{chain} (of arcwise connected sets) is a sequence $\mc{C} = (U_n)_{n\in \N}$ of arcwise connected sets such that $U_{n+1}\subset U_n$ for all $n\in \N$ (\ie a decreasing sequence). We say that
\begin{itemize}
\item $\mc{C}$ is \emph{eventually $\Sigma$-free}, if for each $v\in \Sigma$ there is $n\in \N$ such that $T_v(U_n)\cap U_n=\emptyset$ (hence the same property holds for larger $n$);
\item $\mc{C}$ is \emph{eventually $r$-quasiconvex} for some $r>0$ if $\bigcap_{n\in \N} \ol{U}_n$ is $r$-quasiconvex;
\item $\mc{C}$ is \emph{eventually quasiconvex} if it is eventually $r$-quasiconvex for some $r>0$;
\item $\mc{C}$ \emph{has an asymptotic direction} if there is $v\in \mathbb{S}^1$ such that $\bigcap_{n\in \N} \bd_\infty U_n = \{v\}$;
\item $\mc{C}$ \emph{has bounded deviation} in the direction $v\in \mathbb{\R}^2_*$ if $\bigcap_{n\in \N} \ol{U}_n$ is contained in some strip of the form $\{z\in \R^2 : -M \leq \langle z;v\rangle \leq M\}$.
\end{itemize}
Note that if $\mc{C}$ is \emph{eventually $r$-quasiconvex} for some $r>0$, then for any $n\in \N$ the set $U_n$ intersects every open ball of radius $r$ contained the convex hull of $\bigcap_{n\in \N} \ol{U}_n$.

The main result of this section is the following.
\begin{theorem}\label{th:geometric} Let $\Sigma\subset \R^2$ be a closed discrete $R$-dense set for some $R>0$, and $\mc{C}=(U_n)_{n\in \N}$ an eventually $\Sigma$-free chain of arcwise connected sets. Then $\mc{C}$ is eventually $r$-quasiconvex for any $r>R$. In addition, one of the following holds:
\begin{enumerate}
\item[(1)] There is $n\in \N$ and $w\in \Sigma$ such that $U_n$ is $\Sigma\sm (\R w)$-free,
\item[(2)] $\mc{C}$ has an asymptotic direction, or
\item[(3)] $\mc{C}$ has bounded deviation in some direction $v\in \R^2_*$, and moreover, there is $M>0$ such that $E=\bigcap_{n\in \N} \ol{U}_n$ separates the half-planes $\{z\in \R^2 : \langle z;v\rangle \geq M\}$ and $\{z\in \R^2 : \langle z;v\rangle \leq -M\}$.
\end{enumerate}
\end{theorem}
Note that if the first case holds, then $U_k$ is $\Sigma\sm (\R w)$-free for any $k\geq n$. 

\subsection{A dynamical consequence}\label{sec:chain-annular}
Before moving to the proof of the geometric results, let us state a somewhat technical dynamical consequence which is in the core of the proof of Theorem \ref{th:teoremao}.

Recall the notation $U_\epsilon'$ from \S\ref{sec:uepsilon}.
\begin{proposition}\label{pro:chain-annular} Let $\hat{f}\colon \R^2\to \R^2$ be the lift of a homeomorphism $f$ of $\T^2$ homotopic to the identity. Suppose that for some $\ol{w}\in \R^2$ and $\hat{x}_0\in \R^2$, the sets $O_n:= U'_{1/n}(\hat{x}_0,\hat{f})$ are such that 
\begin{itemize}
\item $\hat{f}(O_n)=O_n$;
\item $\pi(O_n)$ is essential in $\T^2$;
\item the decreasing chain $(O_n)_{n\in \N}$ is eventually $\Z^2\sm (\R\ol{w})$-free.
\end{itemize}
Then there exists $w\in \R^2_*$ and $M>0$ such that 
$$\abs{\smash{p_w(\hat{f}^n(z)-z)}} \leq M\quad \text{for all } n\in \Z,\, z\in \R^2.$$
Furthermore, if $f$ is not annular, then only case (3) of Theorem \ref{th:geometric} is possible for the chain $\mc{C}=(O_n)_{n\in \N}$.
\end{proposition}

\begin{proof}
Let $\Sigma = \Z^2\sm (\R \ol{w})$. Then $(O_n)_{n\in \N}$ is an eventually $\Sigma$-free chain, and clearly $\Sigma$ is $2$-dense (\ie it intersects every ball of radius $2$). Theorem \ref{th:geometric} implies that one of the following holds:
\begin{enumerate}
\item[(1)] There is $n\in \N$ and $w\in \Sigma$ such that $O_n$ is $\Sigma\sm (\R w)$-free,
\item[(2)] $\bigcap_{n\in \N} \bd_\infty O_n$ is a single point, or
\item[(3)] the set $E=\bigcap_{n\in \N} \ol{O}_n$ is contained in a strip $p_w^{-1}((-M,M))$ for some $w\in \R^2_*$, and $E$ separates the half-plane $p_w^{-1}((-\infty,-M])$ from $p_w^{-1}([M, \infty))$.
\end{enumerate}
We will rule out the first two cases; but first, note that we may assume that $f$ is non-annular (otherwise there is nothing to be done). %The definition of $U'_{1/n}(\hat{x}_0,\hat{f})$ and the fact that $O_n$ is invariant for all $n$ imply that $\hat{x}_0$ is a nonwandering point of $\hat{f}$. By Brouwer's Lemma (see for instance \cite[Corollary 2.4]{fathi}), the fact that $\hat{f}$ has a nonwandering point implies that it has a fixed point. Thus, by Proposition \ref{pro:annular-annular}, we also know that $f^n$ is non-annular for any $n\in \N$. 

Assume that case (1) holds. Let us show that there is $n'>n$ such that $O_{n'}$ is in fact $\Sigma$-free: since $w\in \Sigma \subset \Z^2_*$, we have that $\R w \cap \Sigma \subset \R w \cap \Z^2_* \subset \Z w_0$ for some $w_0\in \Z^2_*$. Since $w\in \R^2\sm (\Z \ol{w})$, it follows that $w_0\in \Z^2\sm (\Z\ol{w})=\Sigma$, and therefore we may find $n'>n$ such that $O_{n'}\cap T_{w_0}(O_{n'})=\emptyset$. By Proposition \ref{pro:translacaodisjunta} we conclude that $O_{n'}\cap T_{kw_0}(O_{n'})=\emptyset$ for all $k\in \Z_*$. Thus $O_{n'}$ is $(\Z_* w_0)$-free. 
Since $O_{n'}\subset O_n$ is also $\Sigma\sm (\R w)$-free, and $\R w\cap \Sigma \subset \Z_* w_0$, we conclude that $O_{n'}$ is $\Sigma$-free, as claimed.

Since $\pi(O_n)$ is open, connected, invariant and essential, and since $f$ is non-annular, Proposition \ref{pro:annular-annular} implies that $\pi(O_n)$ is fully essential for each $n\in \N$. The latter fact implies there exist two non-parallel elements $v,v'\in \Z^2_*$ such that $O_{n'}$ intersects both $T_v(O_{n'})$ and $T_{v'}(O_{n'})$. Since at least one element of $\{v,v'\}$ is outside $\R \ol{w}$ (and thus belongs to $\Sigma$), we have a contradiction. This rules out case (1).  

Since case (1) is ruled out, for each $n\in \N$ there exists some $v_n\in \Sigma$ such that $O_n\cap T_{v_n}(O_n)\neq \emptyset$. By using subsequences, we may assume that $v_n/\norm{v_n}\to v\in \SS^1$. From the fact that $(O_n)_{n\in \N}$ is eventually $\Sigma$-free we easily conclude that $\norm{v_n}\to \infty$.

The choice of $v_n$ and the definition of $U'_{1/n}$ imply that, for each $n\in \N$, there is $z_n\in B_{1/n}(\hat{x}_0)$ and $k_n\in \Z$ such that $\hat{f}^{k_n}(z_n)\in B_{1/n}(\hat{x}_0+v_n)$. Thus, $O_n$ intersects $B_{1/n}(\hat{x}_0+v_n)$. But we also have that $z_n' = \hat{f}^{k_n}(z_n)-v_n\in B_{1/n}(\hat{x}_0)$ and $\hat{f}^{-k_n}(z_n')\in B_{1/n}(\hat{x}_0-v_n)$. Thus $O_n$ also intersects $B_{1/n}(\hat{x}_0-v_n)$. 

Since $O_m\subset O_n$ if $m>n$, we conclude that $O_n$ intersects $B_{1/m}(\hat{x}_0\pm v_m)$ for each $m>n$, and it follows easily that $\{v,-v\} \subset \bd_\infty O_n$. This rules out case (2), so only case (3) is possible. 

Hence there is $w\in \R^2_*$ such that $E=\bigcap_{n\in \N} \ol{O}_n\subset p_w^{-1}((-M,M))$, and $E$ separates the half-planes $$H_1 = p_w^{-1}((-\infty,-M])\, \text{ and } \, H_2=p_w^{-1}([M,\infty)).$$ 
If $W_i$ is the connected component of $\R^2\sm E$ containing $H_i$, then one easily verifies that $W_i$ is invariant for $i\in \{1,2\}$ (since $\hat{f}$ permutes the connected components of $\R^2\sm E$ and $\hat{f}(x)-x$ is uniformly bounded).

Let $u\in \Z^2$ be such that $[0,1]^2-u\subset H_1$ and $[0,1]^2+u\subset H_2$. If $z\in [0,1]^2$, then $z-u\subset H_1$ and so $$\hat{f}^n(z-u)\subset H_1\subset \R^2\sm H_2\subset p_w^{-1}((-\infty,M])$$ for all $n\in \Z$. This means that $p_w(\hat{f}^n(z))-p_w(u) = p_w(\hat{f}^n(z-u)) \leq M$ for all $n\in \Z$, and so $p_w(\hat{f}^n(z)) \leq M+p_w(u)$ for all $n\in \Z$. By a similar argument with $H_2$ one concludes that $p_w(\hat{f}^n(z))\geq -(M+p_w(u))$ for all $n\in \Z$. 

Thus $\abs{\smash{p_w(\hat{f}^n(z))}}\leq M+p_w(u)$ for all $z\in [0,1]^2$, from which we easily conclude that 
$$\abs{\smash{p_w(\hat{f}^n(z)-z)}}\leq M' \quad \text{for all } z\in \R^2,\, n\in \Z,$$
where $M' = M+p_w(u)+\sqrt{2}$.
\end{proof}

%Let us state a consequence that we will need later \xxx{move elsewhere}
%\begin{corollary} Let $(U_n)_{n\in \N}$ be a decreasing sequence of open topological disks in $\R^2$, and let $\ol{w}\in \Z^2$ be such that whenever $v\in \Z^2_*$ is not parallel to $\ol{w}$ there exists $n\in \Z^2$ such that $U_n\cap T_v(U_n)=\emptyset$. Assume in addition that $\bigcap_{n\in \N} \bd_\infty U_n$ has more than one point. Then one of the following holds:
%\begin{enumerate}
%\item[(1)] There is $n$ such that $U_n\cap (U_n+v)=\emptyset$ for all $v\in \Z^2_*$, or
%\item[(2)] $K=\bigcap_{n\in \N} \ol{U}_n$ projects to a bounded set in some direction $v$, \ie $p_v(K)$ is bounded, and in addition, the half-planes $\{z\in \R^2:p_v(z) > M\}$ and $\{z\in \R^2:p_v(z)<-M\}$ are in different connected components of $\R^2\sm K$ if $M$ is large enough.
%\end{enumerate}
%\end{corollary}
% xxx stay or go?

\subsection{A quasi-convexity lemma}

We begin with a general result that leads to the quasi-convexity part of Theorem \ref{th:geometric}.

\begin{lemma}\label{lem:quasiconvex} Let $\Sigma \subset \R^2$ be a closed discrete set, and $(U_n)_{n\in \N}$ an eventually $\Sigma$-free chain of arcwise connected subsets of $\R^2$. If $Q\subset \R^2$ is a bounded connected set such that
$$\ol{Q}\subset \inter \conv\bigg(\bigcap_{n\in \N} \ol{U}_n\bigg) \quad \text{ and } \quad \bigcup_{v\in \Sigma} T_v(\ol{Q}) = \R^2,$$
then $U_n\cap Q\neq \emptyset$ for each $n\in \N$. 
\end{lemma}

\begin{corollary}\label{cor:quasiconvex} Suppose $(U_n)_{n\in \N}$ is a decreasing sequence of arcwise connected sets such that $\conv\left(\bigcap_{n\in \N} \ol{U}_n\right) = \R^2$, and for each $v\in \Z^2_*$ there is $n$ such that $U_n\cap T_v(U_n)=\emptyset$. Then $U_n$ intersects the square $[a,a+1)\times [b,b+1)$, for any $(a,b)\in \R^2$ and $n\in \N$.
\end{corollary}

Let us introduce a definition before moving to the proof.
Given $z\in \R^2$ and an arc $\gamma\colon [0,1]\to \R^2$ such that $z\notin \gamma([0,1])$, we define a partial index as follows: consider the map
$$\xi\colon [0,1]\to \SS^1, \quad \xi(t) = \frac{\gamma(t)-z}{\norm{\gamma(t)-z}}$$
and let $\hat{\xi}\colon [0,1]\to \R$ be a lift to the universal covering, so that $e^{2\pi i\hat{\xi}(t)} = \xi(t)$. 
Then we define 
$$I(\gamma,z) = \hat{\xi}(1)-\hat{\xi}(0).$$
This number does not depend on the choice of the lift $\hat{\xi}$ or the parametrization of $\gamma$ (preserving orientation). If $\gamma$ is a closed curve, $I(\gamma,z)$ is an integer and  coincides with the winding number of $\gamma$ around $z$.
If $\gamma$ and $\gamma'$ are arcs with $\gamma(1)=\gamma'(0)$ and $z\notin [\gamma]\cup[\gamma']$, then
$$I(\gamma*\gamma',z) = I(\gamma,z)+I(\gamma',z).$$
Additionally, $I(\gamma,z)$ is invariant by homotopies in $\R^2\sm \{z\}$ fixing the endpoints of $\gamma$. A simple consequence of this fact is that if $I(\gamma,z)\neq 0$ and $\gamma$ is closed, then $z$ must be in a bounded connected component of $\R^2\sm [\gamma]$.

%\begin{lemma} If $E\subset \R^2$ is connected and $z\in \conv(S)$ for some subset $S$ of $E$, then there is $z'\in S$ such that $E\cap E+(z-z')\neq \emptyset$
%\end{lemma}
%
The following proposition  is proved in \cite{m-z} and attributed to A. Douady. 

\begin{proposition}[\cite{m-z}]\label{pro:douady} Let $\gamma$ be a simple arc joining two points $z_0$ and $z_1$ and not intersecting the line segment $\ell$ from $z_0$ to $z_1$ other than at its endpoints. If $z_0+v$ lies in the closure of the disk bounded by $[\gamma]\cup \ell$, then $[\gamma]$ intersects $[\gamma]+v$. 
\end{proposition}

\begin{remark} The original statement uses ``open disk'' instead of ``closed disk'', but since $\gamma$ is compact the two statements are equivalent.
\end{remark}

\begin{lemma} \label{lem:convex-index} Let $E\subset \R^2$ be an arcwise connected set, $P$ a convex polygon with vertices in $E$, and $z\in P\sm E$. Suppose additionally that there is no simple loop in $E$ bounding a disk that contains $z$. Then there is a segment $\ell$ contained in an edge of $P$ and a simple arc $\gamma$ in $E$ joining the endpoints of $\ell$ and not intersecting $\ell$ anywhere else, such that the (closed) disk bounded by $[\gamma]\cup [\ell]$ contains $z$.
\end{lemma}

\begin{figure}[ht]
\begin{minipage}[b]{0.44\linewidth}
\centering
\includegraphics[height=4cm]{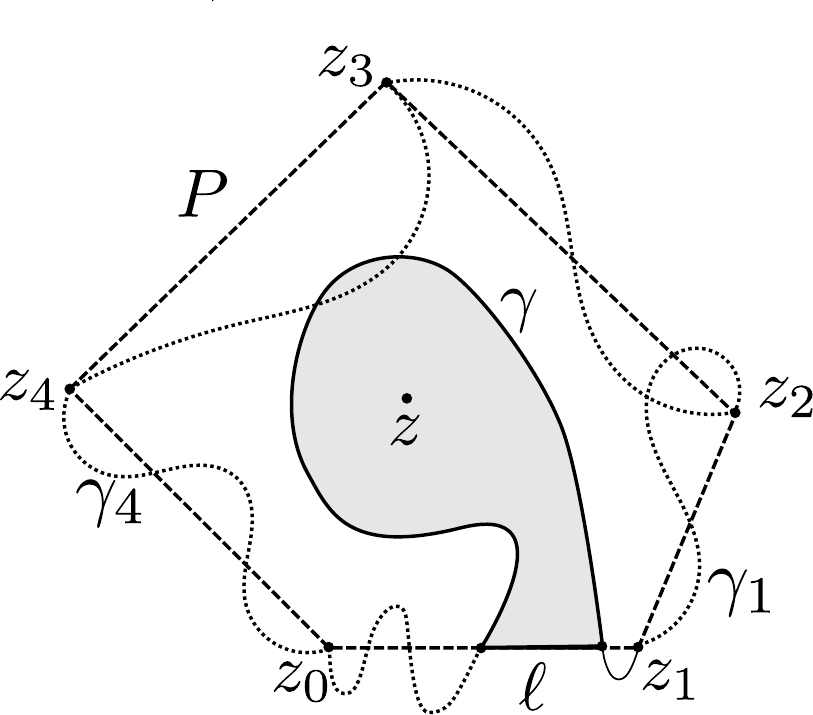}
\caption{Lemma \ref{lem:convex-index}}
\label{fig:qc1}
\end{minipage}
\hfill
\begin{minipage}[b]{0.55\linewidth}
\centering
\includegraphics[height=4cm]{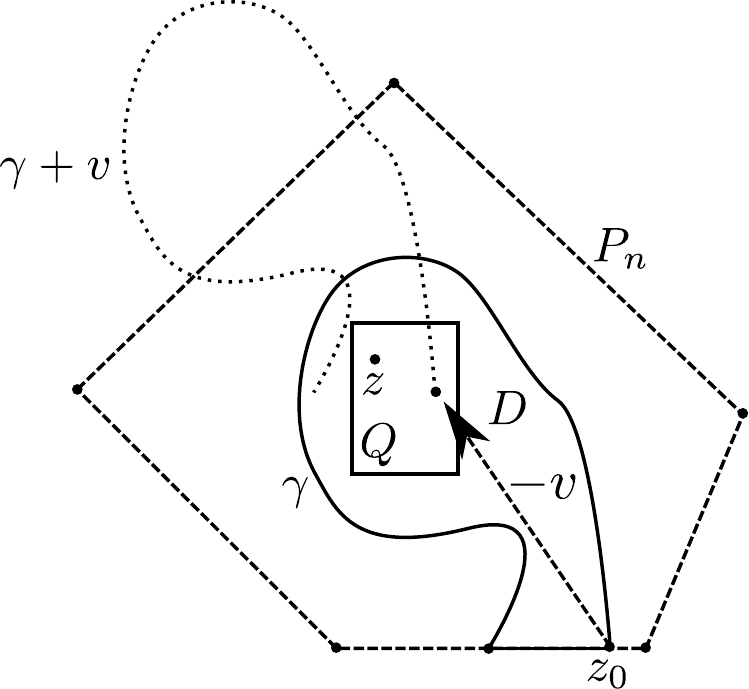}
\caption{Proof of Lemma \ref{lem:quasiconvex}}
\label{fig:qc2}
\end{minipage}
\end{figure}

\begin{proof} 
If $z$ belongs to some edge $\ell'$ of $P$, we may choose any simple arc $\gamma'$ in $E$ joining the endpoints of $\ell'$, and consider a parametrization $\ell$ of the connected component of $[\ell']\sm [\gamma']$ that contains $z$. The two endpoints of $\ell$ are joined by some sub-arc $\gamma$ of $\gamma'$, which does not intersect $\ell$ elsewhere, and since $z\in [\ell]$, the required properties hold.

Now assume $z\in \inter P$, and let $z_0,\dots, z_{n-1}$ be the (positively) cyclically ordered vertices of $P$ (see Figure \ref{fig:qc1}). For each $i\in \{0,\dots,n-1\}$ let $\gamma_i \colon [0,1]\to E$ be a simple arc joining $z_i$ to $z_{i+1\, (\text{mod }n)}$ Using the notation $I(\gamma)= I(\gamma,z)$, we first observe that
$$I(\gamma_0) +I(\gamma_1)+\cdots+ I(\gamma_{n-1})=0.$$ 
This is because $\gamma_0*\cdots*\gamma_{n-1}$ is a loop in $E$, and if $I(\gamma_0*\cdots*\gamma_{n-1})\neq 0$ then $z$ lies in some bounded component $D$ of $\R^2\sm[\gamma_0*\cdots*\gamma_{n-1}]$, and then $\bd D$ is a simple loop in $E$ bounding a disk that contains $z$, contradicting our hypotheses.
 
Denote by $\ell_i\colon [0,1]\to \R^2$ the parametrized edge of $P$ from $z_i$ to $z_{i+1\, (\text{mod }n)}$. Being a straight segment, it is clear that $I(\ell_i)<1/2$. On the other hand, from the fact that $\bd P$ is a positively oriented simple loop and $z$ is in the interior of $P$, it follows that $I(\ell_0)+I(\ell_1)\cdots+I(\ell_{n-1}) = 1$.
From these facts we see that
$$I(\gamma_0*(-\ell_0))+I(\gamma_{1}*(-\ell_{1}))+\cdots + I(\gamma_{n-1}*(-\ell_{n-1})) = -1.$$
Since each $\gamma_i*(-\ell_i)$ is closed, $I(\gamma_i*(-\ell_i))\in \Z$, and the above equation implies that there is some $k\in \{0,\dots,n-1\}$ such that $I(\gamma_k*(-\ell_k))\neq 0$. 
By a standard argument, we show that $\gamma_k*(-\ell_k)$ contains a simple loop that bounds a disk containing $z$: let $A=\{s\in [0,1]: \gamma_k(s) \in [\ell_k]\}$, and for each $s\in A$ let $t_s\in [0,1]$ be the (unique) number such that $\gamma_k(s) =  \ell_k(t_s)$. Note that $A$ is closed and $s\mapsto t_s$ is continuous, so if we
use the notation $\gamma_k^{st}$ to represent the sub-arc $\gamma_k|_{[s,t]}$ affinely reparametrized to have domain $[0,1]$, one easily verifies that the map $s\mapsto I(\gamma_k^{0s}*(-\ell_k^{0t_s}))$ is continuous as well. Thus, there is a largest $s\in A$ such that $I(\gamma_k^{0s}*(-\ell_k^{0t_s}))=0$, and a smallest $r\in A$ such that $r>s$ (for the latter, note that if $r\in A$ and $r>s$, then $I(\gamma_k^{0r}*(-\ell_k^{0t_r}))$ is a nonzero integer, and by continuity there has to be a smallest such $r$). From our choice of $s$, 
$$I(\gamma_k^{sr}*(-\ell_k^{sr})) = I(\gamma_k^{0r}*(-\ell_k^{0r})) - I(\gamma_k^{0s}*(-\ell_k^{0s})) = I(\gamma_k^{0r}*(-\ell_k^{0r})) \neq 0,$$ 
and from our choice of $r$ follows that $\gamma_k^{sr}$ does not intersect $\ell_k$ other than at its endpoints. Since both $\gamma_k$ and $\ell_k$ are simple arcs, it follows that $\gamma_k^{sr}*(-\ell_k^{sr})$ is a simple loop, and the disk it bounds contains $z$ because of the nonzero index. Thus $\gamma=\gamma_k^{sr}$ is the required arc. This completes the proof.
\end{proof}

%\begin{lemma} Let $U\subset \R^2$ be an open, connected set such that, for every $v\in \Z^2_*$,  $T_v(U)$ is disjoint from $U$. Then $U$ is $\sqrt{2}$-quasiconvex.
%\end{lemma}
%
%\begin{proof} Suppose by contradiction that there is some $z\in \conv (U)$ such that $B_{\sqrt{2}}(z)$ is disjoint from $U$. From the fact that $U$ is convex and connected, it is easy to see that there exist $x,y\in U$ such that $z$ lies in the line segment $L_{xy}$ joining $x$ to $y$. Consider the set $U\cup L_{xy}$ 
%\end{proof}

\begin{proof}[Proof of Lemma \ref{lem:quasiconvex}]
Suppose for contradiction that $Q\cap U_{n_0} = \emptyset$ for some $n_0\in \N$. Since the sets $U_i$ are nested, we have $Q\cap U_n=\emptyset$ for any $n\geq n_0$. 
By Steinitz' theorem, each point of $\ol{Q}$ has a neighborhood contained in the convex hull of some finite subset of $U=\bigcap_{n\in \N} \ol{U}_n$, and so by compactness we can find a finite set $S\subset U$ such that $\ol{Q}\subset \inter P$, where $P=\conv{S}$. Let $W$ be a bounded neighborhood of $P$.

Note that the set 
$$V = \{x-y:x\in W,\, y\in W\}\cap \Sigma$$
is bounded, hence finite (because $\Sigma$ is closed and discrete). Thus we can find $n_1>n_0$ such that if $n>n_1$ then $U_n\cap T_v(U_n) = \emptyset$ for any $v\in V$. 

From now on, fix $n>n_1$ and $z\in Q$. Since $\ol{Q}\subset \inter{P}$ and the (finitely many) extremal points of $P$ are in $\ol{U}_n$, by a small perturbation of these points we obtain a new convex polygon $P_n$ with extremal points in $U_n$ such that $\ol{Q}\subset \inter{P_n}$, and $P_n\subset W$. By Lemma \ref{lem:convex-index} applied to $E=U_n$ and $P_n$ instead of $P$, there are two possibilities:

\emph{Case 1.} There is a simple loop $\alpha$ in $U_n$ bounding a disk $D$ containing $z$. Since $Q$ is connected and disjoint from $U_n$ (and so from $[\alpha]$), and $Q\cap D\neq \emptyset$,  it follows that $Q\subset D$. The fact that $\R^2=\bigcup_{v\in \Sigma} \ol{Q}+v$ implies that there is $v\in \Sigma$ such that $\ol{Q}\cap T_v(\ol{Q})\neq \emptyset$. Indeed, a classic theorem of Sierpinski \cite{sierpinski} implies that $\R^2$ is not a countable union of pairwise disjoint closed sets. But then $\ol{D}\cap T_v(\ol{D})\neq \emptyset$, which implies that $[\alpha]\cap T_v([\alpha])\neq \emptyset$ and therefore $U_n\cap T_v(U_n)\neq \emptyset$. On the other hand, since $\ol{Q}\subset \inter P_n\subset W$, we have that $v\in V$ and so $U_n\cap (U_n+v)=\emptyset$, a contradiction.

\emph{Case 2.}
There is an arc $\gamma$ in $U_n$ joining two points of a subset $\ell$ of an edge of $P_n$ such that $[\gamma]\cup [\ell]$ bounds an open disk $D$ such that $z\in \ol{D}$. Since $Q\subset \inter P_n$ is connected and disjoint from $[\gamma]\cup[\ell]$, it follows that $Q\subset D$. Letting $z_0=\gamma(0)$, from the fact that $\bigcup_{v\in \Sigma}\ol{Q}+v=\R^2$ we see that there is $v\in \Sigma$ such that $z_0-v\in \ol{Q}\subset \ol{D}$. Proposition \ref{pro:douady} implies that $[\gamma]\cap T_{-v}([\gamma])\neq \emptyset$ (see Figure \ref{fig:qc2}), and so $U_n\cap T_{-v}(U_n)\neq\emptyset$, which means that $U_n\cap T_v(U_n)\neq \emptyset$. But since $z_0\in \bd P_n\subset W$ and $z_0-v \in \ol{Q}\subset \inter P_n\subset W$, we have that $v\in V$ which implies that $U_n\cap T_v(U_n)=\emptyset$, again a contradiction.
\end{proof}

\subsection{Proof of Theorem \ref{th:geometric}}

%In this subsection we prove the following geometrical lemma, that will allow us in the future to obtain the restriction on the dynamics of rational pseudo-rotations:
%
%\begin{lemma}\label{lem:ineessencialoutemdirecao}
% Let $\ol{w}\in\SS^1$ and let  $\{U_n\}_{n\in \N}$ be a sequence of arcwise connected subsets of $\R{}^2$ such that 
%\begin{itemize}
%\item $U_{n+1}\subset U_n$ for each $n\in \N$;
%%\item $\bigcap_{n\in \N} \ol{U}_n\neq \emptyset$.
%\item for each $w\in \Z^2_*$ not parallel to $\ol{w},$ there is $n\in\N$ such that $U_n\cap T_w(U_n)=\emptyset$. 
%\end{itemize}
%Then one of the following properties holds:
%\begin{enumerate}
%\item{There is $n_0\in \N$ such that $\pi(U_{n})$ is inessential for all $n>n_0$, or} 
%\item{there is $v\in \SS^1$ such that $\bigcap_{n\in\N}\partial_{\infty}\overline{U}_n\subset \{-v,v\}$}. 
%\end{enumerate}
%If, in addition, there is $v\in \SS^1$ such that $\bigcap_{n\in\N}\partial_{\infty}\overline{U}_n= \{-v,v\},$
%%replace for "has more than one point"?
%then $E=\left(\bigcap_{n\in\N}\overline{U}_n\right)$ is contained in a strip $S$ bounded by two lines parallel to $\R v$, and each connected component of $\R^2\sm S$ lies in a different connected component of $\R^2\setminus E$.
%\end{lemma}
%

We will use two classical properties of translations, derived from Brouwer theory.
\begin{definition} Given $v\in \R^2_*$, we say that $\gamma\colon [0,1]\to\R^2$ is a $T_v$-translation arc if $\gamma$ is a simple arc joining a point $x$ to $T_v(x)$ and $[\gamma]\cap T_v[\gamma] =\{T_v(x)\}$.
\end{definition}

\begin{proposition}[Corollary 3.3 of \cite{MortonBrown}]\label{pro:translacaodisjunta}
If $K\subset \R^2$ is an arcwise connected set and $v\in\R^2$ is such that $K\cap T_v(K)=\emptyset,$ then 
$K\cap T_v^n(K)=\emptyset$ for all $n \in \Z_*$.   
\end{proposition}

The next proposition is a direct consequence of Theorem 4.6 of \cite{MortonBrown}.
\begin{proposition}\label{pro:naotocaoarco} 
If $K\subset\R^2$ is an arcwise connected set such that $K\cap T_v(K)=\emptyset$ for some $v\in \R^2_*$, and $\alpha$ is any $T_v$-translation arc disjoint from $K$, then
$$K\cap\left(\bigcup_{i=0}^{\infty} T_v^{i}[\alpha]\right)=\emptyset \quad \text {or} \quad K\cap\left(\bigcup_{i=0}^{\infty} T_v^{-i}[\alpha]\right)=\emptyset$$
\end{proposition}

We will also need the following
\begin{proposition}\label{pr:temarcodetranslacao}
If $K\subset\R^2$ is arcwise connected and $K\cap T_v(K)\neq \emptyset$ for some $v\in \R^2_*$, then $K$ contains a $T_v$-translation arc. 
\end{proposition}
\proof
By hypothesis, there is $y\in K$ be such that $T_v(y)\in K$. Since $K$ is arcwise connected, there is an arc $\gamma\colon [0,1]\to K$, which we may assume simple, joining $y$ to $T_v(y)$. Define $F\colon[0,1]\times[0,1]\to\R^2$ as $F(s,t)=\gamma(t)-\gamma(s),$ and let $(s_0,t_0)$ be a point in the closed set $F^{-1}(v)$ that minimizes the map 
$$F^{-1}(v)\ni (s,t)\mapsto \abs{s-t},$$
Then, the restriction of $\gamma$ to the interval between $s_0$ and $t_0$ is a $T_v$-translation arc in $K$.
\endproof

To prove Theorem \ref{th:geometric}, we begin with an auxiliary result.
\begin{lemma}\label{lem:90graus}
Let $(U_n)_{n\in \N}$ be as in the hypotheses of Theorem \ref{th:geometric}, and let $v_0\in \Sigma$.
Then one of the following holds:
\begin{enumerate}
\item There is $n\in \N$ and $w\in \Sigma$ such that $U_{n}$ is $\Sigma\sm (\R w)$-free, or 
\item For every $\epsilon>0$, there is an open arc $I_\epsilon$ in $\SS^1_\infty$ of length $\frac{\pi}{2}-\epsilon$ with an endpoint $v_0/\norm{v_0}$ such that  $\partial_{\infty} \overline{U}_{n} \cap I_\epsilon =\emptyset$ for some $n\in \N$.
\end{enumerate}
\end{lemma}

\proof
Given $\epsilon>0$, from the fact that $\Sigma$ is $R$-dense we may find $v_0'\in \Sigma$ such that $$\frac{\pi}{2}-\epsilon<\mathrm{angle}(v_0, v_0')<\frac{\pi}{2}+\epsilon.$$ 
Let $n_0\in \N$ be such that both $U_{n_0}\cap T_{v_0}\left(U_{n_0}\right)$ and $U_{n_0}\cap T_{v_0'}\left(U_{n_0}\right)$ are empty.

If $U_{n_0}$ is $\Sigma$-free, then case (1) of the lemma holds and we are done. Otherwise, we may choose $w_1\in \Sigma$ such that $U_{n_0}$ intersects $T_{w_1}(U_{n_0})$. By Proposition \ref{pr:temarcodetranslacao}, there exists a $T_{w_1}$-translation arc $\alpha$ in $U_{n_0}$, joining a point $\hat{y}$ to $T_{w_1}(\hat{y})$. Moreover, since $w_1\in \Sigma$, we may choose $n_1>n_0$ such that $U_{n_1}\cap (T_{w_1}(U_{n_1}))=\emptyset$. If $U_{n_1}$ is $\Sigma\sm(\R w_1)$-free then again case (1) holds and we are done; otherwise, there is $w_2\in \Sigma \sm (\R w_1)$ such that $U_{n_1}\cap T_{w_2}(U_{n_1})\neq \emptyset$. Since $U_{n_1}$ is a subset of $U_{n_0}$, by Proposition \ref{pr:temarcodetranslacao} there is a $T_{w_2}$-translation arc $\beta$ in $U_{n_0}$ joining a point $\hat z$ to $T_{w_2}\hat{z}$.

Let $\gamma$ be an arc in $U_{n_0}$ joining $\hat y$ to $\hat z$, define 
$$\alpha^+=\bigcup_{i=0}^{\infty} T_{w_1}^{i}[\alpha],\quad \alpha^-=\bigcup_{i=0}^{\infty} T_{w_1}^{-i}[\alpha],\quad \beta^+=\bigcup_{i=0}^{\infty} T_{w_2}^{i}[\beta],\quad \beta^-=\bigcup_{i=0}^{\infty} T_{w_2}^{-i}[\beta],$$
and consider the four connected sets
$$C_1=\alpha^+\cup\beta^{+}\cup\gamma, \quad C_2=\alpha^+\cup\beta^{-}\cup\gamma,\quad C_3=\alpha^-\cup\beta^{+}\cup\gamma,\quad C_4=\alpha^-\cup\beta^{-}\cup\gamma.$$

\begin{claim}\label{cl:Bolanocone}
Given $r>0$ and $i\in\{1,2,3,4\}$, there exist two points $z_i$ and $z_i'$ such that $B_r(z_i)$ and $B_r(z_i')$ lie on different connected components of $\R^2\sm C_i$.
\end{claim}

\begin{proof}\setcounter{claim}{1}
We consider the case $i=1$; the other cases are analogous. Note first that $\alpha^+\cup \gamma$ is contained in a half-strip $S_1$ with the direction $w_1$ (\ie a set of the form $\{z : a \leq p_{w_1^\perp}(z)\leq b, \, p_{w_1}(z)\geq c\}$). Similarly, $\beta^+\cup \gamma$ is contained in a half-strip $S_2$ with direction $w_2$ (see Figure \ref{fig:claim1}). Let $O_1$ and $O_2$ be the two connected components of $\R^2\sm (S_1\cup S_2)$. It is easy to verify that $O_1$ and $O_2$ lie in different connected components of $\R^2\sm C_1$, and since each $O_i$ contains a cone, one may find a ball of arbitrarily large radius in each of the two sets.
\begin{figure}[ht]
\includegraphics[height=4cm]{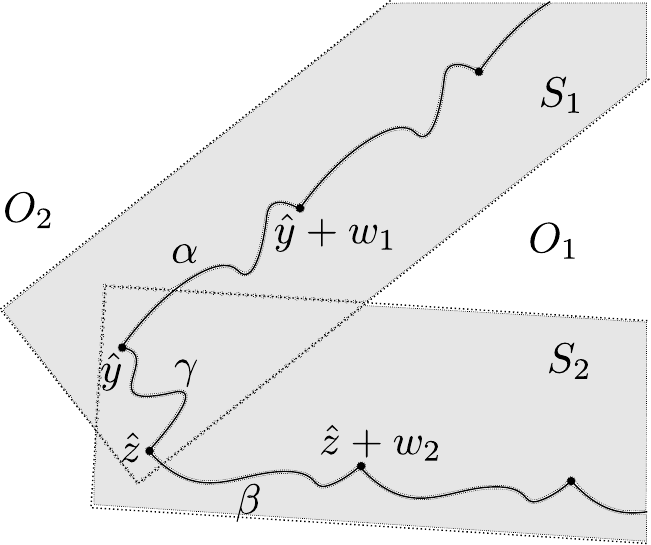}
\caption{The sets $O_1$ and $O_2$.}
\label{fig:claim1}
\end{figure}
\end{proof}

\begin{claim}\label{cl:Bolanocone2}
Given $r>0$, there is $R>0$ such that for any arcwise connected set $K\subset \R^2\sm ([\alpha]\cup[\beta]\cup[\gamma])$ such that $K$ is disjoint from $T_{w_1}(K)$ and $T_{w_2}(K)$ there is $x$ such that $\norm{x}\leq R$ and $B_r(x)$ is also disjoint from $K$.
\end{claim}
\begin{proof}
Fix $r>0$ and let $z_i$ and $z_i'$ be the points from Claim \ref{cl:Bolanocone}, for $i\in \{1,2,3,4\}$. We choose $R$ such that $R>\norm{z_i}$ and $R>\norm{z_i'}$ for $i\in \{1,2,3,4\}$. 

The fact that $K$ is disjoint from $T_{w_1}(K)$ and from $\alpha$ implies, by Proposition \ref{pro:naotocaoarco}, that $K$ is disjoint from one of the sets $\alpha^+$ or $\alpha^-$. Similarly, since $K$ is disjoint from $T_{w_2}(K)$ and from $\beta$, it must be disjoint from one of the sets $\beta^+$ or $\beta^-$. Since $K$ is also disjoint from $\gamma$, it follows that $K$ is disjoint from $C_i$ for some $i\in \{1,2,3,4\}$. Since $K$ is connected, it lies entirely in one connected component of $\R^2\sm C_i$, so Claim \ref{cl:Bolanocone} implies that $K$ is disjoint from $B_r(x)$ where $x$ is either $z_i$ or $z_i'$. 
\end{proof}

We now fix $n_2>n_1$ such that $U_{n_2}$ is disjoint from both $T_{w_1}(U_{n_2})$ and $T_{w_2}(U_{n_2})$. Recall from the beginning of the proof that $U_{n_0}$ is disjoint from $T_{v_0}(U_{n_0})$ and $T_{v_0'}(U_{n_0})$. Since it is arcwise connected,  Proposition \ref{pro:translacaodisjunta} implies that $U_{n_0}$ is also disjoint from $T^k_{v_0}(U_{n_0})$ and $T^k_{v_0'}(U_{n_0})$ for any given $k\in \Z_*$. 

Fix $r=2\max\{\norm{v_0},\norm{v_0'}\}$, and let $R$ be as in Claim \ref{cl:Bolanocone2}. Given $k\in \N$, we have that $T^{-k}_{v_0}(U_{n_2})$ is disjoint from $U_{n_0}$ (because $U_{n_2}\subset U_{n_0}$). In particular, $T^{-k}_{v_0}(U_{n_2})$ is disjoint from $[\alpha]\cup[\beta]\cup[\gamma]$. By Claim \ref{cl:Bolanocone2} applied to $K=T^{-k}_{v_0}(U_{n_2})$ we conclude that there is $x_k$ such that $\norm{x_k}\leq R$ and $B_r(x_k)$ is disjoint from $T^{-k}_{v_0}(U_{n_2})$. This means that $U_{n_2}$ is disjoint from $B_r(y_k)$, where $y_k=x_k+kv_0$. 

Since $U_{n_2}$ is disjoint from $B_r(y_k)$, it is disjoint also from the straight line segment joining $y_k$ to $y_k+v_0$ (which is a $T_{v_0}$-translation arc). Thus, recalling that $U_{n_2}$ is disjoint from $T_{v_0}(U_{n_2})$, Proposition \ref{pro:naotocaoarco} implies that $U_{n_2}$ is disjoint from either $y_k+\R^{+} v_0$ or from $y_k+\R^{-} v_0$. 

We examine two possibilities. First, assume that for all $k\in \N$, the set $U_{n_2}$ is disjoint from $y_k+\R^{-} v_0$.  
We claim that in this case $U_{n_2}$ is disjoint from one of the two half-planes $S_1=\{x:p_{v_0^\perp}(x)>R\}$ or $S_2=\{x:p_{v^\perp}(x)<-R\}$ (see Figure \ref{fig:90graus1}). In fact, if $U_{n_2}$ intersects both $S_1$ and $S_2$, it contains an compact arc $\sigma$ joining a point of $S_1$ to a point of $S_2$. Since $y_k \in B_R(kv_0)$, it follows that $y_k+\R^{-} v_0$ intersects $\sigma$ if $k$ is chosen large enough, contradicting the fact that $U_{n_2}$ is disjoint from $y_k+\R^{-}v_0$. Thus $U_{n_2}$ is disjoint from $S_1$ or $S_2$, and this implies that $\bd_{\infty} {U_{n_2}}$ is disjoint from an open interval of length $\pi$ with one endpoint in $v_0/\norm{v_0}$, concluding the proof of Lemma \ref{lem:90graus} in this case.

\begin{figure}[ht]
\begin{minipage}[b]{0.49\linewidth}
\centering
\includegraphics[height=4cm]{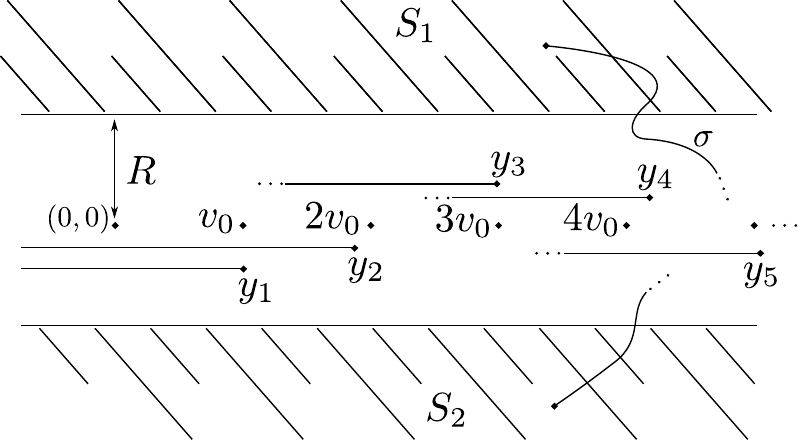}
\caption{-}
\label{fig:90graus1}
\end{minipage}
\hfill
\begin{minipage}[b]{0.4\linewidth}
\centering
\includegraphics[height=4cm]{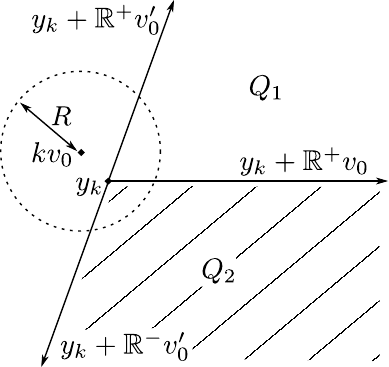}
\caption{}
\label{fig:90graus2}
\end{minipage}
\end{figure}

The second possibility is that, for some $k\in \N$, the set $U_{n_2}$ is disjoint from $y_k + \R^{+} v_0$. Since $U_{n_2}$ is disjoint from $B_{r}(y_k)$ with $r\geq 2\norm{v_0}$, it follows that $U_{n_2}$ is also disjoint from the line segment joining $y_k$ to  $y_k+v_0'$, which is a $T_{v_0'}¡$-translation arc. Since $U_{n_2}\subset U_{n_0}$ and $U_{n_0}$ is disjoint from $T_{v_0'}^k(U_{n_0})$, we also have that $U_{n_0}$ is disjoint from $T_{v_0'}^k(U_{n_2})$. Thus Proposition \ref{pro:naotocaoarco} again implies that $U_{n_2}$ is disjoint from either
$y_k+\R^{+} v_0'$ or $y_k+\R^{-} v_0'$. See figure \ref{fig:90graus2}.

We claim that $U_{n_2}$ is disjoint from one of the two ``quadrants'' 
$$Q_1 = \{y_k+xv_0+yv_0': x\geq 0, y\geq 0\}, \quad  Q_2 = \{y_k+xv_0+yv_0' : x\geq 0, y\leq 0\}.$$ In fact, in the case that $U_{n_2}$ is disjoint from $y_k+\R^{+} v_0'$, since it is also disjoint from $y_k+\R^+ v_0$, it follows that $U_{n_2}$ is disjoint from $\bd Q_1$, so that either $U_{n_2}\subset Q_1$ (in which case it is disjoint from $Q_2$) or $U_{n_2}$ is disjoint from $Q_1$. Similarly, in the case that $U_{n_2}$ is disjoint from $y_k+\R^-v_0'$, it follows that $U_{n_2}$ is disjoint from $\bd Q_2$, so either $U_{n_2}$ is contained in $Q_2$ (hence disjoint from $Q_1$) or $U_{n_2}$ is disjoint from $Q_2$, proving our claim.

Since $U_{n_2}$ is disjoint form one of $Q_1$ or $Q_2$ and $\pi/2-\epsilon<\mathrm{angle}(v_0, v_0')<\pi/2+\epsilon$, it follows that one of the two open intervals of length $\frac{\pi}{2}-\epsilon$ in $\bd_\infty U_{n_2}$ with an endpoint $v_0/\norm{v_0}$ is disjoint from $\bd_{\infty} U_{n_2}$, completing the proof of Lemma \ref{lem:90graus}.
\endproof

%\begin{proposition}\label{pr:direcaonoinfinitotemretanoconvexo}
%Let $U\subset \R{}^2,$ and $v\in \SS^1$ be such that $v\in \partial_{\infty} U.$ Then, if $p\in U,$ then  the semi-line $p+\R{}_{+} v$ is contained in $\overline{Conv(U)}$  
%\end{proposition}
%\proof
%Let $(q_n)_{n\in\N}$ be a sequence of points in $U$ such that $\lim_{n\to\infty}\frac{q_n}{\norm{q_n}}=v.$ Then the sequence of straight line segments joining $p$ to $q_n$ are all in $Conv(U),$ and since these segments converge to the semi-line  $p+\R{}_{+} v$ we have the result. 
%\endproof
%
We are now ready to prove Theorem \ref{th:geometric}.
\begin{proof}[Proof of Theorem \ref{th:geometric}]
Let $$K=\bigcap_{n\in\N}\partial_{\infty}\overline{U}_n\subset \SS^1.$$ 
Assume that case (1) of the theorem does not hold. 
Let us first show that $K\subset \{-v, v\}$ for some $v\in \SS^1$. Indeed, if this is not the case, then there are two different directions $v_1$ and $v_2$ in $K$ such that $\mathrm{angle}(v_1, v_2)<\pi$. Using the fact that $\Sigma$ is $R$-dense, we may find $v_0\in \Sigma$ and $\epsilon>0$ such that $\mathrm{angle}(v_1, v_0) < \frac{\pi}{2}-\epsilon$ and $\mathrm{angle}(v_0,v_2)<\frac{\pi}{2}-\epsilon$ (it suffices to choose $v_0$ such that $v_0/\norm{v_0}$ is close enough to the midpoint of the smaller interval between $v_1$ and $v_2$ in $\SS^1$). 
Since both $v_1$ and $v_2$ belong to $\bd_\infty \ol{U}_n$ for each $n\in \N$, case (2) of Lemma \ref{lem:90graus} cannot hold. Thus case (1) of Lemma \ref{lem:90graus} holds, and this contradicts our assumption that case (1) of the theorem does not hold.

Thus $K\subset \{-v,v\}$ for some $v\in \SS^1$. To see that $K$ is nonempty, it suffices to show that $U_n$ is unbounded for each $n\in \N$. Suppose on the contrary that $U_{n_0}$ is bounded for some $n_0\in \N$. Since $(U_n)_{n\in \N}$ is a decreasing chain, the sets $W_n=\{v\in \Sigma: T_v(U_n)\cap U_n\neq\emptyset\}$ define a decreasing chain of sets as well, and our assumption that $U_{n_0}$ is bounded implies that $W_{n_0}$ is bounded as well. Being a bounded subset of the closed discrete set $\Sigma$, it follows that $W_{n_0}$ is finite. Since $(U_n)_{n\in \N}$ is eventually $\Sigma$-free, we may choose $n\geq n_0$ so large that $U_n\cap T_v(U_n)=\emptyset$ for all $v\in W_{n_0}$, and since $W_n\subset W_{n_0}$ it follows that $W_n=\emptyset$. This means that $U_n$ is $\Sigma$-free, again contradicting our assumption that case (1) of the theorem does not hold.

Thus $K$ is nonempty. If $K$ has a single element, then case (2) of the Theorem holds, and we are done.  We are left with the case where $K= \{-v, v\}$.  Let
$$E = \bigcap_{n\in \N} \ol{U}_n.$$
Let us first show that $\bd_\infty E = K$. To do this, fix $k\in \Z$ and consider the closed sets $A_n = \ol{U}_n\cap (kv + \R v^\perp)$. Note that the fact that $\bd_\infty \ol{U}_n$ contains both $v$ and $-v$ implies that $A_n$ is nonempty. Moreover, $A_n$ is bounded if $n$ is chosen large enough: indeed, if $A_n$ is unbounded, then $\bd_\infty \ol{U}_n$ contains either $v^\perp$ or $-v^\perp$. But if $n$ is large enough, then $\bd_\infty \ol{U}_n$ cannot contain $v^\perp$; otherwise, since the sets $\ol{U}_n$ are nested, it would follow that $v^\perp\in K$, which is a contradiction (and similarly, $-v^\perp$ is not in $\bd_\infty \ol{U}_n$ if $n$ is large enough). The boundedness of $A_n$ for large $n$ implies that $\bigcap_{n\in \N} A_n \subset E\cap(kv+\R v^\perp)$ is a nested intersection of compact sets, hence nonempty. 
Thus we can choose a sequence of points $x_k\in E\cap(kv + \R v^\perp)$ for each $k\in \Z$. Choosing an appropriate subsequence $(k_i)_{i\in \N}$ with $k_i\to \pm\infty$ when $i\to \pm \infty$, we may assume that $x_{k_i}/\norm{x_{k_i}}\to u^{\pm}$ as $i\to \pm\infty$, where $u^+$ and $u^-$ are elements of $\SS^1$ with $p_v(u^-)\leq 0\leq p_v(u^+)$. Since $u^\pm\in  \bd_\infty E \subset K = \{-v,v\}$, it follows at once that $u^+=v$ and $u^-=-v$. Thus $\bd_\infty E = \{-v,v\}=K$, as claimed.

To prove the quasiconvexity, observe that since $\Sigma$ is $R$-dense, if $Q\subset \R^2$ is an open ball of radius greater than $R$, then $\bigcup_{v\in \Sigma} T_v(Q) = \R^2$. In particular, if $B\subset \conv(E)$ is an open ball of radius $r>R$, then $B$ contains the closure of some open ball $Q$ of radius greater than $R$, and so Lemma \ref{lem:quasiconvex} implies that $U_n$ intersects $Q$ for each $n\in \N$. Since $\ol{Q}\subset B$ is compact and $E$ is a decreasing intersection of closed sets, we conclude that $E\cap \ol{Q}\neq \emptyset$ and therefore $E$ intersects $B$. This proves that $E$ is $r$-quasiconvex for any $r>R$, as required.

Now fix $x_0\in E$, and recall that the closed convex hull $\ol{\conv(E)}$ is the intersection of all closed half-planes containing $E$, i.e. all sets of the form $\{x\in \R^2 : p_w(x)\geq t\}$ or $\{x\in \R^2 : p_w(x)\leq t\}$ containing $E$, for $w\in \R^2_*$ and $t\in \R$. The fact that $\bd_\infty E=\{-v,v\}$ implies that any half-plane containing $E$ must be bounded by a line parallel to $\R v$, \ie it must have the form 
$$S_t^+=\{x\in \R^2:p_{v^\perp}(x) \geq t\}\quad \text{ or }\quad S_t^-= \{x\in \R^2 : p_{v^\perp}(x)\leq t\},$$
for some $t\in \R$. 
%Thus $$\conv(E) = \bigcap_{t: E\subset S_t^+} S_t^+ \cap \bigcap_{s:E\subset S_s^-} S_s^- \subset \bigcap_{t>b} S_t^+ \cap \bigcap_{s<a} S_t^-,$$
We claim that $\sup p_{v^\perp}(E) < \infty$. To see this, suppose for contradiction that $\sup p_{v^\perp}(E)=\infty$. Then $S_t^+$ does not contain $E$ when $t>t_0:=p_{v^\perp}(x_0)$, and $S_s^-$ does not contain $E$ for any $s\in \R$. Thus $\ol{\conv(E)}$ is an intersection of sets of the form $S_t^+$ with $t\leq t_0$, which implies that $\ol{\conv(E)}$ contains the half-plane $S_{t_0}^+$, from which follows that $\conv(E)$ contains the half-plane $\{x\in \R^2:p_{v^\perp}(x)>t_0\}$. The quasiconvexity of $E$ then implies that $E$ intersects any ball of radius greater than $R$ contained in $S_{t_0}^+$, from which follows that $\bd_\infty E$ contains a whole interval of length $\pi$, a contradiction.

By a similar argument $\inf p_{v^\perp}(E)>-\infty$, proving that $E\subset S:=p_{v^\perp}^{-1}((-M,M))$ for some $M>0$. To show that the two connected components $O^+$ and $O^-$ of $\R^2\sm S$ are contained in different connected components of $\R^2\sm E$, suppose that this is  not the case. 

\begin{figure}[ht]
\includegraphics[height=4cm]{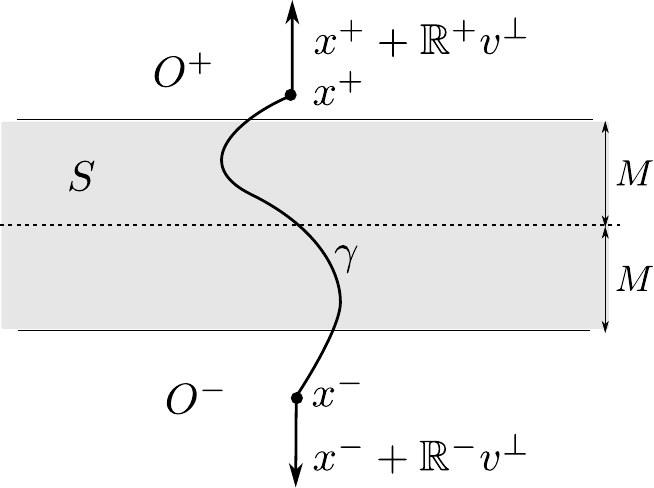}
\caption{}
\label{fig:geometric}
\end{figure}

Then, since $E$ is closed, there is an arc $\gamma\subset \R^2\sm E$ joining a point $x^-\in O^-$ to a point $x^+\in O^+$. We assume that $O^+$ is the component such that $p_{v^\perp}(x^+) > M$ (see Figure \ref{fig:geometric}). Consider the set $\Theta = (x^++\R^+ v^\perp)\cup [\gamma] \cup (x^-+\R^- v^\perp)$, which is disjoint from $E$. Clearly $\ol{U}_n\cap \Theta$ is nonempty for all $n\in \N$, because $\{-v,v\}\subset \bd_\infty \ol{U}_n$ and $\ol{U}_n$ is connected. Moreover, $\ol{U}_n\cap \Theta$ is bounded if $n$ is large enough (because $\ol{U}_n$ does not contain $v^\perp$ or $-v^\perp$), so $E\cap \Theta$ contains a nested intersection of compact nonempty sets, contradicting the fact that $E$ is disjoint from $\Theta$. 

This completes the proof of Theorem~\ref{th:geometric}. 
\end{proof}

\section{Poincar\'e recurrence in the lift for irrotational measures}\label{sec:poincare}

Throughout this section we assume that  $f\colon \T^2\to \T^2$ is a homeomorphism homotopic to the identity and $\hat{f}\colon \R^2\to \R^2$ is a lift of $f$. 

\subsection{Invariant measures and rotation vectors}

If $\mathcal{C}$ is any set of Borel probability measures on $\T^2$, we write 
$$\supp(\mathcal{C}) = \ol{\bigcup_{\mu\in \mathcal{C}} \supp(\mu)},$$
where $\supp(\mu)$ denotes the support of $\mu$. Equivalently, $x\in \supp(\mathcal{C})$ if every neighborhood of $x$ has positive $\mu$-measure for some $\mu\in \mc{C}$.

\begin{remark} Note that if $\mathcal{C}$ is convex, then it is not necessary to take the closure in the previous definition. In fact, if $x\in \supp(\mathcal{C})$, then for each $n\in \N$ there is $\mu_n\in \mathcal{C}$ such that $B_{1/n}(x)$ intersects $\supp(\mu_n)$, and therefore $\mu_n(B_{1/n}(x))>0$. Letting $\mu = \sum_{k=1}^\infty \frac{1}{2^n}\mu_n$, it follows that $\mu(B_{1/n}(x))>0$ for all $n\in \N$, so that $x\in \supp(\mu)$. The convexity implies that $\mu\in \mathcal{C}$.
\end{remark}

Denote by $\mathcal{M}(f)$ the set of all $f$-invariant Borel probability measures. 
For $\mu\in \mathcal{M}(f)$, the rotation vector of $\mu$ is defined as 
$$\rho_{\mu}(\hat{f}) = \int_{\T^2}\phi\, d\mu$$
where $\phi\colon \T^2\to \R^2$ is the ``displacement function'', defined for each $x\in \T^2$ as $\phi(x) = \hat{f}(\hat{x})-\hat{x}$ for some (hence any) $\hat{x}\in \pi^{-1}(x)$. 
If $\rho_\mu(\hat{f})=(0,0)$, we say that $\mu$ is an irrotational measure.

For any $v\in \R^2$, we denote by $\mathcal{M}_v(\hat{f})$ the set of all $\mu\in \mathcal{M}(f)$ such that $\rho_\mu(\hat{f})=v$. Note that $\mathcal{M}_v(\hat{f})$ is convex. Finally, we write $\mathcal{M}^e(f)$ and $\mathcal{M}^e_v(\hat{f})$ for the ergodic elements of $\mathcal{M}(f)$ and $\mathcal{M}_v(\hat{f})$, respectively.

Let us recall some classic facts:
\begin{proposition}[\cite{m-z}]\label{pro:m-z} The following properties hold:
\begin{itemize}
\item $\rho(\hat{f}) = \{\rho_\mu(\hat{f}): \mu\in \mathcal{M}(f) \}$
\item Any extremal point of $\rho(\hat{f})$ is the rotation vector of some ergodic measure.
\item If $\mu \in \mathcal{M}^e_v(\hat{f})$, then $\mu$-almost every $z\in \T^2$ is such that $(\hat{f}^n(\hat{z})-\hat{z})/n\to v$ as $n\to \infty$ for all $\hat{z}\in \pi^{-1}(z)$.
\end{itemize}
\end{proposition}

The next  proposition says that if $v\in \rho(\hat{f})$ is extremal, then the support of the set of measures with  rotation vector $v$ coincides with the support of the subset of all ergodic measures with the same rotation vector.

\begin{proposition}\label{pro:temmedidaergodica}
If $v$ is an extremal point of $\rho(\hat{f})$, and $B$ is a Borel set such that $\mu(B)>0$ for some $\mu\in \mathcal{M}_v(\hat{f})$, then there is an ergodic $\nu\in \mathcal{M}_v^e(\hat{f})$ such that $\nu(B)>0$. In particular,
$$\supp\left(\mathcal{M}_v(\hat{f})\right) = \supp\left(\mathcal{M}_v^e(\hat{f})\right).$$
\end{proposition} 

%  Then $x\in \mathcal{M}_v(\hat{f})$ and $\varepsilon>0$ be given. Then there exists a measure $\mu\in \mu_{v}(\hat{f})$ such that $f$ invariant ergodic measure $\nu$ such that $\nu(B_{\varepsilon}(x))>0$ and $\rho_{\nu}(\hat f)=(0,0).$

\proof 
%For any $x\in \supp(\mathcal{M}_v(\hat{f})$ and $\epsilon>0$, there is $\mu\in \mathcal{M}_v(\hat{f})$ such that $\mu(B_\epsilon(x))>0$. 
Since the map
$$\mathcal{M}(f)\ni \mu \mapsto \rho_\mu(\hat{f})$$
is affine, it follows from the ergodic decomposition  theorem that there is a probability measure $\til{\mu}$ defined on $\mathcal{M}^e(f)$ such that 
$$v=\rho_\mu(\hat{f}) = \int \rho_\nu(\hat{f})\, d\til{\mu}(\nu).$$
Since $v$ is extremal in $\rho(\hat{f})$, this easily implies that $\rho_\nu(\hat{f})=v$ for $\til{\mu}$-almost every $\nu$. From the ergodic decomposition we also have
$$\int \nu(B)\, d\til{\mu}(\nu) = \mu(B)>0,$$
thus $\nu(B)>0$ for a set of positive $\til{\mu}$-measure of elements $\nu\in \mathcal{M}^e_v(\hat{f})$. In particular, there exists one such $\nu$. The claim about the support of $\mathcal{M}_v(\hat{f})$ follows immediately considering $B=B_\epsilon(x)$ for $\epsilon>0$ and $x\in \mathcal{M}_v(\hat{f})$.
\endproof
%
%
%The set $\mathcal{M}_v(\hat{f})$  is convex and (weak-$*$) compact, and from the previous proposition, it is nonempty. Let $\nu$ be an extreme point of $\mathcal{M}_v(\hat{f})$. We claim that $\nu$ is also an extreme point of $\mathcal{M}(f)$. To see this, suppose that $\nu = t\nu_1 + (1-t)\nu_2$ for some pair $\nu_1, \nu_2 \in \mathcal{M}(f)$ and $t\in (0,1)$. Then 
%$$t\rho_{\nu_1}(\hat{f}) + (1-t)\rho_{\nu_2}(\hat{f})= v,$$ 
%and since $v$ is extremal it follows that $\rho(\nu_1) = \rho(\nu_2) = v$. But this means that $\nu_1, \nu_2 \in \mathcal{M}_v(\hat{f})$, and since $\nu$ is an extremal point of $\mathcal{M}_v(\hat{f})$ it follows that $\nu_1 = \nu_2 = \nu$. The Krein-Milman theorem implies that any element of $\mathcal{M}_v(\hat{f})$ is the closed convex hull of its extremal points, and the generalized Choquet theorem \cite{choquet} 
%
%Let $x\in \supp\left(\mathcal{M}_v(\hat{f})\right)$. Then there exists $\nu\in \mathcal{M}_v(\hat{f})$ such that $\nu(B_\epsilon(x))>0$. 

\subsection{Directional recurrence for ergodic irrotational measures}\label{sec:atkinson}

We will use several times the following lemma, which provides a sort of ``directional'' recurrence when an ergodic measure has nonzero rotation vector.

\begin{lemma}\label{lem:Atkinsonnossocaso} 
Let $\mu\in \mathcal{M}^e_{(0,0)}(\hat{f})$. Then, given $v_0\in \R^2_*$, there is a set $E_{v_0}\subset \T^2$ such that $\mu(E_{v_0})=1$ with the following property: for all $x\in E_{v_0}$ there is a sequence $(n_k)_{k\in \N}$ of integers such that $n_k\to \infty$ as $k\to \infty$ and, for any $\hat{x}\in \pi^{-1}(x)$, 
\begin{itemize}
\item $f^{n_k}(x) \to x$,
\item $(\hat{f}^{k}(\hat{x})-\hat{x})/k \to (0,0)$, and
\item $\big\langle \hat{f}^{n_k}(\hat{x})-\hat{x}\,;v_0\big\rangle\to 0$,
\end{itemize}
as $k\to \infty$. In particular, if $v_0\in \Z^2_*$ and $v_0$ is not a multiple of a different element of $\Z^2_*$, there is a sequence $(m_k)_{k\in \N}$ of integers such that $m_k/n_k\to 0$  and $\hat{f}^{n_k}(\hat{x})-\hat{x}-m_k v_0^\perp \to (0,0)$ as $k\to \infty$, for any $\hat{x}\in \pi^{-1}(x)$.
\end{lemma}  

To prove the lemma, let us recall a classical result from ergodic theory:
\begin{lemma}[Atkinson's Lemma, \cite{atkinson}]
Let $(X,\mathcal{B},\mu)$ be a non-atomic probability space, and let $T:X\to X$ be an ergodic automorphism. If $\phi:X\to\R$ belongs to $L_1(\mu)$ and $\int \phi d\mu=0.$ Then, for all $B\in\mathcal{B}$ and all $\epsilon>0,$
$$\mu\Bigg(\bigcup_{n\in\N}B\cap T^{-n}(B)\cap\Big\{x\in X : \Big\vert\sum_{i=0}^{n-1} \phi(T^i(x))\Big\vert<\epsilon\Big\}\Bigg)=\mu(B)$$
\end{lemma}

\begin{corollary}\label{coro:atkinson} Let $X$ be a separable metric space, $f\colon X\to X$ a homeomorphism, and $\mu$ an $f$-invariant ergodic non-atomic Borel probability measure. If $\phi\in L_1(\mu)$ is such that $\int \phi\, d\mu = 0$, then for $\mu$-almost every $x\in X$ there is an increasing sequence $(n_i)_{i\in \N}$ of integers such that $$f^{n_i}(x)\to x \quad \text{ and } \quad \sum_{k=0}^{n_i-1} \phi(f^k(x)) \to 0 \quad \text{ as } \quad i\to \infty.$$
\end{corollary}
\begin{proof}
It suffices to show that the set $E_i$ of all $x\in X$ for which there is $n\in \N$ such that $\big|\sum_{k=0}^{n-1}\phi(f^k(x))\big|<i^{-1}$ and $f^n(x)\in B_{1/i}(x)$ has full measure for each $i\in \N$. Suppose on the contrary that $\mu(X\sm E_i)>0$. Since $X$ is separable, $X$ is covered by countably many balls of radius $1/(2i)$. Thus, there is $x\in X$ such that $\mu(B_{1/(2i)}(x)\sm E_i)>0$. But Atkinson's Lemma applied to $B=B_{1/(2i)}(x)\sm E_i$ and $\epsilon = 1/i$ implies that there is $n\in \N$ and $x'\in B_{1/(2i)}(x)\sm E_i$ such that $f^n(x')\in B_{1/(2i)}(x)\sm E_i$ and $\big|\sum_{k=0}^{n-1}\phi(f^k(x'))\big|<i^{-1}$. In particular, $f^n(x')\in B_{1/i}(x')$, so by definition $x'\in E_i$, which is a contradiction.
\end{proof}

\begin{proof}[Proof of Lemma \ref{lem:Atkinsonnossocaso}]
Let $E = E_{v_0}$ be the set of all $x\in \T^2$ such that the three items of the lemma hold. 

Suppose first that $\mu$ is atomic. Then $\mu$ is supported in the orbit of some periodic point $p$. If $\hat{p}\in \pi^{-1}(p)$ and $n\in \N$ is the period of $p$, then $\hat{f}^n(\hat{p}) = \hat{p}+w$ for some $w\in \Z^2$. Since $\mu$ is ergodic and has rotation vector $(0,0)$, the fact that $\mu(\{p\})>0$ implies that $w=(0,0)$ (by Proposition \ref{pro:m-z}). Thus $\hat{f}^n(\hat{p}) = \hat{p}$, and it follows easily from this fact that $p\in E$. Since this can be done for any iterate of $p$, the orbit of $p$ is contained in $E$ and so $\mu(E)=1$ as we wanted.

Now suppose that $\mu$ is non-atomic. Then the last item of Proposition \ref{pro:m-z} implies that the second item of the lemma holds for $\mu$-almost every point, and applying Corollary \ref{coro:atkinson} to the displacement function in the direction $v_0$ defined by $\phi(x) = \big\langle \hat{f}(\hat{x})-\hat{x}\,;v_0\big\rangle$ for any $\hat{x}\in \pi^{-1}(x)$, we see that the first and third items of the lemma hold for $\mu$-almost every point as well. Thus $\mu(E)=1$, completing the proof.

Note that the final claim of the lemma is an immediate consequence of the three items and the fact that, setting $\epsilon_0 = \inf\{|\langle w\, ; v_0\rangle| : {w\in \Z^2\sm (\Z v_0^\perp)}\}$ (which is positive thanks to our assumption on $v_0$), if $\pi(\hat{y})\in B_\epsilon(x)$ for some $\epsilon<\epsilon_0$ and $\abs{\smash \langle \hat{y}-\hat{x}\,; v_0\rangle} < \epsilon_0$ then  $\hat{y}\in B_\epsilon(x+mv_0^\perp)$ for some $m\in \Z$. 
\end{proof}
%$\big\langle \hat{f}^n(\hat{z}) - \hat{z}\,;v_0\big\rangle<\epsilon$ for any $\hat{z}\in \pi^{-1}(z)$. In view of the last item of Proposition \ref{pro:m-z}, we may choose one such $z\in B_\epsilon(\pi(x))$ such that, in addition $(\hat{f}^k(\hat{z}) - \hat{z})/k\to (0,0)$ as $k\to \infty$ for any $\hat{z}\in \pi^{-1}(z)$. We may further assume, by Poincar\'e recurrence, that $z$ is recurrent, so that there is a sequence $(n_k)_{k\in \N}$ of integers with $n_k\to \infty$ such that $\hat{f}^{n_k}(z)\to z$ as $k\to \infty$. Letting $y\in \R^2$ be the element of $\pi^{-1}(z)$ in $B_\epsilon(x)$, we see that the first two properties of the lemma hold.  
%
%Furthermore, the fact that $\hat{f}^{n_k}(y)$ projects to a point near $z$ if $k$ is large enough implies that $\hat{f}^{n_k}(y) \in B_\epsilon$
%
%
%
%This follows directly from Proposition \ref{temmedidaergodica} and Atkinson's Lemma.

\subsection{The sets $\omega_v$}\label{sec:omegas}

Let us recall some sets and constructions from \cite{Transitivo,aneltransitivo} and their relevant properties. 

Given $v\in \R^{2}_*$,  denote by 
$$H_v^{+}:=\{u\in\R^2 : \langle u; v\rangle \ge 0 \}, \quad H_v^{-}:=\{u\in\R^2 : \langle u; v\rangle \le 0 \}$$
the closed half planes determined by $v$. Fix a homeomorphism $f\colon \T^2\to \T^2$ homotopic to the identity, and a lift $\hat{f}\colon \R^2\to \R^2$, and define the set $B_{v,\hat f}$ as the union of the unbounded connected components of 
$$\bigcap_{i=0}^{\infty}\hat f^{-i}(H^{+}_v),$$  
and $\omega_{v,\hat f}$ as the union of the unbounded connected components of 
$$\bigcap_{i=-\infty}^{\infty}\hat f^{i}(H^{+}_v).$$ 
Whenever the context is clear, we will simplify the notation and just write $B_v$ and  $\omega_v$ for these sets.

The following properties are easy consequences of the definitions (see \cite[\S2]{Transitivo}):
\begin{proposition}\label{pro:omegapro} The sets $B_v$ and $\omega_v$ are closed, and
\begin{enumerate}
\item[(1)] $\hat{f}(B_v)\subset B_v$, and $\hat{f}(\omega_v)=\omega_v$;
\item[(2)] if $u\in\Z^2$ and $\langle u; v\rangle \ge 0$, then $T_u(B_v) \subset B_v$ and $T_u(\omega_v)\subset \omega_v$;
\item[(3)] $\omega_v$ is non-separating, and its complement is simply connected.
\end{enumerate}
\end{proposition}

These sets are particularly useful whenever the origin belongs to $\rho(\hat{f})$. In this case, Lemma 3 of \cite{borto-tal} implies that $B_v$ and $B_{-v}$ are nonempty for any $v\in \Z^2_*$.
Also of interest is the case where the origin lies in the boundary of the rotation set. For these cases, we have the following result, which is a direct consequence of Lemma 3 of \cite{borto-tal} and Corollary 1 of \cite{Transitivo}:
\begin{proposition}\label{pr:omeganaovazio}
Suppose that $\rho(\hat f)\subset H^{+}_v$ for some $v\in \R^2_*$, and $(0,0)\in \rho(\hat f)$. Then both $\omega_v$ and $\omega_{-v}$ are non-empty. 
\end{proposition}

We will also need the following technical fact.

\begin{fact}\label{fact481} Let $\gamma\subset \R^2$ be an arc joining $x\in \R^2$ to $x+v^\perp$ and $\Gamma = \bigcup_{k\in \Z} [\gamma] + kv^\perp$. If $\omega_v\cap \Gamma=\emptyset$ and $W_+$ is the connected component of $\R^2\sm \Gamma$ such that $\sup \pr_v W_+ = \infty$, then $\omega_v \subset W_+$.
\end{fact}
\begin{proof} Let $W_-$ and $W_+$ be the two unbounded connected components of $\R^2\sm \Gamma$, being $W_+$ the one that satisfies $\sup \pr_v W_+=\infty$. Assume $\omega_v\neq \emptyset$. Since $\omega_v+v \subset \omega_v$, we know that $\omega_v$ intersects $W_+$. Suppose for a contradiction that $\omega_v$ intersects a connected component of $\R^2\sm \Gamma$ other than $W_+$. Let $\theta$ be a connected component of $\omega_v$ not contained in $W_+$.
Since $\theta$ is an unbounded subset of $\R^2\sm \Gamma$, it follows that $\theta \subset W_-$. 

Let $p\in \inter H_v^-\cap W_-$. Since $\R^2\sm \omega_v$ is connected, there is an arc $\sigma$ in $\R^2\sm \omega_v$ joining $p$ to some point $q\in \Gamma$. The fact that $\theta$ is unbounded and contained $H_v^+\cap W_-$ implies that there is $m\in \Z$ such that $\theta+mv^\perp$ intersects $\sigma$, which is a contradiction.
\end{proof}

\subsection{Poincar\'e recurrence on the lift: Theorems \ref{th:naoerrante-ext} and \ref{th:irrotational-nw}}\label{sec:poincare-proof}

Theorem \ref{th:irrotational-nw} is an immediate corollary of Theorem \ref{th:naoerrante-ext}. The latter, in turn, follows from the next result, the proof of which is the focus of the remainder of this section.

%\begin{theorem}\label{th:naoerrante-ext} 
%Suppose that $(0,0)$ is an extremal point of $\rho(\hat{f})$, and $\rho(\hat{f})$ is contained in a closed half-plane $H^+_v$ for some $v\in \Z^2_*$. Then, for any $\mu\in \mathcal{M}_{(0,0)}(\hat{f})$, the set of $\hat{f}$-recurrent points projects to a set of full $\mu$-measure. 
%Moreover, the nonwandering set of $\hat{f}$ contains $\pi^{-1}(\supp(\mathcal{M}_{(0,0)}(\hat{f})))$.
%\end{theorem}

\begin{theorem}\label{th:poincare-omega}
Let $\hat{f}\colon \R^2\to \R^2$ be a lift of a homeomorphism $f$ of $\T^2$ homotopic to the identity. Suppose that $\omega_v\neq \emptyset\neq \omega_{-v}$ for some $v\in \Z^2_*$. Then, for any $\mu\in \mathcal{M}_{(0,0)}^e(\hat{f})$, the set of $\hat{f}$-recurrent points projects to a set of full $\mu$-measure. 
Moreover, the nonwandering set of $\hat{f}$ contains $\pi^{-1}(\supp(\mathcal{M}_{(0,0)}^e(\hat{f})))$.
\end{theorem} 	

Before moving to the proof, let us show how Theorem \ref{th:naoerrante-ext} follows from the above. 

\begin{proof}[Proof of Theorem \ref{th:naoerrante-ext}]
If $(0,0)$ is an extremal point of $\rho(\hat{f})$ and $\rho(\hat{f})\subset H^+_v$ where $v\in \Z^2_*$, then by Proposition \ref{pr:omeganaovazio} we know that $\omega_v$ and $\omega_{-v}$ are both nonempty. Let $\mathrm{Rec}(\hat{f})$ denote the set of $\hat{f}$-recurrent points, and suppose for contradiction that there is $\mu\in \mathcal{M}_{(0,0)}(\hat{f})$ such that $\mu(\T^2\sm \pi(\mathrm{Rec}(\hat{f})))>0$.  Then, by Proposition \ref{pro:temmedidaergodica}, there is $\nu\in \mathcal{M}^e_{(0,0)}(\hat{f})$ such that $\nu(\T^2\sm \pi(\mathrm{Rec}(\hat{f})))>0$, which contradicts Theorem \ref{th:poincare-omega}.
Thus we conclude that $\pi(\mathrm{Rec}(\hat{f}))$ has full $\mu$-measure, proving Theorem \ref{th:naoerrante-ext}.
\end{proof}

The proof of Theorem \ref{th:poincare-omega} will be divided into several independent propositions, some of which will be useful for other purposes.

\setcounter{claim}{0}
\subsection{The case where $\pi(\omega_v)$ and $\pi(\omega_{-v})$ are disjoint}\label{sec:omeganaotoca}
In this subsection, our only assumption is that $\pi(\omega_{v})$ and $\pi(\omega_{-v})$ are nonempty and disjoint (where $v\in \Z^2_*$).
%We will denote by $T\colon \R^2\to \R^2$ the translation $T_{v^\perp}:z\mapsto z+v^\perp$.
Fix $p_1,p_2\in \Z$, and write 
$$\omega_- = \omega_{-v}+p_1v, \quad \omega_+=\omega_{v}+p_2v.$$
Note that from the definitions, both sets are closed, $\hat{f}$-invariant and non-separating. Moreover, $T_{v^\perp}^k(\omega_-)=\omega_-$ and $T_{v^\perp}^k(\omega_+)=\omega_+$ for all $k\in \Z$.

Let $\ol{\ell}\subset \R^2$ be the image of a compact arc joining $x_1 \in \omega_-$ to $x_2\in \omega_+$, such that $\ell = \ol{\ell}\sm \{x_1,x_2\}$ is disjoint form $\omega_-\cup \omega_+$, and assume further that $\ell$ is disjoint from $T_{v^\perp}^k(\ell)$ for all $k\in \Z_*$ (the latter holds, for instance, if $\diam(\ell)<1$).
\begin{proposition}\label{pro:omeganaotoca} Suppose $\hat{f}(\ell)\cap \ell=\emptyset$, and let $F = \omega_+\cup \ell\cup \omega_-$. Also assume that $\hat{f}$ has a fixed point. Then
\begin{enumerate}
\item[(1)] $\R^2\sm F$ has exactly two connected components $\Omega_1$ and $\Omega_2$ such that $\ell\subset \bd \Omega_2\cap \bd \Omega_1$, $T_{v^\perp}(\Omega_2)\subset \Omega_2$ and $T_{v^\perp}^{-1}(\Omega_1)\subset \Omega_1$. 
\item[(2)] For each $z\in \R^2\sm (\omega_+\cup\omega_-)$ there is $n_0\in \Z$ such that $T_{v^\perp}^n(z)\in \Omega_2$ if $n>n_0$ and $T_{v^\perp}^n(z)\in \Omega_1$ if $n<n_0$.
\item[(3)] Both $\Omega_2$ and $\Omega_1$ contain some fixed point of $\hat{f}$.
\item[(4)] Either $f(\Omega_2)\subset \Omega_2$ or $f^{-1}(\Omega_2)\subset\Omega_2$. 
\item[(5)] If $f(\Omega_2)\subset \Omega_2$, then for each $x\in \R^2\sm (\omega_+\cup \omega_-)$ there exists $\delta>0$ and $m_0\in \Z$ such that $f^k(B_\delta(x))\cap T_{v^\perp}^n(B_\delta(x))=\emptyset$ for all integers $k>0$ and $n\leq m_0$ (and in the case that $f^{-1}(\Omega_2)\subset \Omega_2$, a similar property holds with $k<0$).
\item[(6)] There is a wandering open set $W$ containing $\ell$.
\end{enumerate}
\end{proposition}

\begin{proof}
Most of the claims in this propositions are contained in \cite{Transitivo}, although some not explicitly. We include them here for the sake of completeness. We will assume $v=(1,0)$, so $T_{v^\perp}$ is the translation $(x,y)\mapsto (x,y+1)$. The same proof works for any $v\in \Z^2_*$ after a change of coordinates in $\SL(2,\Z)$ (or after minimal modifications). To simplify the notation, we write $T=T_{v^\perp}$. 

First, since $\omega_-$ and $\omega^+$ are disjoint, non-separating, and have only unbounded connected components, the set $\omega_-\cup \omega_+$ is also non-separating (see \cite[Proposition 8]{Transitivo}). Moreover, since one may choose a neighborhood of any point $x\in \ell$ that is disjoint from $\omega_-\cup \omega_+$ and locally separated by $\ell$ into exactly two connected components, one easily concludes that $\R^2\sm F$ has exactly two connected components (see \cite[Lemma 8]{Transitivo}).
%
%From the definition of the sets $\omega_v$ one has that $\omega_v\cup\{\infty\}$ is a non-separating continuum in the one-point compactification $\R^2\sqcup\{\infty\}$ (equivalently, it is a nested intersection of closed topological disks). In particular, $\omega_-\cup\{\infty\}$ and $\omega_+\cup\{\infty\}$ are non-separating continua which intersect at a unique point. The union of two non-separating continua of the sphere which intersect at a unique point is still non-separating (see for instance \cite{newman-book}). Therefore, $\omega_+\cup\omega_-\cup\{\infty\}$ is a non-separating continuum in $\R^2\sqcup\{\infty\}$, and this implies that $\omega_+\cup\omega_-$ is a non-separating subset of $\R^2$. 

Since $\omega_-\cup \omega_+$ is non-separating, given $z\in \R^2\sm (\omega_+\cup\omega_-)$ there is a compact arc $\gamma$ disjoint from $\omega_+\cup\omega_-$ joining $z$ to $T(z)$. %We may choose $\gamma$ such that the concatenation 
Let $\Gamma = \bigcup_{n\in \Z} T^n[\gamma]$. Note $\R^2\sm \Gamma$ has exactly two unbounded connected components: one unbounded to the left, which we call $W_-$  and one unbounded to the right which we denote $W_+$.
It follows from Fact \ref{fact481} that $\omega_-\subset W_-$ and $\omega_+\subset W_+$, and so $[\Gamma]\cap \ell\neq \emptyset$.

Let $n_1\in \Z$ be the smallest integer such that $T^{n_1}[\gamma]$ intersects $\ell$, and let $n_2\geq n_1$ be the largest integer with the same property. 
Then the sets $\Gamma^+ = \cup_{n> n_2} T^n[\gamma]$ and $\Gamma^- = \cup_{n<n_1} T^{n}[\gamma]$ are both contained in $\R^2\sm F$. We claim that $\Gamma^-$ and $\Gamma^+$ are in different connected components of $\R^2\sm F$. Indeed, suppose otherwise. Then we can find a compact arc $\sigma$ in $\R^2\sm F$ joining a point of $\Gamma^-$ to a point of $\Gamma^+$. Let $\Theta = \Gamma^+\cup [\sigma] \cup \Gamma^-$. Then $\R^2\sm \Theta$ has two unbounded connected components $W_-'$ and $W_+'$, which are unbounded to the left and to the right, respectively. 
Let $\Theta' = \Gamma\cup \Theta$. Again, $\R^2\sm \Theta'$ has exactly two unbounded connected components, $W_-''\subset W_-\cap W_-'$ and $W_+''\subset W_+\cap W_+'$. Since $\omega_+$ is disjoint from $\Theta'$ and its components are unbounded, each connected component of $\omega_+$ is contained in either $W_-''$ or $W_+''$, and since $\omega_+\subset W_+$ we conclude that $\omega_+\subset W_+''$. Therefore, $\omega_+\subset W_+'$. Likewise, $\omega_-\subset W_-'$. Since $\ell$ joins $\omega_-$ to $\omega_+$, it must therefore intersect $\Theta$, which is a contradiction.

Thus the sets $\Gamma^-$ and $\Gamma^+$ are contained in different connected components of $\R^2\sm F$. We remark that from the definitions, $T(\Gamma^-)\cap \Gamma^-\neq \emptyset$ and $T(\Gamma^+)\cap \Gamma^+\neq \emptyset$. In particular, each connected component of $\R^2\sm F$ intersects its image by $T$.

\begin{figure}[ht]
\includegraphics[height=4cm]{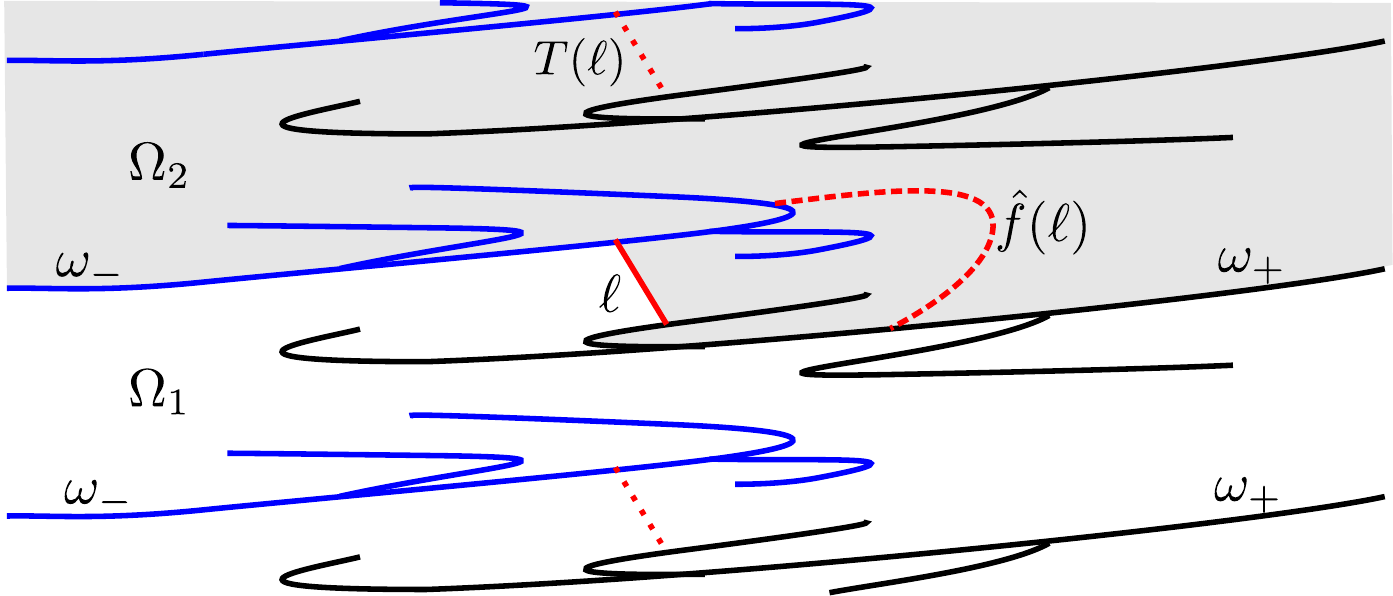}
\caption{The sets $\Omega_1$, $\Omega_2$ and $F$.}
\label{fig:omega-ell}
\end{figure}

Note that $T(\ell)$ is disjoint from $\ell$ and $T(F) = \omega_-\cup \omega_+\cup T(\ell)\subset F\cup T(\ell)$. Letting $\Omega_2$ be the connected component of $\R^2\sm F$ containing $T(\ell)$, we have that $T(F)$ is disjoint from $\Omega_1$. This implies that one of the two connected components of $\R^2\sm T(F)$ contains $\Omega_1$, \ie either $\Omega_1\subset T(\Omega_1)$ or $\Omega_1\subset T(\Omega_2)$ (see Figure \ref{fig:omega-ell}). But the latter case implies that $T(\Omega_1)$ is disjoint from $\Omega_1$, contradicting the fact that $\Omega_1$ contains either $\Gamma^+$ or $\Gamma^-$. Thus the only possible case is $\Omega_1\subset T(\Omega_1)$. 
 It follows from this fact that $T(\Omega_2)$ is disjoint from $\Omega_1$. Note that this also implies that $T(\Omega_2)$ is disjoint from $\ell$, because $\ell\subset \bd \Omega_1$ and $\Omega_2$ is open. Since $T(\Omega_2)$ is disjoint from $\omega_-\cup \omega_+$ as well, we conclude that $T(\Omega_2)$ is disjoint from $\Omega_1\cup F$, so that $T(\Omega_2)\subset \Omega_2$. This proves (1).

To prove (2), recall that $z$ in the previous argument was assumed to be any point in $\R^2\sm (\omega_-\cup\omega_+)$. Since one of the sets $\Gamma^+$ or $\Gamma^-$ is contained in $\Omega_2$, and we showed that $T(\Omega_2)\subset \Omega_2$, we see that the only possibility is $\Gamma^+\subset \Omega_2$, and $\Gamma^-\subset \Omega_1$. In particular, $T^n(z)$ belongs to $\Omega_2$ if $n> n_2$ and to $\Omega_1$ if $n<n_1$. Let $n_0$ be the smallest integer such that $T^{n_0}(z)\notin \Omega_1$. Since $z\in \R^2\sm (\omega_-\cup\omega_+)$, we have that $T^{n_0}(z)\in \Omega_2\cup \ell$. Since $T(\Omega_2\cup \ell)\subset \Omega_2$, we see that $T^{n}(z)\in \Omega_2$ for any $n>n_0$. The definition of $n_0$ also implies that $T^n(z)\in \Omega_1$ for $n<n_0$, completing the proof of (2).

To prove (3), recall that we assumed that $\hat{f}$ has some fixed point $z_0$. Moreover, since we are assuming that $\pi(\omega_+)=\pi(\omega_{(1,0)})$ is disjoint from $\pi(\omega_-)= \pi(\omega_{(-1,0)})$, we have that either $\pi(z_0)\notin \pi(\omega_-)$ or $\pi(z_0)\notin \pi(\omega_+)$. Suppose that $\pi(z_0)\notin \pi(\omega_-)$ (the other case is analogous). Then $z_0+w\notin \omega_-$ for any $w\in \Z^2$. Since $\omega_{(1,0)}\subset H_{(1,0)}^+$, the definition of $\omega_+$ implies that the first coordinate of a point of $\omega_+$ is at least $p_2$. In particular, if we choose $w\in \Z^2$ such that the first coordinate of $z_1=z_0+w$ is smaller than $p_2$ we have that $z_1\notin \omega^+\cup \omega^-$. Part (2) of the theorem implies that there is $n_0$ such that $T^n(z_1)$ belongs to $\Omega_2$ if $n>n_0$ and to $\Omega_1$ if $n<n_0$. Since $T^n(z_1)$ is a fixed point of $\hat{f}$ for any $n\in \Z$, we conclude that there are fixed points in both $\Omega_2$ and $\Omega_1$.

To prove (4), note that since $\hat{f}(\ell)$ is disjoint from $\ell$ and both are disjoint from $\omega_-\cup \omega_+$, we have that $\hat{f}(\ell)\subset \R^2\sm F$. Thus, $\hat{f}(\ell)\subset \Omega_i$ for some $i\in \{1,2\}$. Assume without loss of generality that $\hat{f}(\ell)\subset \Omega_2$ (the same argument applies to $\hat{f}^{-1}$ otherwise).
The facts that $\hat{f}(\omega_+\cup\omega_-)=\omega_+\cup \omega_-$ and $\hat{f}(\ell)\subset \Omega_2$ imply that $\hat{f}(F)$ is disjoint from $\Omega_1$, so that $\Omega_1$ is contained in one of the two connected components of $\R^2\sm \hat{f}(F)$, \ie either $\Omega_1\subset \hat{f}(\Omega_2)$ or $\Omega_1\subset \hat{f}(\Omega_1)$. But if $\Omega_1\subset \hat{f}(\Omega_2)$, then $\hat{f}^{-1}(\Omega_1)\subset \Omega_2$ and in particular $\hat{f}^{-1}(\Omega_1)$ is disjoint from $\Omega_1$. This is not possible because, by (3), $\Omega_1$ contains a fixed point of $\hat{f}$. Thus $\Omega_1\subset \hat{f}(\Omega_1)$, and it follows that $\hat{f}(\Omega_2)$ is disjoint from $\Omega_1$. Since $\hat{f}(\Omega_2)$ is also disjoint from $\omega_+\cup \omega_-$, we see that $\hat{f}(\Omega_2)\subset \Omega_2\cup \ell$. Since $\ell\subset \bd \Omega_1$ and $\hat{f}(\Omega_2)$ is open, we also have that $\hat{f}(\Omega_2)$ is disjoint from $\ell$, so $\hat{f}(\Omega_2)\subset \Omega_2$ as claimed.

For part (5), let $x\in \R^2\sm (\omega_-\cup \omega_+)$ and let us first prove the following fact: there is $\delta>0$ and $k_1\leq k_2\in \Z$ such that $T^j(B_\delta(x))$ is contained in $\Omega_2$ if $j> k_2$ and in $\Omega_1$ if $j<k_1$.
To see this, note that since $\omega_+\cup \omega_-$ is closed and does not contain $x$, there is $\delta>0$ such that $B_\delta(x)$ is disjoint from $\omega_+\cup \omega_-$. Recalling that $\omega_+\cup \omega_-$ is $T$-invariant, it follows that $T^n(B_\delta(x))\cap (\omega_+\cup \omega_-)=\emptyset$ for all $n\in \Z$. Since $\ol{\ell}$ is bounded, there is $r_0>0$ such that $T^j(B_\delta(x))$ is disjoint from $\ell$ if $\abs{j}>r_0$. Thus $T^j(B_\delta(x))$ is disjoint from $F$ for all $j\in \Z$ with $\abs{j}>r_0$. By part (2) (using $z=x$) there is $n_0$ such that $T^j(x)$ lies in $\Omega_2$ if $j>n_0$ and in $\Omega_1$ if $j<n_0$. Choosing $k_2=\max\{r_0, n_0\}$ and $k_1 =\min\{-r_0, n_0\}$, we have that if $j>k_2$ then $T^j(x)\in \Omega_2$. Since $j\geq r_0$, we also have $T^j(B_\delta(x))\subset \R^2\sm F$, and since it intersects $\Omega_2$ we conclude that $T^j(B_\delta(x))\subset \Omega_2$ for all $j>k_2$. Similarly, for $j<k_1$ we have that $T^j(B_\delta(x))\subset \Omega_1$, as claimed.

To finish the proof of (5), note that $T^{k_2+1}(B_\delta(x))\subset \Omega_2$, and so $\hat{f}^k(T^{k_2+1}(B_\delta(x)))\subset \Omega_2$ for all $k>0$. Setting $m_0 = k_1-k_2-2$, we see that, if $k>0$ and $n\leq m_0$, then 
$$T^{k_2+1}(\hat{f}^k(B_\delta(x))\cap T^n(B_\delta(x)))\subset \hat{f}^k(T^{k_2+1}(B_\delta(x)))\cap T^{n+k_2+1}(B_\delta(x)) \subset \Omega_2\cap \Omega_1 = \emptyset,$$
so that $\hat{f}^k(B_\delta(x))\cap T^n(B_\delta(x))=\emptyset$ as required.

Finally, to prove (6) assume $\hat{f}(\ell)\subset \Omega_2$ (otherwise use $\hat{f}^{-1}$ instead of $\hat{f}$). We have from (4) that $\hat{f}(\Omega_2)\subset \Omega_2$. Since $\ell$ and $\hat{f}(\ell)$ are disjoint closed subsets of $\R^2\sm (\omega_-\cup \omega_+)$ (which is an open subset of $\R^2$ and a normal topological space) there are open sets $W_0\supset \ell$ and $W_1\supset \hat{f}(\ell)$ in $\R^2\sm (\omega_-\cup \omega_+)$ such that $W_0\cap W_1=\emptyset$. Note that $\ell\subset \R^2\sm \cl{\hat{f}(\Omega_2)}$, since $$\cl{\hat{f}(\Omega_2)}=\hat{f}(\cl(\Omega_2))\subset \hat{f}(\Omega_2\cup F) \subset \Omega_2\cup \omega_-\cup \omega_+\cup \hat{f}(\ell),$$ which is disjoint from $\ell$. Thus,
$$W = \hat{f}^{-1}(W_1)\cap W_0\cap \hat{f}^{-1}(\Omega_2)\sm \cl{\hat{f}(\Omega_2)}$$ is an open neighborhood of $\ell$ in $\R^2$.
Note that $\hat{f}^k(W)\subset \hat{f}^{k-1}(\Omega_2) \subset \hat{f}(\hat{f}^{k-2}(\Omega_2))\subset \hat{f}(\Omega_2)$ if $k\geq 2$, and so $\hat{f}^k(W)\cap W=\emptyset$ whenever $k\geq 2$. In addition $\hat{f}(W)\cap W\subset W_1\cap W_0= \emptyset$, so $\hat{f}^k(W)\cap W=\emptyset$ for all $k\in \N$ as we wanted.
\end{proof}

\subsection{When $\pi(\omega_v)$ intersects $\pi(\omega_{-v})$}

In this section we assume that $\pi(\omega_v)\cap \pi(\omega_{-v})\neq \emptyset$. In this case, there exist $p_1, p_2\in \Z$ such that $(\omega_v + p_1v)\cap (\omega_{-v}+p_2v) \neq \emptyset$. 
%Again, we denote by $T$ the translation $T_{v^\perp}: z\mapsto z+v^\perp$.

\begin{proposition}\label{pro:omegasetocam} Suppose $(\omega_v + p_1v)\cap (\omega_{-v}+p_2v) \neq \emptyset$ for some $p_1,p_2\in \Z$, and let
$$F=(\omega_v + p_1v)\cup (\omega_{-v}+p_2v).$$
Given $z\in \R^2\sm F$, denote by $O(z)$ the connected component of $\R^2\sm F$ containing $z$. Then the following properties hold:
\begin{enumerate}
\item[(1)] $O(z)$ is an open topological disk and $O(z)\cap T_{v^\perp}^k(O(z))=\emptyset$ for any $k\in \Z_*$.
\item[(2)] If $\mu \in \mathcal{M}^e_{(0,0)}(\hat{f})$, then for $\mu$-almost every $x\in \T^2$ such that $x\notin \pi(\omega_v)\cap \pi(\omega_{-v})$, any $\hat{x}\in \pi^{-1}(x)$ is $\hat{f}$-recurrent.
\item[(3)] If $\hat{x}\in \R^2$ is $\hat{f}$-nonwandering and  $\pi(\hat{x}) \notin \pi(\omega_v)\cap \pi(\omega_{-v})$, then there is $\epsilon>0$ such that $U_\epsilon(\hat{x}) := U_\epsilon(\hat{x},\hat{f})$ is disjoint from $T_{v^\perp}^k(U_\epsilon(\hat{x}))$ for any $k\in \Z_*$ (recall the definition of $U_\epsilon$ from \S\ref{sec:uepsilon}).
\end{enumerate}
\end{proposition}
\begin{proof} Again, we assume for simplicity that $v=(1,0)$, so $T_{v^\perp}= T\colon (x,y)\mapsto (x,y+1)$.
To prove (1), write $O=O(z)$. That $O$ is a topological disk follows from the fact that the connected components of $F$ are unbounded. Assume by contradiction that $O\cap T^n(O) \neq \emptyset$ for some integer $n\neq 0$. Then by Proposition \ref{pro:translacaodisjunta}, $O\cap T(O)\neq \emptyset$. Since $F = T(F)$, the set $T(O)$ is also a connected component of $\R^2\sm F$, and so $T(O)=O$. Let $\alpha$ be an arc in $O$ joining a point $x\in O$ to $T(x)$, and let $\beta=\bigcup_{k\in\Z}T^{k}[\alpha]\subset O$. By Fact \ref{fact481}, $\omega_v+p_1v$ and $\omega_v+p_2v$ are contained in different connected components of the complement of $\beta$, contradicting the fact that they intersect.

To prove (2), note that we may assume that $v$ is not a multiple of an element of $\Z^2_*$ (by choosing it appropriately), and observe that if $x$ is chosen in the set $E_{v}$ from Lemma \ref{lem:Atkinsonnossocaso} then there are sequences $(m_k)_{k\in \N}$ and $(n_k)_{k\in \N}$ such that $\hat{f}^{n_k}(\hat{x}) - T^{m_k}(\hat{x})\to (0,0)$ and $m_k/n_k\to 0$ as $k\to \infty$ for any $\hat{x}\in \pi^{-1}(x)$. We will show that some element of $\pi^{-1}(x)$ is $\hat{f}$-recurrent, which implies that all elements of $\pi^{-1}(x)$ are as well.

Let us show that if $x\notin \pi(\omega_v)\cap \pi(\omega_{-v})$, then $\hat{x}$ may be chosen in $\pi^{-1}(x)\cap (\R^2\sm F)$.  Suppose that $x\in \T^2\sm\pi(\omega_v)$ (the other case is analogous). Then, if $\hat{x}_0$ is any element of $\pi^{-1}(x)$, we have that $\hat{x}_0+w \notin \omega_v+p_1v$ for any $w\in \Z^2$. We may choose $w$ such that $\hat{x}:=\hat{x}_0+w$ belongs to the open half-plane $\R^2\sm (H_v^- + p_2v)$. Since $\omega_{-v}+p_2v\subset H_v^-+p_2v$, it follows that $\hat{x}\notin \omega_{-v}+p_2v$, and so $\hat{x}\notin F$, as required.

Thus we assume that $\hat{x}\in \R^2\sm F$. Fix $\delta>0$ such that $B_\delta(\hat{x})\subset O:=O(\hat{x})$. Then there is $k_0\in \N$ such that, for $k\geq k_0$,  
$$\hat{f}^{n_k}(\hat{x})\in B_\delta(T^{m_k}(\hat{x}))\subset T^{m_k}(O).$$
In particular, there is a smallest integer $n\in \N$ such that $\hat{f}^n(O)$ intersects $T^m(O)$ for some $m\in \Z$. Since $T^m(O)$ is also a connected component of $\R^2\sm F$ and $\hat{f}$ permutes the connected components of $\R^2\sm F$, it follows that $\hat{f}^n(O)=T^m(O)$.
Assume $m\neq 0$. By the minimality in the choice of $n$, if $\hat{f}^{n'}(O)$ intersects $T^{m'}(O)$ for some $n'\in \N$ and $m'\in \Z$, then $n'=ln$ and $m'=lm$. 
Therefore, if $k\geq k_0$, then $n_k = l_kn$ and $m_k = l_km$. But then $m/n = m_k/n_k\to 0$ as $k\to \infty$, and we conclude that $m=0$, contradicting our assumption. Thus $m=0$, and it follows that  $m_k = 0$ for all $k\geq k_0$. This means that $\hat{f}^{n_k}(\hat{x})- \hat{x} \to (0,0)$ as $k\to \infty$, proving that $\hat{x}$ is recurrent.

Finally, to prove (3), note that if the claim is true for some $\Z^2$-translation of $\hat{x}$ then it is also true for $\hat{x}$. Using the fact that $\pi(\hat{x})\notin \pi(\omega_+)\cap\pi(\omega_-)$ we may assume that $\hat{x}\in \R^2\sm F$ (as done in the previous item, replacing $\hat{x}$ by an appropriate $\Z^2$-translation of $\hat{x}$). By (1) we know that $O(\hat{x})$ is disjoint from $T^k(O(\hat{x}))$ for any $k\in \Z_*$. Note that the connected components of $\R^2\sm F$ are permuted by $\hat{f}$, and since $\hat{x}$ is nonwandering we have that $\hat{f}^n(O(\hat{x})) = O(\hat{x})$ for some $n\in \N$, which we choose minimal with that property. Fix $\epsilon>0$ such that $B_\epsilon(\hat{x})\subset O(\hat{x})$. Then it follows from the definitions in \S\ref{sec:uepsilon} that $U_\epsilon'(\hat{x})\subset O(\hat{x})$, and so $U_\epsilon'(\hat{x})$ is disjoint from $T^k(U_\epsilon'(\hat{x}))$ for any $k\in \Z_*$. By Proposition \ref{pro:fill-ine} we conclude that $U_\epsilon(\hat{x})$ is disjoint from $T^k(U_\epsilon(\hat{x}))$ for any $k\in \Z_*$.  Finally, since $U_\epsilon(\hat{x}+w) = U_\epsilon(\hat{x})+w$ for any $w\in \Z^2$, we conclude (3).
\end{proof}

%We are ready to conclude the proof. The fact that $\hat{x}\in \Lambda$ and the definitions of $\omega_-$ and $\omega_+$ imply that the first coordinate $\hat{f}^n(\hat{x})$ is uniformly bounded: 
%$$\abs{\hat{f}^n(\hat{x})_1} \leq M:= \max\{\abs{p_1}, \abs{p_2}\} \quad \text{ for all } n\in \Z.$$
%Recall from the beginning of this proof that, in our choice of $\hat{x}$, we also assumed that the thesis of Lemma \ref{lem:Atkinsonnossocaso} holds using the direction $v_0=(0,1)$. In particular, there are sequences of integers $(m_k')_{n\in \N}$ and $(n_k')_{n\in \N}$ with $n_k\to \infty$ and $$\norm{\smash{\hat{f}^{n_k'}(\hat{x})-T_{(1,0)}^{m_k}}}\to 0$$ as $k\to \infty$. In particular, if $k$ is large enough then $\hat{f}^{n_k'}(\hat{x})\in B_{\epsilon_0}(\hat{x}+(m_k',0))$. Since $\epsilon_0<1/4$, this implies that $\abs{m_k'}\leq M+1$, and so there are finitely many possible values of $m_k'$. In particular, we may find $m\in \N$ and arbitrarily large numbers $i,j\in \N$ such that $m_i'=m_j'=m$ and $n_i'<n_j'$. This means that $\hat{f}^{n_j'-n_i'}(B_{\epsilon_0}(\hat{x}+(m,0)))$ intersects $B_{\epsilon_0}(\hat{x}+(m,0))$, contradicting the fact that $B_{\epsilon_0}(\hat{x}+(m,0))$ is wandering for $\hat{f}$ (because $B_{\epsilon_0}(\hat{x})$ is wandering for $\hat{f}$).
%

\subsection{A lemma on the support of irrotational ergodic measures}

\begin{lemma}\label{lem:tresopcoes} Let $\mu\in \mathcal{M}^e_{(0,0)}(\hat{f})$, and let $v\in \Z^2_*$ be such that $\omega_v\neq \emptyset$. Then one of the following holds:
\begin{enumerate}
\item[(1)] $\supp(\mu)\subset \ol{\pi(\omega_v)}$, or 
\item[(2)] for $\mu$-almost every $x$, any $\hat{x}\in \pi^{-1}(x)$ is $\hat{f}$-recurrent and belongs to some $\hat{f}$-periodic open connected set $\hat{U}$ such that $\pi(\hat{U})$ is not fully essential.
\end{enumerate}
\end{lemma}

\proof
Assume that case (1) does not hold. Then $\T^2\sm  \ol{\pi(\omega_v)}$ is an invariant set of positive $\mu$-measure, so by the ergodicity it has measure $1$. In particular, if $E_{v_0}$ is the set from the statement of Lemma \ref{lem:Atkinsonnossocaso}, the set
$$E=(\T^2\sm \ol{\pi(\omega_v)})\cap \bigcap_{v_0\in \Z^2_*} E_{v_0},$$
is such that $\mu(E)=1$, and the definition implies that every $x\in E$ is $f$-recurrent.
Fix $x\in E$, and note that since $x\notin \ol{\pi(\omega_v)}$, we may choose $\delta>0$ so small that $B_{\delta}(x)$ is disjoint from $\pi(\omega_v)$.  As $\pi(\omega_v)$ is an invariant set, this implies that $\mc{O}(B_{\delta}(x))=\bigcup_{i=-\infty}^{\infty} f^{i}(B_{\delta}(x))$ is disjoint from $\pi(\omega_v)$. Let $U$ be the connected component of $\mc{O}(B_{\delta}(x))$ containing $x$. 

Let us observe that $U$ cannot be fully essential. Indeed, if $U$ is fully essential, since it is open it follows that all connected components of $\pi^{-1}(\T^2\sm U)$ are bounded, contradicting the fact that (by our assumptions) $\omega_v$ is nonempty and contained in $\pi^{-1}(\T^2\sm U)$ (because $\omega_v$ has only unbounded connected components).

Since $x$ is recurrent, there is a smallest $n\in \N$ such that $f^n(U)$ intersects $U$. The fact that $f$ permutes the connected components of $O(B_{\delta}(x))$ implies that $f^n(U)=U$. Note that this means that if $f^i(U)\cap U\neq \emptyset$ for some $i\in \Z$ then $i=ln$ for some $l\in \Z$. % s/j_0/n

%We remark that, since $\pi^{-1}(U)$ is invariant by $\Z^2$-translations, if $\hat{U}\cap T_u(\hat{U})$ is not empty for some $u\in\Z^2$, then $\hat{U}=T_u(U)$. 

Let $\hat{x}\in \pi^{-1}(x)$, and let $\hat{U}$ be the connected component of $\pi^{-1}(U)$ that contains $\hat{x}$. Then there is $w\in \Z^2$ such that $\hat{f}^{n}(\hat{U})=\hat{U}+w$. From the fact that $U$ is not fully essential we have that $U$ is either contained in a topological annulus or inessential. In particular, there exists $v_0\in \Z^2_*$ such that whenever $u\in \Z^2_*$ is such that $\hat{U}+u$ intersects $\hat{U}$, then $u\in \R v_0$ (if $U$ is inessential, we choose $v_0$ arbitrarily). 

Note that we may assume that $v_0$ is not a multiple of any other element of $\Z^2_*$ by choosing it minimal in $\R v_0$. Since we assumed that $E$ is contained in the set $E_{v_0}$ from Lemma \ref{lem:Atkinsonnossocaso}, in particular $x\in E_{v_0}$ so the conclusion of the lemma holds for $\hat{x}$. This means that there are sequences of integers $(n_k)_{k\in \N}$ and $(m_k)_{k\in \N}$ such that $$\hat{f}^{n_k}(\hat{x})-\hat{x} - m_k v_0^\perp \to (0,0), \quad n_k\to \infty, \quad \text{and} \quad m_k/n_k\to 0$$ as $k\to \infty$. In particular, if $k$ is large enough, $f^{n_k}(U)\cap U\neq \emptyset$, so that $n_k=l_k n$ for some integer $l_k$. Moreover, if $0<\epsilon<\delta$, there is $k_0\in \N$ such that $k\geq k_0$ implies that
$\hat{f}^{l_kn}(\hat{x})\in B_\epsilon(\hat{x}+m_kv_0^\perp)\subset \hat{U}+m_kv_0^\perp$, so that $\hat{f}^{l_kn}(\hat{U})$ intersects $\hat{U}+ m_kv_0^\perp$.
The fact that $\hat{f}^n(\hat{U}) = \hat{U}+w$ implies that $\hat{f}^{l_kn}(\hat{U}) = \hat{U}+l_kw$, and so $\hat{U}+l_kw$ intersects $\hat{U}+m_kv_0^\perp$. This means that $\hat{U}$ intersects $\hat{U}+l_kw-m_kv_0^\perp$, and since $\hat{U}$ is a connected component of $\pi^{-1}(U)$, it follows that $\hat{U}=\hat{U}+l_kw-m_kv_0^\perp$. From our choice of $v_0$ follows that $l_kw-m_kv_0^\perp = r_kv_0$ for some $r_k\in \R$. But then 
$$(r_k/n_k)v_0 = (1/n)w - (m_k/n_k)v_0^\perp \to (1/n)w\quad  \text{as } k\to \infty,$$ and we conclude that $w=rv_0$ for some $r\in \R$. But then $(l_kr-r_k)v_0 = m_kv_0^\perp$, and therefore $m_k=0$, whenever $k\geq k_0$. The fact that $m_k=0$ for large $k$ in turn implies that $\hat{f}^{l_kn}(\hat{x})\in B_\epsilon(\hat{x})\subset \hat{U}$, showing that $\hat{x}$ is $\hat{f}$-recurrent. Since $\hat{f}^{lk_n}(\hat{U}) = \hat{U}+l_kw$ intersects $\hat{U}$, it must be equal to $\hat{U}$, so it follows that $\hat{U}$ is $\hat{f}$-periodic. Thus case (2) holds.
\endproof

\subsection{Proof of Theorem \ref{th:poincare-omega}}
\setcounter{claim}{0}

We assume that $v$ is not a multiple of any other element of $\Z^2_*$ by choosing it minimal in $\R v$. 
Let $E=E_v\cap E_{v^\perp}$ where $E_v$ is the set defined in Lemma \ref{lem:Atkinsonnossocaso}. Thus $E$ has full $\mu$-measure and the thesis of Lemma \ref{lem:Atkinsonnossocaso} holds both for $v_0=v$ and for $v_0=v^\perp$ at points of $E$. Finally, let $E^b$ be the set of all points $x\in E$ such that some (hence any) $\hat{x}\in \pi^{-1}(x)$ has a bounded orbit in the $v$ direction, \ie such that $\sup_{n\in \N} \abs{\smash{p_v(\hat{f}^n(\hat{x}))}} < \infty$.

\begin{claim}\label{claim:bdrec} If $x\in E^b$ and $\hat{x}\in \pi^{-1}(x)$, then $\hat{x}$ is $\hat{f}$-recurrent.
\end{claim}
\begin{proof}
Since $x\in E_{v^\perp}$, by Lemma \ref{lem:Atkinsonnossocaso} (with $v_0=v^\perp$) there are sequences $(n_k)_{k\in \N}$ and $(m_k)_{k\in \N}$ of integers such that $n_k\to \infty$,
$$\hat{f}^{n_k}(\hat{x}) - \hat{x} - m_kv\to (0,0) \quad \text{ and } \quad m_k/n_k\to 0 \quad \text{ as } k\to \infty.$$
Since $\hat{f}^{n_k}(\hat{x})$ is bounded in the $v$ direction, $m_k$ belongs to a finite set $\{m\in \Z : \abs{m}\leq M\}$. In particular, we may find $m\in \Z$ and a sequence $k_i\to \infty$ of integers such that $m_{k_i}=m$ for all $i\in \N$. This means that $\hat{f}^{n_{k_i}}(\hat{x}) \to \hat{x}+mv$, \ie $\hat{x}+mv$ is in the $\omega$-limit set $\omega(\hat{x},\hat{f})$. The latter is a closed $\hat{f}$-invariant set, so 
$$\omega(\hat{x}, \hat{f}) + mv = \omega(\hat{x}+mv,\hat{f})\subset \omega(\hat{x},\hat{f}).$$
It follows that $\omega(\hat{x}, \hat{f}) + lmv \subset \omega(\hat{x}, \hat{f})$ for any $l\in \N$. Thus, for any $l\in \N$,  the point $\hat{x}+lmv$ is accumulated by the orbit of $\hat{x}$. This is not possible if $l$ is large enough and $m\neq 0$, because the orbit of $\hat{x}$ is bounded in the $v$ direction. Thus $m=0$, and we conclude that $\hat{x}$ is recurrent, as required.
\end{proof}

Since $E$ has full measure, to prove the theorem it suffices to show that the set $N$ consisting of all  $x\in \supp(\mu)\cap E$ such that any $\hat{x}\in \pi^{-1}(x)$ is non-recurrent satisfies $\mu(N)=0$. Assume for a contradiction that $\mu(N)>0$, and note that by the previous claim, $N\cap E^b=\emptyset$, so any $\hat{x}\in \pi^{-1}(N)$ has an unbounded orbit in the $v$ direction. Note also that $N$ is $f$-invariant.

\begin{claim} $\pi(\omega_v)\cap \pi(\omega_{-v})=\emptyset$
\end{claim}
\begin{proof}
Suppose for contradiction that $\pi(\omega_v)\cap \pi(\omega_{-v})\neq \emptyset$. Then the hypotheses of Proposition \ref{pro:omegasetocam} hold for some choice of $p_1,p_2\in \Z$, so by part (2) of said proposition we conclude that $\mu(N\cap \pi(\omega_v)\cap \pi(\omega_{-v}))>0$ (since $\mu(N)>0$ and points of $\pi^{-1}(N)$ are non-recurrent). 
In particular, there exists $z\in \R^2$ such that $\pi(z)\in N \cap \pi(\omega_v)\cap \pi(\omega_{-v})$. We claim that $z$ has a bounded orbit in the $v$ direction. In fact, since $\pi(z)\in \pi(\omega_v)$, there is $w\in \Z^2$ such that $z+w\in \omega_v$. From the definition of $\omega_v$, this implies that $\hat{f}^n(z)+w \subset H_v^+$ for all $n\in \Z$, and so $p_v(\hat{f}^n(z))\geq -p_v(w)$ for all $n\in \Z$. Similarly, since $\pi(z)\in \pi(\omega_{-v})$ we conclude that $p_{-v}(\hat{f}^n(z))\geq -p_{-v}(w')$ for some $w'\in \Z^2$, which means that $p_v(\hat{f}^n(z))\leq -p_{v}(w')$ for all $n\in \Z$. Thus $z$ has a bounded orbit in the $v$ direction. Since $N\subset E$, this means that $\pi(z)\in E^b$. But then the previous claim implies that $z$ is $\hat{f}$-recurrent, contradicting the fact that $\pi(z)\in N$.
\end{proof}

Note that, by Proposition \ref{pr:omeganaovazio}, both $\omega_v$ and $\omega_{-v}$ are nonempty. Since $\pi^{-1}(N)$ consists of non-recurrent points, only case (1) of Lemma \ref{lem:tresopcoes} is possible, so that $\supp(\mu)\subset \ol{\pi(\omega_v)}$. By the same argument applied to $-v$ instead of $v$ we also have that $\supp(\mu)\subset \ol{\pi(\omega_{-v})}$. 

Let $y\in N$ and $\hat{y}\in \pi^{-1}(y)$. Since $\hat{y}$ is non-recurrent, it is in particular not fixed, so we may choose a positive $\epsilon<1$ such that $\hat{f}(B_\epsilon(\hat{y}))$ is disjoint from $B_\epsilon(\hat{y})$. Since $y\in \ol{\pi(\omega_v)}\cap \ol{\pi(\omega_{-v})}$, we may choose integers $p_1$ and $p_2$ such that both $\omega_v+p_1v$ and $\omega_v+p_2v$ intersect $B_\epsilon(\hat{y})$.

Let $\ell$ be a straight line segment joining a point of $B_\epsilon(\hat{y})\cap (\omega_v+p_1v)$ with a point of $B_\epsilon(\hat{y})\cap (\omega_{-v}+p_2v)$, not including its endpoints. The segment $\ell$ can be chosen in a way that it is disjoint from $(\omega_v+p_1v)\cup (\omega_{-v}+p_2v)$ (this can be done replacing $\ell$ by an appropriate connected component of $\ell\sm ((\omega_v+p_1v)\cup (\omega_{-v}+p_2v))$. Since $\ell\subset B_\epsilon(\hat{y})$, we have that $\hat{f}(\ell)$ is disjoint from $\ell$, so the hypotheses of Proposition \ref{pro:omeganaotoca} hold (observing that the fact that $(0,0)$ is the rotation vector of some ergodic measure implies that $\hat{f}$ has a fixed point; see \cite{franks-ergodic}). 

Let $F$, $\Omega_1$ and $\Omega_2$ be the sets from Proposition \ref{pro:omeganaotoca}, and assume $f(\Omega_2)\subset \Omega_2$ (the other case is analogous). Thus we have that
\begin{equation}\label{eq:trap}
\hat{f}(\Omega_2)\subset \Omega_2, \quad T_{v^\perp}(\Omega_2)\subset \Omega_2,\quad \hat{f}^{-1}(\Omega_1)\subset \Omega_1,\quad \text{ and } T_{v^\perp}^{-1}(\Omega_1)\subset \Omega_1. 
\end{equation}

\begin{claim} There is $x\in N$ and $\hat{x}\in \pi^{-1}(x)\cap B_\epsilon(\hat{y})$ such that $\hat{x}\notin (\omega_v+p_1v)\cup (\omega_{-v}+p_2v)$.
\end{claim}
\begin{proof} 
Since $\pi(\omega_v)\cap \pi(\omega_{-v})=\emptyset$, we may assume that $y\notin \pi(\omega_{-v})$ (the other case is similar). If $\hat{y}\notin \omega_v+p_1v$ then the claim holds with $\hat{x}=\hat{y}$. 

Assume that $\hat{y}\in \omega_v+p_1v$. Since $y\in N\subset E_{v^\perp}$, by Lemma \ref{lem:Atkinsonnossocaso} there exist $n_k\in \N$ and $m_k\in \Z$ such that $\hat{f}^{n_k}(\hat{y})-\hat{y}-m_kv\to 0$ as $k\to \infty$. We may assume that $\abs{m_k}\to \infty$, using the argument from the proof of Claim \ref{claim:bdrec}: Indeed, if $\{m_k\}_{k\in \N}$ assumes only finitely many values, there exists $m\in \Z$ and a sequence $k_i\to \infty$ such that $m_{k_i} = m$, so that $\hat{f}^{n_{k_i}}(\hat{y}) \to \hat{y}+mv$. In other words, $\hat{y}+mv$ belongs to the $\omega$-limit set $\omega(\hat{y}, \hat{f})$. Since $\hat{y}$ is non-recurrent, it follows that $m\neq 0$, and since the $\omega$-limit set is closed and $\hat{f}$-invariant, it follows that $\omega(\hat{y},\hat{f})+mv = \omega(\hat{y}+mv,\hat{f})\subset \omega(\hat{y},\hat{f})$. From these facts we deduce that $\hat{y}+lmv\in \omega(\hat{y},\hat{f})$ for any $l\in \N$, so a new choice of the sequences $(n_k)_{k\in \N}$ and $(m_k)_{k\in \N}$ can be made so that $|m_k|\to \infty$, as claimed.

Since $\hat{y}\in \omega_v+p_1v$, the fact that $|m_k|\to \infty$ implies that $m_k\to \infty$. In particular, we may choose $k$ so large that $\hat{y}-m_kv\notin H_v^+ + p_1v$. Letting
$\hat{x}=\hat{f}^{n_k}(\hat{y})-m_kv\in B_\epsilon(\hat{y})$, we have that $\hat{f}^{-n_k}(\hat{x}) = \hat{y}-m_kv$ which is not in $\omega_v+p_1v\subset H_v^+ + p_1v$, and since $x:=\pi(\hat{x}) = f^{n_k}(y) \notin \pi(\omega_{-v})$, we also have that $\hat{x}\notin \omega_{-v}+p_2v$. Finally, since $N$ is $f$-invariant and $y\in N$, we also have that $x\in N$ concluding the proof of the claim.

%In particular, choosing $k$ large enough and letting $m=m_k, n=n_k$ we have that $m>p_1+\langle \hat{y}, v\rangle$ and $\hat{f}^{n}(\hat{y})-mv\in B_\epsilon(\hat{y})$.

%Since $\hat{y}\in N\subset E_v$ as well, Lemma \ref{lem:Atkinsonnossocaso} implies that there exists $n'>n$ and $m'\in \Z$ such that $\hat{f}^{n'}(\hat{y})\in B_\epsilon(\hat{y}+m'v^\perp)$. 

\end{proof}

In order to simplify notation, let $T=T_{v^\perp}$.
\begin{claim}\label{claim:nzero} There is $\delta>0$ and $n_0\in \N$ such that $T^n(B_\delta(\hat{x}))$ is contained in $\Omega_2$ if $n> n_0$ and in $\Omega_1$ if $n<n_0$. 
\end{claim}
\begin{proof}
By part (2) of Proposition \ref{pro:omeganaotoca} there exists $n_0$ such that $T^n(\hat{x})$ lies in $\Omega_2$ if $n>n_0$ and in $\Omega_1$ if $n<n_0$. Let $\delta$ be such that $T^{n_0+1}B_\delta(\hat{x}))\subset \Omega_2$ and $T^{n_0-1}(B_\delta(\hat{x}))\subset \Omega_1$. Then the claim follows from (\ref{eq:trap}).
\end{proof}

\begin{claim}\label{claim:sube} For any given $r_0\in \Z$ there is $r>r_0$ and $k_0\in \N$ such that $\hat{f}^{k_0}(\hat{x})\in T^{r}(B_\delta(\hat{x}))$ 
\end{claim}
\begin{proof}
Recall that $N\subset E= E_v\cap E_{v^\perp}$. Thus, Lemma \ref{lem:Atkinsonnossocaso} implies that there are sequences $(n_k)_{k\in \N}$ of positive integers and $(m_k)_{k\in \N}$ of integers such that $n_k\to \infty$,
$$\hat{f}^{n_k}(\hat{x}) - \hat{x} - m_kv^\perp\to (0,0) \quad \text{ and } \quad m_k/n_k\to 0 \quad \text{ as } k\to \infty.$$
In particular, there is $k_0\in \N$ such that $\hat{f}^{n_k}(\hat{x})\in T^{m_k}(B_\delta(\hat{x}))$ when $k\geq k_0$. Thus, if $(m_k)_{k\in \N}$ is unbounded above, then we are done.

Suppose that $(m_k)_{k\in \N}$ is bounded above. By part (5) of Proposition \ref{pro:omeganaotoca} we see that $(m_k)_{k\in \N}$ is also bounded below. Repeating what was done in the proof of Claim \ref{claim:bdrec} (using $v$ instead of $v^\perp$) we conclude that there is $m\in \N$ such that $\hat{x}+lmv^\perp$ belongs to the $\omega$-limit set of $\hat{x}$ for any $l\in \N$. Again by part (5) of Proposition \ref{pro:omeganaotoca} we see that $m\geq 0$, and since $\hat{x}$ is not recurrent, we deduce that $m\neq 0$. Therefore $m>0$, and the claim easily follows.
\end{proof}

\begin{claim} There is a neighborhood $W$ of $\hat{x}$ and $k_0\in \N$ such that $\hat{f}^k(W)$ is disjoint from $T(W)$ for all $k\geq k_0$.
\end{claim}
\begin{proof}
By Claim \ref{claim:nzero}, we know that $T^{n_0-1}(B_\delta(\hat{x}))\subset \Omega_1$ and $T^{n_0+1}(B_\delta(\hat{x}))\subset \Omega_2$. In particular, $T(B_\delta(\hat{x}))\subset T^{2-n_0}(\Omega_1)$.
%$T(B_\delta(\hat{x}))\subset T^{2-n_0}(\Omega_1)$. 
By the previous claim, there is $r\geq 3$ and $k_0\in \N$ such that $$\hat{f}^{k_0}(\hat{x})\in T^r(B_\delta(\hat{x}))= T^{r-(n_0+1)}(T^{n_0+1}(B_\delta(\hat{x}))) \subset T^{r-n_0-1}(\Omega_2)\subset T^{2-n_0}(\Omega_2),$$
where the latter inclusion follows from (\ref{eq:trap}) and from the fact that $r\geq 3$. Thus, there is a neighborhood $W\subset B_\delta(\hat{x})$ of $\hat{x}$ such that $\hat{f}^{k_0}(W)\subset T^{2-n_0}(\Omega_2)$. Again by (\ref{eq:trap}) we deduce that $\hat{f}^k(W)\subset T^{2-n_0}(\Omega_2)$ for any $k\geq k_0$. Since $$T(W) \subset T(B_\delta(\hat{x}))\subset T^{2-n_0}(\Omega_1)$$ and the latter set is disjoint from $T^{2-n_0}(\Omega_2)$, we conclude that $\hat{f}^k(W)\cap T(W)=\emptyset$ for any $k\geq k_0$.
\end{proof}

Due to the previous claim, if we start the proof again but choosing $\epsilon$ so small that $B_\epsilon(\hat{x})\subset W$ (and so $\ell\subset W$) we may assume the following:
\begin{equation}\label{eq:notocauno}
\hat{f}^k(\ell)\cap T(\ell)=\emptyset \text{ if } k\geq k_0.
\end{equation}

\begin{claim} There is $z\in \Omega_1$ and $k_1\geq k_0$ such that $\hat{f}^{k_1}(z)\in T(\Omega_2)$. 
\end{claim}
\begin{proof} Let $z=T^{n_0-1}(\hat{x})$, so $B_\delta(z)\subset \Omega_1$ and $T^2(B_\delta(z))\subset \Omega_2$. By Claim \ref{claim:sube}, there is $k_1\in \N$ and $r\geq 3$ such that $\hat{f}^{k_1}(\hat{x})\in T^r(B_\delta(\hat{x}))$. This means that
$$\hat{f}^{k_1}(z)\in T^r(B_\delta(z)) = T^{r-2}(T^2(B_\delta(z))\subset T^{r-2}(\Omega_2)\subset T(\Omega_2),$$
where we used (\ref{eq:trap}) and the fact that $r\geq 3$ for the last inclusion.
\end{proof}
%$$\hat{f}^{k_0}(\hat{x})\in T^r(B_\delta(\hat{x})) \subset T^{r-(n_0-1)}(B_\delta(z))\subset T^{r-n_0-1}(T^2(B_\delta(z)))\subset T^{r-n_0-1}(\Omega_2),$$
%Thus 
%$\hat{f}^{k_0}(z) \in T^{n_0-1}(T^{r-n_0-1}(\Omega_2)) = T^{r-2}(\Omega_2)\subset T^2(\Omega_2)$
%where the latter is true due to (\ref{eq:trap}) and the fact that $r>4$. 

The last claim implies that $\hat{f}^{k_1}(\Omega_1)$ intersects $T(\Omega_2)$. But $\hat{f}^{k_1}(\Omega_2)$ also intersects $T(\Omega_2)\subset \R^2\sm T(\Omega_1)$: indeed, if $z_0$ is a fixed point of $\hat{f}$ in $\Omega_2$ (which exists by part (3) of Proposition \ref{pro:omeganaotoca}) then $T(z_0)$ is a fixed point of $\hat{f}$ in $T(\Omega_2)\subset \Omega_2$  and so $T(z_0)$ belongs to $\hat{f}^{k_1}(\Omega_2)\cap T(\Omega_2)$.

Thus $T(\Omega_2)$ is connected and intersects both $\hat{f}^{k_1}(\Omega_1)$ and its complement, and we deduce that $T(\Omega_2)$ intersects $\bd\hat{f}^{k_1}(\Omega_1)$. The latter is a subset of 
$$\hat{f}^{k_1}(F) = \hat{f}^{k_1}(\omega_{v}\cup \ell\cup \omega_{-v}) = \omega_{v}\cup \hat{f}^{k_1}(\ell)\cup \omega_{-v},$$ and since $T(\Omega_2)$ is disjoint from $\omega_{v}\cup \omega_{-v}$ we see that $\hat{f}^{k_1}(\ell)$ intersects $T(\Omega_2)$. This also means that $\hat{f}^{k}(\ell)$ intersects $T(\Omega_2)$ for all $k\geq k_1$, because $\hat{f}(T(\Omega_2))\subset T(\Omega_2)$.

Suppose that $\hat{f}^{k_1}(\ell)\cap (\R^2\sm T(\Omega_2))\neq \emptyset$. Then $\hat{f}^{k_1}(\ell)$ is a connected set intersecting $T(\Omega_2)$ and its complement, so it intersects the boundary of $T(\Omega_2)$, which is a subset of
$$T(F) = T(\omega_v\cup \ell \cup \omega_{-v}) = \omega_v\cup T(\ell)\cup \omega_{-v}.$$ 
Since $\ell$ (and thus $\hat{f}^{k_1}(\ell)$ as well) is disjoint from $\omega_v\cup \omega_{-v}$, we see that $\hat{f}^{k_1}(\ell)$ intersects $T(\ell)$, contradicting the fact that $k_1\geq k_0$.

Thus $\hat{f}^{k_1}(\ell)\cap (\R^2\sm T(\Omega_2))=\emptyset$, \ie $\hat{f}^{k_1}(\ell)\subset T(\Omega_2)$. This implies that $T(\Omega_1)$ is disjoint from $\hat{f}^{k_1}(\ell)$, and being also disjoint from $\hat{f}^{k_1}(\omega_{v}\cup \omega_{-v}) = \omega_v\cup\omega_{-v}$ we deduce that $T(\Omega_1)$ is disjoint from $\hat{f}^{k_1}(F)$. Therefore $T(\Omega_1)$ is contained in one of the two connected components of $\R^2\sm \hat{f}^{k_1}(F)$, which are $\hat{f}^{k_1}(\Omega_2)$ and $\hat{f}^{k_1}(\Omega_1)$. Since $\emptyset\neq \Omega_1\subset T(\Omega_1)\cap \hat{f}^{k_1}(\Omega_1)$ due to (\ref{eq:trap}), the only possibility is $T(\Omega_1)\subset \hat{f}^{k_1}(\Omega_1)$. 

The last fact implies that $T^m(\Omega_1)\subset \hat{f}^{mk_1}(\Omega_1)$ for all $m\in \N$. In particular, if $\hat{z_0}$ is a fixed point of $\hat{f}$ in $\Omega_1$, then $T^m(z) = \hat{f}^{-mk_1}(T^{m}(z))\in \Omega_1$ for all $m \in \N$. But part (2) of Proposition \ref{pro:omeganaotoca} implies that there is $T^m(z)\in \Omega_2$ if $m$ is large enough. This contradiction completes the proof of the theorem.

\section{Some results relying on equivariant Brouwer theory}\label{sec:brouwer}

In this section we recall the results and definitions from \cite{KT2012} and we use them to prove some properties of homeomorphisms with a nonwandering lift. The main concepts behind these results is the equivariant Brouwer theory developed by Le Calvez \cite{lecalvez-equivariant} and a recent result of Jaulent on maximal unlinked sets \cite{jaulent}. We do not intend to explain how these results are used in this context; for that, the reader is directed to Section 3 of \cite{KT2012}.

\subsection{Gradient-like Brouwer foliations for nonwandering lifts}

Let $S$ be an orientable surface (not necessarily compact), and  $\pi\colon \hat{S}\to S$ is the universal covering of $S$. Let $\mc{I}=(f_t)_{t\in [0,1]}$ be an isotopy from $f_0=\id_S$ to some homeomorphism $f_1=f$, and $\hat{\mc{I}}=(\hat{f}_t)_{t\in [0,1]}$ the lift of the isotopy $\mc{I}$ such that $\hat{f}_0=\id_{\hat{S}}$. Define $\hat{f}=\hat{f}_1$, so that $\hat{f}$ is a lift of $f$ which commutes with every covering transformation.

Suppose $X\subset S$ is a totally disconnected set of fixed points of $f$. We regard an oriented topological foliation $\mc{F}$ of $S\sm X$ as a foliation with singularities of $S$. Suppose that every point of $X$ is fixed by the isotopy $\mc{I}$ (\ie $f_t(x)=x$ for all $t\in [0,1]$ and $x\in X$). We say that an arc $\gamma\colon [0,1]\to S\sm X$ is positively transverse to $\mc{F}$ if $\gamma$ crosses the leaves of the foliation locally from left to right. 
We say that the isotopy $\mc{I}$ is transverse to $\mc{F}$ if for each $x\in S$, the arc $(f_t(x))_{x\in [0,1]}$ is homotopic, with fixed endpoints in $S\sm X$, to an arc that is positively transverse to $\mc{F}$. In this case, it is also said that $\mc{F}$ is dynamically transverse to $\mc{I}$. If $\hat{X}=\pi^{-1}(X)$, then the isotopy $\hat{\mc{I}}$ fixes $\hat{X}$ pointwise. If $\mc{F}$ is dynamically transverse to $\mc{I}$, then the lifted foliation $\hat{\mc{F}}$ (with singularities in $\hat{X}$) of $\hat{S}$ is also dynamically transverse to $\hat{\mc{I}}$. 

Until the end of this section, we fix a homeomorphism $f\colon \T^2\to \T^2$ isotopic to the identity, and a lift $\hat{f}\colon \R^2\to \R^2$. The main existence result that we will use, which is a consequence of \cite{lecalvez-equivariant} and \cite{jaulent}, is stated as Proposition 3.10 in \cite{KT2012} (we include some of the preceding comments in the statement here)

\begin{proposition}\label{pro:kt-brouwer} If $\fix(\hat{f})$ is totally disconnected, then there exists a compact set $X\subset \pi(\fix(\hat{f}))$, an oriented foliation $\mc{F}$ of $\T^2$ with singularities in $X$, and an isotopy $\mc{I}=(f_t)_{t\in [0,1]}$ from the identity to $f$ such that 
\begin{itemize}
\item $\mc{I}$ lifts to an isotopy $\hat{\mc{I}} = (\hat{f}_t)_{t\in [0,1]}$ from $\id_{\R^2}$ to $\hat{f}$, 
\item $\mc{I}$ fixes $X$ pointwise, and $\hat{\mc{I}}$ fixes $\hat{X} = \pi^{-1}(X)$ pointwise,
\item $\mc{F}$ is dynamically transverse to $\mc{I}$ and the lifted foliation $\hat{\mc{F}}$ on $\R^2$ with singularities in $\hat{X}$ is dynamically transverse to $\hat{\mc{I}}$.
\end{itemize}
\end{proposition}
\begin{remark} Any oriented foliation with singularities such as $\mc{F}$ and $\hat{\mc{F}}$ is the orbit space of a continuous flow \cite{whitney,whitney2}.
\end{remark}

Let $\mc{F}$ be the foliation from Proposition \ref{pro:kt-brouwer}. For a loop $\gamma$ in $\T^2$, we denote by $\gamma^*$ its homology class in $H_1(\T^2, \Z)\simeq \Z^2$. Fix $z\in \T^2\sm X$, and consider the set $\mc{C}(z)$ of all homology classes $\kappa\in H^1(\T^2,\Z)$ such that there is a positively transverse loop $\gamma$ with basepoint $z$ such that $\gamma^* = \kappa$. Identifying $H^1(\T^2,\Z)$ with $\Z^2$ naturally and choosing $\hat{z}\in \pi^{-1}(z)$, we see that $\mc{C}(z)$ coincides with the set of all $v\in \Z^2$ such that there is an arc in $\R^2$ positively transverse to the lifted foliation $\hat{\mc{F}}$ joining $\hat{z}$ to $\hat{z}+v$. Note that $\mc{C}(z)$ is closed under addition: if $v,w\in \mc{C}(z)$ then $v+w\in \mc{C}(z)$. By part (4) of Proposition 3.6 of \cite{KT2012}, any pair of points lying in a connected subset of the nonwandering set of $\hat{f}$ can be joined by a positively transverse arc. This implies that $\mc{C}(z) = \Z^2$ for all $z\in \T^2\sm X$. Thus, putting together Proposition 3.11 and Lemma 3.8 of \cite{KT2012} we have the following result.

\begin{proposition}\label{pro:gradient} Under the hypotheses of Proposition \ref{pro:kt-brouwer}, if the nonwandering set of $\hat{f}$ is $\R^2$, then $\mc{F}$ is a \emph{gradient-like} foliation, \ie the following properties hold:
\begin{itemize} 
\item[(1)] every regular leaf of $\mc{F}$ is a connection, and so is every regular leaf of $\hat{\mc{F}}$,
\item[(2)] $\mc{F}$ and $\hat{\mc{F}}$ have no generalized cycles, and
\item[(3)] there is a constant $M$ such that $\diam(\Gamma)<M$ for each regular leaf $\Gamma$ of $\hat{\mc{F}}$.
\end{itemize}
\end{proposition}
Let us recall that a regular leaf of $\mc{F}$ is any element of $\mc{F}$ that is not a singularity. A leaf $\Gamma$ of $\mc{F}$ is a \emph{connection} if both its $\omega$-limit and its $\alpha$-limit are one-element subsets of $\sing(\mc{F})$. By a \emph{generalized cycle of connections} of $\mc{F}$ we mean a loop $\gamma$ such that $[\gamma]\sm \sing(\mc{F})$ is a disjoint union of regular leaves of $\mc{F}$ that are traversed positively by $\gamma$.

\subsection{Boundedness of periodic free disks}\label{sec:uepsilon-invariant}

A version of the next result was proved in \cite{KT2012} under the  assumption that there is a gradient-like Brouwer foliation.
\begin{theorem}\label{th:free-disk} Let $\hat{f}\colon \R^2\to \R^2$ be a lift of a homeomorphism $f\colon \T^2\to \T^2$ isotopic to the identity, and suppose that $\hat{f}$ is nonwandering and $\pi(\fix(\hat{f}))$ is inessential. Then every periodic open topological disk in $\R^2$ that is disjoint from its image by $\hat{f}$ is bounded.
\end{theorem}

Before proving \ref{th:free-disk} let us state the consequence that will be useful in our setting. Recall the notation introduced in \S\ref{sec:uepsilon}.

\begin{corollary}\label{coro:uepsilon-invariant} Under the hypotheses of Theorem \ref{th:free-disk}, if $U_\epsilon(z)=U_\epsilon(z,\hat{f})$ is unbounded for some $z\in \R^2$ and $\epsilon>0$, then $\hat{f}(U_\epsilon(z))=U_\epsilon(z)$.
\end{corollary}

\begin{proof}[Proof of Theorem \ref{th:free-disk}]
In the case that $\fix(\hat{f})$ is totally disconnected, this is a direct consequence of Propositions \ref{pro:kt-brouwer} and \ref{pro:gradient}, together with Corollary 4.7 of \cite{KT2012}. We will show how to reduce the case that $\fix(\hat{f})$ is not totally disconnected to this case.

First recall that a compact set $K$ is \emph{filled} if $\T^2\sm K$ has no inessential connected components (i.e. $K=\fil(K)$). If $K$ is inessential, this is the same as saying that $\T^2\sm K$ is connected. Under the hypotheses of Theorem \ref{th:free-disk}, we know that the set $K_0=\pi(\fix(\hat{f}))$ is inessential, so it's filling $K=\fil(K_0)$ is a compact filled set. Thus, we may apply the following \cite[Proposition 1.6]{KT2012}:
\begin{proposition}\label{pro:collapse} Let $K\subset \T^2$ be a compact inessential filled set, and $f\colon \T^2\to \T^2$ a homeomorphism such that $f(K)=K$. Then there is a continuous surjection $h\colon \T^2\to \T^2$ and a homeomorphism $f'\colon \T^2 \to \T^2$ such that
\begin{itemize}
\item $h$ is homotopic to the identity;
\item $hf = f'h$;
\item $K' = h(K)$ is totally disconnected;
\item $h|_{\T^2\sm K}\colon \T^2\sm K \to \T^2\sm K'$ is a homeomorphism.
\end{itemize}
\end{proposition}
The fact that $K_0\subset \fix(f)$ implies that the connected components of $K$ are $f$-invariant, which in turn implies that $K'\subset \fix(f')$. Moreover, if $\hat{h}$ and $\hat{f}'$ are lifts of $h$ and $f$ to $\R^2$, then $\hat{h}\hat{f}=\hat{f}'\hat{h}$ and $\hat{h}|_{\R^2\sm \pi^{-1}(K)}$ is a homeomorphism onto $\R^2\sm \pi^{-1}(K')$. 
Assume for contradiction that there is some unbounded $\hat{f}$-periodic free topological disk $U\subset \R^2$. 
It is easy to see from the definition that $\bd\pi^{-1}(K)$ consists of fixed points of $\hat{f}$. Moreover, since $K$ is compact and inessential, Proposition \ref{pro:compact-ine} implies that every connected component of $\pi^{-1}(K)$ is bounded. Since $U$ is disjoint from $\fix(\hat{f})$ and unbounded, we deduce that $U$ is disjoint from $\pi^{-1}(K)$. This implies that $\hat{h}$ is injective on $U$, and so $U' = \hat{h}(U)$ is an $\hat{f}'$-periodic and $\hat{f}'$-free topological disk. Furthermore, since $h$ is homotopic to the identity, there is $M'$ such that $\norm{\smash{\hat{h}(z)-z}}\leq M'$ for all $z\in \R^2$, and therefore $U'=\hat{h}(U)$ is unbounded as well. Finally, since the nonwandering set of $\hat{f}|_{\R^2\sm \pi^{-1}(K)}$ is $\R^2\sm \pi^{-1}(K)$, we have that the nonwandering set of $\hat{f}'|_{\R^2\sm \pi^{-1}(K')}$ is $\R^2\sm \pi^{-1}(K')$, and the latter set is dense in $\R^2$ because $\pi^{-1}(K')$ is totally disconnected (since $K'$ is totally disconnected). Since the nonwandering set is closed, it follows that $f'$ is nonwandering. 

Therefore $\hat{f'}$ is nonwandering, it has a totally disconnected set of fixed points, and it has a periodic unbounded free topological disk $U'$. But we already explained at the beginning of the proof that this is not possible, hence we obtain a contradiction.
\end{proof}

\subsection{Engulfing and $(\hat{f},\hat{\mc{F}})$-arcs}

In this section we assume that $f$ and its corresponding lift $\hat{f}$ satisfies the hypotheses of Proposition \ref{pro:gradient} (i.e. $\hat{f}$ is nonwandering and there is a gradient-like Brouwer foliation).
%In particular, we have $\hat{\mc{I}}=(\hat{f}_t)_{t\in [0,1]}$, $\hat{F}$, and $\hat{X}$ which are the isotopy, the oriented foliation, and the set of singularities given by Propositions \ref{pro:kt-brouwer} and \ref{pro:gradient} (so that $\hat{\mc{F}}$ is a \emph{gradient-like Brouwer foliation} and $\hat{X}\subset \fix(\hat{f}))$. 

Let us state an immediate consequence of \cite[Propositions 4.6 and 4.8]{KT2012}: % and 4.6 for size
\begin{proposition}[Engulfing]\label{pro:engulfing} If $U$ is an unbounded $\hat{f}$-invariant open topological disk, then every leaf of $\hat{\mc{F}}$ that intersects $\ol{U}$ has one endpoint in $U$. 
\end{proposition}

We will need the following improvement.

\begin{proposition}\label{pro:engulfing-finite} Given a nonempty compact set $R\subset \R^2$, there is a \emph{finite} set $P\subset \hat{X}$ such that any unbounded open $\hat{f}$-invariant topological disk intersecting $R$ also intersects $P$.
\end{proposition}

To prove the previous proposition, we need a definition: let us say that a compact arc $\gamma$ in $\R^2$ is a $(\hat{\mc{F}},\hat{f})$-arc if $[\gamma]$ is contained in the union of finitely many $\hat{f}$-iterates of leaves of $\hat{\mc{F}}$ and elements of $\hat{X}$ (the orientation of these arcs is irrelevant).

\begin{proposition}\label{pro:fF-arc} Any two points of $\R^2\sm \hat{X}$ can be joined by an $(\hat{\mc{F}},\hat{f})$-arc. 
\end{proposition}

\begin{proof}
Define a relation on $\R^2\sm \hat{X}$ by $z\sim z'$ if there is an $(\hat{\mc{F}},\hat{f})$-arc joining $z$ to $z'$. Clearly $\sim$ is an equivalence relation. Since $\R^2\sm \hat{X}$ is connected, to prove that there is a unique equivalence class it suffices to show that the equivalence classes are open. Denote by $\mc{W}(z_0)$ the equivalence class of $z_0\in \R^2\sm \hat{X}$.

 Given $z\in \mc{W}(z_0)$, let $\Gamma_z$ be the leaf of $\hat{\mc{F}}$ containing $z$. Note that $\Gamma_z$ must join some point $q_0\in \hat{X}$ to a different point $q_1\in \hat{X}$. The isotopy $\hat{\mc{I}}=(\hat{f}_t)_{t\in [0,1]}$ extends to the one-point compactification $\R^2\sqcup\{\infty\}$ by fixing $\infty$ (we still denote it $\hat{\mc{I}}$), and we may regard as $\hat{\mc{F}}$ as a foliation of $\R^2\sqcup\{\infty\}$ with singularities in $\hat{X}\cup\{\infty\}$, which is still dynamically transverse to $\hat{\mc{I}}$. Let $\til{\pi}\colon \til{A}\to A$ be the universal covering of the topological annulus $A = \R^2\sqcup\{\infty\}\sm \{q_0, q_1\}$. The restriction of the isotopy $\hat{\mc{I}}$ to $A$ lifts to an isotopy $\til{\mc{I}}=(\til{f}_t)_{t\in [0,1]}$ from the identity to $\til{\mc{F}}$ from $\til{f}_0=\id_{\til{A}}$ to some lift $\til{f}:=\til{f}_1$ of $\hat{f}|_{A}$. We also have a lifted foliation $\til{\mc{F}}$ of $\til{A}$ with singularities in $\til{X} = \til{\pi}^{-1}(X\cup\{\infty\}\sm \{q_0, q_1\})$, a set which is fixed pointwise by $\til{\mc{I}}$. The foliation $\til{\mc{F}}$ is also dynamically transverse to $\til{\mc{I}}$. 

Consider $\til{z}\in \til{\pi}^{-1}(z)$ and let $\til{\Gamma}_z$ be the leaf of $\til{\mc{F}}$ containing $\til{z}$ (which is a lift of $\Gamma_z$). Since $\Gamma_z$ joins $q_0$ to $q_1$, it follows that $\til{\Gamma}_z$ is a properly embedded line, so it separates $\til{A}\simeq \R^2$ into exactly two connected components. Furthermore, the fact that $\til{\mc{F}}$ is dynamically transverse implies that $\til{\Gamma}_z$ is a Brouwer line for $\til{f}$ in the traditional sense, \ie $\til{f}(\til{\Gamma}_z)$ and $\til{f}^{-1}(\til{\Gamma}_z)$ lie in different connected components of $\til{A}\sm \til{\Gamma}_z$. Let $H^-$ be the connected component of $\til{A}\sm \til{\Gamma}_z$ containing $\til{f}^{-1}(\til{\Gamma}_z)$ and $H^+$ the remaining component (which contains $\til{f}(\til{\Gamma}_z)$). 
Then $\cl(\til{f}(H^+))\subset H^+$ and $\cl(\til{f}^{-1}(H^-))\subset H^-$. The set $\til{V} = \til{f}(H^-) \cap \til{f}^{-1}(H^+)$ is an open neighborhood of $\til{\Gamma}_z$ and $\til{f}^2(\til{V})\subset \til{f}(H^+)$ which is disjoint from $\til{f}(H^-)$ (hence from $\til{V}$). Thus $\til{f}^2(\til{V})\cap \til{V}=\emptyset$, and $\til{V}$ contains no fixed points of $\til{f}$.

Let $V = \til{\pi}(\til{V})$, which is a neighborhood of $z$. We will show that $V\subset \mc{W}(z_0)$. Fix $y\in V$ and let $\Gamma_y$ be the leaf of $\hat{\mathcal{F}}$ containing $y$. Let $\til{y}\in \til{V}\cap \til{\pi}^{-1}(y)$, and let $\til{\Gamma}_y$ be the leaf of $\til{\mathcal{F}}$ containing $\til{y}$ (so $\til{\Gamma}_y$ projects to $\Gamma_y$). 

Since $\hat{\mc{F}}$ is gradient-like we know that $\Gamma_y$ connects two different elements $p_0$ and $p_1$ of $\hat{X}$. Suppose first that $p_i\notin\{q_0,q_1\}$ for some $i\in \{0,1\}$. Then there is $\til{p}_i\in \til{\pi}^{-1}(p_i)$ such that $\til{\Gamma}_y$ has one endpoint (\ie its $\omega$-limit or $\alpha$-limit) in $\til{p}_i$. But since $\til{V}$ contains no fixed points of $\til{f}$, and $\til{p}_i$ is fixed, it follows that $\til{p}_i\notin \til{V}$. Since $\til{y}\in \til{\Gamma}_y$, we conclude that $\til{\Gamma}_y$ intersects $\bd \til{V}$. This means that $\til{\Gamma}_y$ intersects $\til{f}^{-1}(\til{\Gamma}_z)\cup \til{f}(\til{\Gamma}_z)$, and so $\Gamma_y$ intersects $\hat{f}^{-1}(\Gamma_z)\cup \hat{f}(\Gamma_z)$. But for each $i\in \Z$, the arc $\hat{f}^{i}(\Gamma_z)$ is an $(\hat{f},\hat{\mc{F}})$-arc joining $q_0$ to $q_1$, and one of them (namely $\Gamma_z$) contains a point of $\mc{W}(z_0)$. By concatenation, it follows that $\hat{f}^{i}(\Gamma_z)\subset \mc{W}(z_0)$ for all $i\in \Z$. Since we showed that $\Gamma_y$ intersects $\hat{f}^i(\Gamma_z)$ for some $i\in \{-1,1\}$, we conclude again by concatenation that $y\in \mc{W}(z_0)$, as we wanted to show. See Figure \ref{fig:fF-arc}.

\begin{figure}[ht]
\includegraphics[height=5cm]{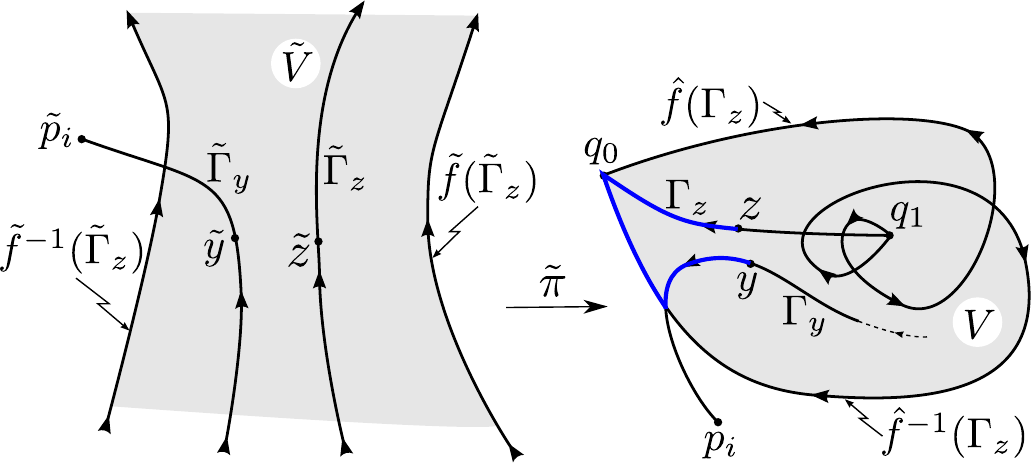}
\caption{Proof of Proposition \ref{pro:fF-arc}.}
\label{fig:fF-arc}
\end{figure}

It remains to consider the case where $\{p_0, p_1\} = \{q_0,q_1\}$. But in this case, $\Gamma_z\cup \Gamma_y$ contains an $(\hat{f},\hat{\mc{F}})$-arc joining $z\in \mc{W}(z_0)$ to $y$, and so by concatenation $y\in \mc{W}(z_0)$ again. This shows that $V\subset \mc{W}(z_0)$, proving that $\mc{W}(z_0)$ is open and concluding the proof.
\end{proof}

%
%\begin{claim} $\hat{f}(\mc{W}(z_0)) = \mc{W}(z_0)$ for any $z_0\in \R^2\sm \hat{X}$.
%\end{claim}
%\begin{proof} Let $z\in \mc{W}(z_0)$, and note that $\mc{W}(z)\subset \mc{W}(z_0)$. Let $\Gamma_z$ be the leaf of $\hat{\mc{F}}$ containing $z$. If $x_0\in \hat{X}$ is one of the endpoints of $\Gamma_z$, then $\Gamma_z\cup \hat{f}(\Gamma_z)\cup \{x_0\}$ contains an arc $\gamma$ joining $z$ to $f(z)$, and so $f(z)\in \mc{W}(z)\subset \mc{W}(z_0)$, proving that $\hat{f}(\mc{W}(z_0))\subset \mc{W}(z_0)$. By the same argument, $\hat{f}^{-1}(\mc{W}(z_0))\subset \mc{W}(z_0)$, so the claim follows.
%\end{proof}

\begin{proof}[Proof of Proposition \ref{pro:engulfing-finite}]
Choose $x_0\notin \hat{X}$ and an $(\hat{f},\hat{\mc{F}})$-arc  $\gamma_0$ joining $x_0$ to $x_0+(1,0)$. Let $\gamma_1$ be another $(\hat{f},\hat{\mc{F}})$-arc joining $x_0$ to $x_0+(0,1)$. 
If $N\in \N$ is chosen large enough and  
$$C=\bigcup_{n=-N}^{N-1} ([{\gamma_0}]+(n,-N)) \cup ([\gamma_1]+(-N,n)) \cup ([\gamma_0]+(n,N)) \cup ([\gamma_1]+(N,n)),$$
the connected component $Q$ of $\R^2\sm C$ which contains $x_0$ is bounded and contains an arbitrarily large square centered at $x_0$ (see Figure \ref{fig:engulfing}). In particular, if $N$ is chosen large enough we have that $R\subset Q$.

\begin{figure}[ht]
\includegraphics[height=5cm]{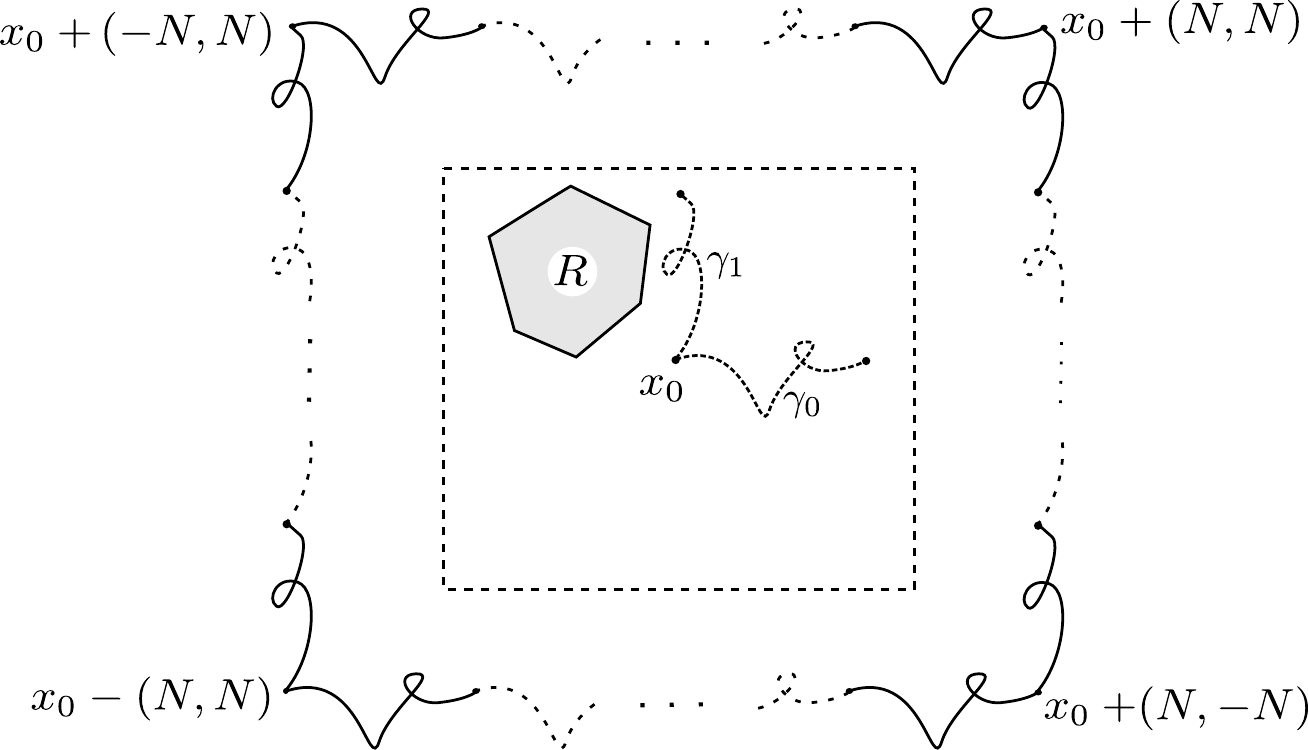}
\caption{Proof of Propoisition \ref{pro:engulfing-finite}.}
\label{fig:engulfing}
\end{figure}

Moreover, $\bd Q$ is an $(\hat{f},\hat{\mc{F}})$-arc, so there are finitely many leaves $\Gamma_1, \dots, \Gamma_m$ of $\hat{\mc{F}}$ such that $\bd Q$ is contained in the union of a finite set of iterates of these leaves together with some elements of $\hat{X}$. Recall that each $\Gamma_i$ joins two different points of $\hat{X}$. Let $P\subset \hat{X}$ be the (finite) set consisting of all endpoints of the arcs $\Gamma_i$, with $i\in \{1,\dots,m\}$.

Suppose that $U\subset \R^2$ is an open $\hat{f}$-invariant topological disk intersecting $R$. Since $R\subset Q$ and $Q$ is bounded, it follows that $U$ intersects both $Q$ and $\R^2\sm Q$, and therefore $U$ intersects $\bd Q$. Since $U$ is open and $\hat{X}$ is totally disconnected, this means that $U$ intersects some iterate of $\Gamma_i$ for some $i\in \{1,\dots, m\}$, and the fact that $U$ is invariant implies that $U$ intersects $\Gamma_i$. From Proposition \ref{pro:engulfing} we conclude that one of the endpoints of $\Gamma_i$ lies in $U$, hence $U\cap P\neq \emptyset$.
\end{proof}

\section{Proof of Theorem \ref{th:teoremao}}\label{sec:teoremao}

Throughout this section, we will assume that $f\colon \T^2\to \T^2$ is an irrotational homeomorphism preserving a Borel probability measure $\mu$ of full support, and $\hat f\colon \R^2\to \R^2$ its irrotational lift. 
Recall that, by Theorem \ref{th:irrotational-nw}, this implies that $\hat{f}$ is nonwandering. 

We will assume that none of cases (i), (ii) or (iii) from Theorem \ref{th:teoremao} holds, and we will seek a contradiction. Thus we assume from now on that $f$ is not annular, $\fix(f)$ is not fully essential and there exists some point $\hat{x}_0\in \R^2$ with an unbounded $\hat{f}$ orbit. We let $x_0=\pi(\hat{x}_0)$.
\setcounter{claim}{0}

Note that in the case that $\fix(f)$ is neither fully essential nor inessential, then Proposition \ref{pro:annular-annular} implies that $f$ is annular, contradicting our assumption. Hence $\fix(f)$ is in fact inessential.

Let us list the properties that we have so far thanks to our assumptions:
\begin{itemize} 
\item $f$ is non-annular;
\item $\fix(f)$ is inessential;
\item the orbit $O(\hat{x}_0):=  \{\hat f^{n}(\hat{x}_0): n\in \Z\}$ is unbounded; 
\item the nonwandering set of $\hat{f}$ is $\R^2$ (due to Theorem \ref{th:irrotational-nw});
\end{itemize}
As in \S\ref{sec:uepsilon-invariant}, we may use Proposition \ref{pro:collapse} to find a map which satisfies, in addition to all the previous facts,
\begin{itemize}
\item $\fix(f)$ is totally disconnected.
\end{itemize}

Since $O(\hat{x}_0)$ is unbounded, we have that $\bd_\infty O(\hat{x}_0)$ is nonempty. Choose any $\ol{w}\in \bd_\infty O(\hat{x}_0)$, which will remain fixed until the end of the proof. The following claim should be obvious:

\begin{claim}\label{claim:wbar} For any $w\in \R^2_*$ such that $w$ is not perpendicular to $\ol{w}$, the orbit of $\hat{x}_0$ is unbounded in the direction of $w$, \ie
$$\sup_{n\in \Z} \abs{\smash{p_w(\hat{f}^n(\hat{x}_0))}} = \infty.$$
\end{claim} 

\begin{claim} $f^n$ is non-annular for any $n\in \N$.
\end{claim}
\begin{proof} 
This follows from Proposition \ref{pro:annular-annular}, noting that $f$ has a fixed point (since it is irrotational).
\end{proof}

Fix any $v\in \Z^2_*$, and recall the definition of the sets $\omega_v$ and $\omega_{-v}$ from \S\ref{sec:omegas}. 
\begin{claim} $\omega_v$ and $\omega_{-v}$ are nonempty.
\end{claim}
\begin{proof} It follows directly from Proposition \ref{pr:omeganaovazio}.
\end{proof}

\begin{claim} $x_0\in \ol{\pi(\omega_v)}\cap \ol{\pi(\omega_{-v})}$
\end{claim}
\begin{proof}
We show that $x_0\in \ol{\pi(\omega_v)}$; the other part is analogous. Suppose for contradiction that $x_0\notin \ol{\pi(\omega_v)}$, and fix $\epsilon>0$ such that $B_\epsilon(x_0)$ is disjoint from $\pi(\omega_v)$. Since the latter set is invariant, it follows that $U'_\epsilon(x_0,f)$ is also disjoint from $\pi(\omega_v)$, where we use the notation from \S\ref{sec:uepsilon}.

We claim that $U'_\epsilon(x_0,f)$ is essential. Suppose on the contrary that it is inessential. Then $U=\fil(U'_{\epsilon}(x_0,f))$ is a topological disk which is either invariant or periodic and disjoint from its image. Choose $\hat{U}$ as the connected component of $\pi^{-1}(U)$ containing $\hat{x}_0$. Since $\hat{f}$ is nonwandering and the components of $\pi^{-1}(U)$ are permuted, it follows that $\hat{U}$ is either invariant or periodic and disjoint from its image. Since $\hat{x}_0\in \hat{U}$, it follows that $\hat{U}$ is unbounded. But in the case that $\hat{U}$ is invariant, this contradicts Theorem \ref{th:essine} (as $f$ is non-annular and $\fix(f)$ is not fully essential), and in the case that $\hat{U}$ is periodic and disjoint from its image it contradicts Corollary \ref{coro:uepsilon-invariant}.

Thus $U'_\epsilon(x_0,f)$ is essential, and since it is a periodic open set and $f^n$ is not annular for any $n$, Proposition \ref{pro:annular-annular} implies that $U'_\epsilon(x_0,f)$ is in fact fully essential.
But then Proposition \ref{pro:compact-ine} says that all the connected components of $\R^2\sm \pi^{-1}(U'_\epsilon(x_0,f))$ are bounded. Since $\omega_v$ is contained in the latter set, and all the connected components of $\omega_v$ are unbounded, we have a contradiction, proving the claim.
\end{proof}

\begin{claim}\label{claim:omegasetocam} $\pi(\omega_v)\cap \pi(\omega_{-v})\neq \emptyset$
\end{claim}
\begin{proof} Suppose the contrary. Since $\hat{x}_0$ is not fixed, there is $\epsilon>0$ such that $B_\epsilon(\hat{x}_0)$ is disjoint from $\hat{f}(B_\epsilon(\hat{x}_0))$. Since $x_0=\pi(\hat{x}_0)$ belongs to $\ol{\pi(\omega_v)}\cap \ol{\pi(\omega_{-v})}$, there are integers $p_1,\, p_2$ such that $\omega_v+p_1v$ and $\omega_{-v}+p_2v$ both intersect $B_\epsilon(\hat{x}_0)$. Let $\ell$ be a straight line segment (without its endpoints) contained in $B_\epsilon(\hat{x}_0)$ joining a point of $\omega_v+p_1v$ to a point of $\omega_{-v}+p_2v$. We may assume that $\ell$ is disjoint from $\omega_v+p_1v$ and $\omega_{-v}+p_2v$, by replacing it by an appropriate connected component of $\ell\sm (\omega_v+p_1v)\cup(\omega_{-v}+p_2v)$. Since $\hat{f}(\ell)$ is disjoint from $\ell$, we are under the hypotheses of Proposition \ref{pro:omeganaotoca} (note that $\hat{f}$ has a fixed point for being the irrotational lift of $f$).
But part (6) of said proposition implies that $\hat{f}$ has a wandering open set. This is a contradiction, since $\hat{f}$ is nonwandering under our current assumptions.
\end{proof}

\begin{claim}\label{claim:notinomega} If $v$ is not perpendicular to $\ol{w}$, then $x_0\notin \pi(\omega_v)\cap \pi(\omega_{-v})$
\end{claim}
\begin{proof} If $x_0\in \pi(\omega_v)\cap \pi(\omega_{-v})$, then there exist $v_1$ and $v_2\in \Z^2$ such that $\hat{x}_0 \in (\omega_v+v_1)\cap (\omega_{-v}+ v_2)$. By the definition of $\omega_v$ and $\omega_{-v}$ this implies that the orbit of $\hat{x}_0$ is bounded in the direction of $v$, \ie that $p_v(O(\hat{x}_0))$ is a bounded set. This contradicts Claim \ref{claim:wbar}.
\end{proof}

For each $n\in \N$, let $O_n = U_{1/n}(\hat{x}_0, \hat{f})$ (using the notation from \S\ref{sec:uepsilon}). Note that $O_{n+1}\subset O_n$ for all $n\in \N$.

%For each $n\in \N$, let $U_n = U_{1/n}(\hat{x}_0, \hat{f})$ (using the notation from \S\ref{sec:uepsilon}). Note that $U_{n+1}\subset U_n$ for all $n\in \N$.
\begin{claim}\label{claim:cortadir} If $w\in \Z^2_*$ is not parallel to $\ol{w}$, then there is $n$ such that $O_n\cap T_w(O_n)=\emptyset$.
\end{claim}
\begin{proof} Let $v=w^\perp$, so that $v$ is not perpendicular to $\ol{w}$. By Claim \ref{claim:omegasetocam} we know that $\pi(\omega_v)\cap \pi(\omega_{-v})\neq \emptyset$, so the hypotheses of Proposition \ref{pro:omegasetocam} hold. The previous claim implies that $x_0\notin \pi(\omega_v)\cap \pi(\omega_{-v})$, so if $\epsilon$ is the number from  part (3) of Proposition \ref{pro:omegasetocam}, our claim follows choosing $n>1/\epsilon$.
\end{proof}

\begin{claim}\label{claim:uepsilon-inv} For any $n\in \N$, the set $O_n$ is $\hat{f}$-invariant, unbounded, and $\ol{w}\in \bd_\infty O_n$.
\end{claim}
\begin{proof} 
Since $\hat{f}$ is nonwandering, the definition of $U_\epsilon$ (see \S\ref{sec:uepsilon}) implies that $O_n$ is $\hat{f}$-periodic, and if its least period is not $1$ then it is disjoint from its image. Moreover, the fact that the orbit of $\hat{x}_0$ is unbounded implies that $O_n$ is unbounded. 
Suppose for a contradiction that $O_n$ is not invariant. Then $O_n$ is a periodic unbounded open topological disk disjoint from its image, contradicting Corollary \ref{coro:uepsilon-invariant}.  Thus $O_n$ is invariant, and the remaining claim is obvious from our choice of $\ol{w}$.
\end{proof}

\begin{claim}\label{claim:Uess} $\pi(O_n)$ is essential for each $n\in \N$.
\end{claim}
\begin{proof} If $\pi(O_n)$ were inessential, it would be an $f$-invariant open topological disk which has an unbounded lift, contradicting Theorem \ref{th:essine}. %Thus $\pi(U_n)$ is essential, and the fact that $f$ is non-annular (by our assumptions) implies that $\pi(U_n)$ is fully essential (see Proposition \ref{pro:annular-annular}).
\end{proof}

Note that, if $O_n'=U_{1/n}'(\hat{x}_0,\hat{f})$, then by definition $O_n=\fil(O_n')$. Thus Claim \ref{claim:cortadir} implies that $(O_n')_{n\in \N}$ is an eventually $\Z^2\sm (\Z\ol{w})$-free chain of open connected sets, Claim \ref{claim:uepsilon-inv} implies that each $O_n'$ is invariant (note that, as remarked in \S\ref{sec:uepsilon}, $O_n'$ is invariant if and only if $O_n$ is invariant), and Claim \ref{claim:Uess} implies that $O_n'$ is essential for each $n\in \N$ (due to Proposition \ref{pro:fill-ine}). Thus all the hypotheses of Proposition \ref{pro:chain-annular} hold, so we conclude that there exist $M>0$ and $w\in \R^2_*$ such that  
$$\abs{\smash{p_w(\hat{f}^n(z)-z)}}\leq M \quad \text{for all } z\in \R^2,\, n\in \Z.$$
Moreover, since with our assumptions $f$ is not annular, $\R w$ is a line of irrational slope,  
$$E=\bigcap_{n\in \N} \ol{O}_n'\subset S:=p_w^{-1}((-M,M)),$$
and $E$ separates the two boundary components of the strip. This means that $p_{w^\perp}(E)=\R$.
Note that, since $O_n=\fil(O_n')$, which is the union of $O_n'$ with the bounded components of its complement, this easily implies a similar property for the chain $(O_n)_{n\in \N}$; namely,
$$K=\bigcap_{n\in \N} \ol{O}_n\subset S$$
and since $E\subset K$, also $K$ separates the two boundary components of the strip $S$ and $p_{w^\perp}(K)=\R$.

Since $\hat{f}$ is nonwandering, and by our assumption $\fix(f)$ is totally disconnected, if $\mc{F}$ is the foliation with singularities in a set $X\subset \fix(f)$ given by Proposition \ref{pro:kt-brouwer}, and $\hat{\mc{F}}$, $\hat{X}$ are the corresponding lifts, then the hypotheses of Proposition \ref{pro:gradient} hold, so the foliations are gradient-like. 

Let $S' = p_w^{-1}(-M-1,M+1)$ and $R=S'\cap p_{w^\perp}^{-1}(-M-1,M+1)$. If $v\in \Z^2_*$ is such that $\abs{p_w(v)}<1$, then $p_w(K+v)\subset S'$, and since $p_{w^\perp}(K+v) = p_{w^\perp}(K)=\R$, it follows that $K+v$ intersects $R$. 

Let $P\subset\hat{X}$ be the finite set given by Proposition \ref{pro:engulfing-finite}, so that any unbounded $\hat{f}$-invariant open topological disk $U$ intersecting $R$ necessarily contains an element of $P$. Let $m$ be the number of elements of $P$. If $w=(a,b)$, we know that $a/b$ is irrational, and so for any $\kappa>0$ there exist integers $c,d$, with $d\neq 0$ such that $|a/b-c/d|<\kappa/d$. Using this remark with $\kappa$ small enough we may find $v\in \Z^2_*$ such that $|p_w(v)|<1/(m+1)$.

Note that $\hat{x}_0\in K$, so the orbit of $\hat{x}_0$ is bounded in the $w$ direction. Since the orbit of $\hat{x}_0$ is unbounded in the direction of $\ol{w}$ (by our choice of $\ol{w}$ at the beginning of the proof), it follows that $w=\ol{w}^\perp$. In particular, $\ol{w}$ has irrational slope, so $v\in \Z^2_*$ is not parallel to $\ol{w}$.

From our previous observations, when $1\leq j\leq m+1$, the fact that $|p_w(jv)|<j/(m+1)<1$ implies that $p_w(K+jv)$ intersects $R$. Thus $(\ol{O}_n+jv)\cap R\neq \emptyset$ for all $n\in \N$. Since $v$ is not parallel to $\ol{w}$, if we fix $n$ large enough we may assume (by Claim \ref{claim:cortadir}) that $O_n$ is disjoint from $O_n+jv$ for any $j\in \{1,2,\dots,m+1\}$. This implies that the sets $\{O_n+jv:1\leq j\leq m+1\}$ are pairwise disjoint. On the other hand, since $\ol{O}_n+jv$ intersects the open set $R$, so does $O_n+jv$, and since $O_n+jv$ is unbounded and $\hat{f}$-invariant we conclude from Proposition \ref{pro:engulfing-finite} that $O_n+jv$ contains an element of $P$, for each $j\in \{1,2,\dots,m+1\}$. Since $P$ has $m$ elements, it follows that there exist two different elements of $\{O_n+jv:1\leq j\leq m+1\}$ containing the same point of $p$, contradicting their disjointness.
This contradiction shows that $v$ cannot have irrational slope, concluding the proof of Theorem \ref{th:teoremao}.

\section{Proof of Proposition \ref{pro:non-lc}}\label{sec:non-lc}
Suppose that $f$ and $\hat{f}$ satisfy the hypotheses of Theorem \ref{th:teoremao} and case (i) holds, so that $\fix(f)$ is fully essential. In particular, there is a connected component $K_0$ of $\fix(f)$ which is fully essential. Assume that $K_0$ is locally connected. To prove Proposition \ref{pro:non-lc}, we need to show that one of cases (ii) or (iii) holds. We will in fact show that, under these assumptions, case (iii) always holds, \ie $f$ is annular.

Since $K_0$ is compact and locally connected, $\hat{K}_0=\pi^{-1}(K_0)$ is closed and locally connected. %, and since $K_0$ is fully essential one has $\hat{K}_0 = \hat{K}_0+v$ for any $v\in \Z^2$. 
The relation defined on $\hat{K}_0$ by $z\sim z'$ if there is a compact connected subset of $\hat{K}_0$ containing both $z$ and $z'$ is an equivalence relation. Let $E(z)$ be the equivalence class of $z$, which coincides with the union of all compact connected subsets of $\hat{K}_0$ containing $z$. The local connectedness of $\hat{K}_0$ implies that each $E(z)$ is open in $\hat{K}_0$ (and since $\{E(z):z\in \hat{K}_0\}$ is a partition, $E(z)$ is both open and closed in $\hat{K}_0$). 

Note also that $\{\pi(E(z)): z\in \hat{K}_0\}$ is a partition of $K_0$, since $E(z+v) = E(z)+v$ for each $z\in \hat{K}_0$ and $v\in \Z^2$. Since $\pi$ is a local homeomorphism, the set $\pi(E(z))$ is open in $K_0$ for each $z\in \hat{K}_0$, and again since these sets partition $K_0$ it follows that $\pi(E(z))$ is both open and closed in $K_0$. Since $K_0$ is connected and $\pi(E(z))$ is nonempty, it follows that $\pi(E(z)) = K_0$ for each $z\in \hat{K}_0$. 

Fix $z\in \hat{K}_0$. We claim that there is $v\in \Z^2_*$ such that $E(z)$ intersects $E(z)+v$. Indeed, if this is not the case then $\pi|_{E(z)}$ is an injective map from $E(z)$ to $\pi(E(z))=K_0$. Since $\pi|_{\hat{K}_0}\colon \hat{K}_0\to K_0$ is a local homeomorphism and $E(z)$ is an open subset of $\hat{K}_0$, it follows that $\pi|_{E(z)}$ is an open map onto $K_0$. Being an open continuous injection, it follows that $\pi_{E(z)}$ is a homeomorphism. Thus $E(z)$ is homeomorphic to $K_0$, and in particular $E(z)$ is compact. But since $E(z)$ is compact and disjoint from $E(z)+v$ for all $v\in \Z^2_*$, one may find an open neighborhood $U$ of $E(z)$ such that $U$ is disjoint from $U+v$ for all $v\in \Z^2_*$. This means that $\pi(U)$ is an open inessential set, and therefore $\pi(E(z)) = K_0$ is inessential, a contradiction.

This shows that there exists $v\in \Z^2_*$ such that $E(z)$ intersects $E(z)+v$, and so there is a compact set $C\subset E(z)$ containing $z$ and $z+v$. Letting $\Theta = \bigcup_{n\in \Z} C+nv$ we obtain a closed connected set such that $\Theta = \Theta+v$, and the half-planes $\{z:p_{v^\perp}(z)>M\}$ and $\{z:p_{v^\perp}(z)<-M\}$ lie in different connected components of $\R^2\sm \Theta$ if $M$ is chosen large enough. Since $\pi(\Theta)\subset K_0\subset \fix(f)$ and $\hat{f}$ is irrotational, it follows easily that $\Theta\subset \fix(\hat{f})$. In particular, if $V$ is the connected component of $\R^2\sm \Theta$ containing $\{z:p_{v^\perp}(z)<-M\}$, then $V$ is invariant and 
$$\{z:p_{v^\perp}(z)<-M\} \subset V\subset \{z:p_{v^\perp}(z)\leq M\}$$ 
This easily implies that $f$ is annular (see for instance Proposition 2.5 of \cite{KT2012}). Thus, case (iii) holds, as we wanted to show.\qed

\subsection*{Acknowledgements}
The authors would like to thank the anonymous referee for the very useful suggestions and corrections that helped improve this article.

\bibliographystyle{koro} 
\bibliography{irrotational}

\end{document}